\documentclass{article}

\usepackage[frenchb,english]{babel}

\usepackage{amsmath}
\usepackage{amssymb}
\usepackage{amsthm}
\usepackage{stmaryrd}
\usepackage{graphicx}
\usepackage{comment}
\usepackage{dsfont}
\usepackage{xcolor}
\usepackage{mathtools}
\usepackage{bigints}
\usepackage{yhmath}
\usepackage{enumitem} 
\usepackage{subcaption} 
\usepackage{appendix} 
\usepackage{geometry} 
\usepackage{authblk} 

\newtheorem{lemma}{Lemma}
\newtheorem{propo}{Proposition}
\newtheorem{theor}{Theorem}

\newtheorem{defin}{Definition}
\newtheorem{remar}{Remark}

\newcommand{\vertii}[1]{{\left\vert\kern-0.3ex\left\vert #1 
    \right\vert\kern-0.3ex\right\vert}}

\newcommand{\vertiii}[1]{{\left\vert\kern-0.3ex\left\vert\kern-0.3ex\left\vert #1 
    \right\vert\kern-0.3ex\right\vert\kern-0.3ex\right\vert}}
    
\newcommand*\dd{\mathop{}\!\mathrm{d}}
    
\DeclareMathOperator*{\supess}{sup\,ess}

\title{About Lanford's theorem in the half-space with specular reflection}
\author{Th\'eophile Dolmaire\thanks{Department of Mathematics and Computer Science, University of Basel, Spiegelgasse 1, 4051 Basel, Switzerland, e-mail: \texttt{theophile.dolmaire@unibas.ch}}}

\begin{document}

\maketitle

\begin{abstract}
The present article proposes a rigorous derivation of the Boltzmann equation in the half-space. We show an analog of the Lanford's theorem in this domain, with specular reflection boundary condition, stating the convergence in the low density limit of the first marginal of the density function of a system of $N$ hard spheres towards the solution of the Boltzmann equation associated to the initial data corresponding to the initial state of the one-particle-density function.\\
The original contributions of this work consist in two main points: the rigorous definition of the collision operator and of the functional space in which the BBGKY hierarchy is solved in a strong sense; and the adaptation to the case of the half-space of the control of the recollisions performed by Gallagher, Saint-Raymond and Texier, which is a crucial step to obtain the Lanford's theorem.
\end{abstract}

\tableofcontents

%

\section{Introduction}

In 1872, starting from an atomistic description of matter, Boltzmann obtained an evolution equation for the density of the particles describing a rarefied gas, marking a milestone in kinetic theory (see \cite{Bolt}). In particular, he managed to grasp the irreversible trend of fluids of tending to equilibirum states. Although this is one of the strengths of his model, it was also at the origin of important debates: how could a reversible, microscopic description of a fluid generate irreversible behaviours on a macroscopic scale? Is this model trustable and rigorously obtained?\footnote{See for instance the criticisms from Loschmidt (\cite{Losc}) and Zermelo (\cite{Zerm}).}\\
The first rigorous derivation of the Boltzmann equation was obtained in 1973 by Lanford (see his pioneering work \cite{Lanf}), for a non trivial (albeit small) time interval, and was completed over time by  Cercignani, Illner and Pulvirenti (\cite{CeIP}), Gerasimenko and Petrina (\cite{CGPY}) and more recently by Pulvirenti, Saffirio and Simonella for short-range potentials (\cite{PuSS}), and by Gallagher, Saint-Raymond and Texier for the hard sphere model in \cite{GSRT}. This derivation may be presented with the following statement, which is not completely formalized at this step for the sake of simplicity.
\begin{theor}[Lanford's theorem]
\label{INTROTheorLanfo}
We consider a system of $N$ particles of radius $\varepsilon$ interacting via the hard sphere model, or via a radial, singular at $0$ and repulsive potential $\Phi_\varepsilon$ supported in $B(0,\varepsilon)$ which enables the parametrization of the scattering of particles by their deflection angle.\\
Let $f_0:\mathbb{R}^{2d} \rightarrow \mathbb{R}_+$ be a continuous density of probability such that, for some $\beta > 0$, $\mu \in \mathbb{R}$, we have:
$$
\vertii{f_0 \exp\big(\frac{\beta}{2} \vert v \vert^2\big)}_{L^\infty(\mathbb{R}^d_x \times \mathbb{R}^d_v)} \leq \exp(-\mu).
$$
Let us assume that the $N$ particles are initially identically distributed according to $f_0$, and ``independent". Then, there exists $T>0$ (depending only on $\beta$ and $\mu$) such that, in the Boltzmann-Grad limit $N \rightarrow +\infty$, $N\varepsilon^{d-1} = 1$, the distribution function of the particles converges to the solution of the Boltzmann equation:
$$
\partial_t f + v\cdot\nabla_x f = \int_{\mathbb{S}^{d-1} \times \mathbb{R}^d} \big[f(v')f(v'_*) - f(v)f(v_*)\big] b(v-v_*,\omega) \dd v_* \dd \omega
$$
with initial data $f_0$, and with the cross-section $b(w,\omega) = [\omega\cdot w]_+$ for the hard sphere interactions, or with a bounded cross-section depending implicitly on the potential $\Phi_\varepsilon$ in the case of the interactions through a repulsive potential.
\end{theor}

So far, the previous theorem was stated only for domains without any boundary (namely, $\mathbb{R}^d$ or $\mathbb{T}^d$). In this work, we address the question of the rigorous derivation of the Boltzmann equation in the half-space, prescribing the specular reflection as boundary condition. We will start from the hard shere model, to obtain the analog of Theorem \ref{INTROTheorLanfo} in our setting, which is Theorem \ref{SECT3TheorLanford___Vers_Quali}, stated page \pageref{SECT3TheorLanford___Vers_Quali}. To the best of our knowledge, it is the first rigorous derivation of the Boltzmann equation in a domain with a boundary.

\section{From the dynamics of the particles to the relevant hierarchies}

\subsection{The hard sphere dynamics}

Let $d$ be an integer strictly larger than $1$, $N$ be a positive integer, and $\varepsilon$ be a positive real number. One considers a system of $N$ spherical particles of mass $1$ and of radius $\varepsilon/2$, evolving inside a domain of the Euclidean space $\mathbb{R}^d$, namely the half-space $\{ x \in \mathbb{R}^d \ /\ x\cdot e_1 > 0\}$, where $e_1$ denotes the first vector of the canonical basis. The complement of the domain will be called an obstacle, denotes as $\Omega$. The boundary $\{ x\in\mathbb{R}^d \ /\ x\cdot e_1 = 0 \}$ of the domain in which the particles evolve (which is of course also the boundary of the obstacle) will be called \emph{the wall}.

\paragraph*{The notations for the state of the system of particles.}

The position, respectively the velocity, at time $t$ of the $i$-th particle (for $1\leq i \leq N$) will be denoted $x_i(t)$, respectively $v_i(t)$. One will assume that the particles cannot overlap, neither cannot cross the wall $\{ x\in\mathbb{R}^d \ /\ x\cdot e_1 = 0 \}$, so that, if for all time $t$ one collects all the positions and velocities of the particles of the system to create the vector $Z_N$ defined as
$$
Z_N(t) = \big( x_1(t),\dots,x_N(t),v_1(t),\dots,v_N(t) \big)
$$
(such a vector is called a \emph{configuration} of the system), then this vector $Z_N(t)$ has to lie in
\begin{align}
\label{SECT1DefinEspaPhase}
\big\{ (x_1,\dots,x_N,v_1,\dots,v_N) \in \mathbb{R}^{2dN} \, /\ \forall\, 1 \leq i \leq N,\ x_i\cdot e_1 \geq \varepsilon/2,\text{ and } \forall\, 1 \leq j < k \leq N,\ \vert x_j - x_k \vert \geq \varepsilon \big\}.
\end{align}
This part, defined by the expression \eqref{SECT1DefinEspaPhase}, will be called \emph{the phase space for $N$ hard spheres of radius $\varepsilon/2$}, and it will be denoted $\mathcal{D}^\varepsilon_N$.\\
Sometimes, it will be useful to designate the position or the velocity of a certain particle $i$ of a configuration $Z_N=(x_1,v_1,\dots,x_N,v_N)$. One will then use the notations $Z_N^{X,i} = x_i$ for the position of the particle $i$, and $
Z_N^{V,i} = v_i$ for its velocity. The collection of all the positions $(x_1,\dots,x_N)$ of the configuration $Z_N$ will be denoted $Z_N^X = (x_1,\dots,x_N)$, and the collection of all the velocities $(v_1,\dots,v_N)$ of $Z_N$ will be denoted $Z_N^V = (v_1,\dots,v_N)$.

\paragraph*{The dynamics inside the phase space.}
\label{SSSecDynamDans_EspacPhase}

Inside this phase space $\mathcal{D}^\varepsilon_N$, one will prescribe the most simple dynamics: the particles will travel in straight lines, conserving a constant velocity. In other words, far from the boundary they are subject to the free flow.

\paragraph*{Interaction with the boundary of the domain: the specular reflection.}
\label{SSSecInteractioAvecLeBord}

When a particle reaches the boundary, that is when there exists a time $t_b$ and an integer $1\leq i\leq N$ such that $x_i(t_b)\cdot e_1 = \varepsilon/2$, the velocity of the particle has to be changed in order to keep this particle inside the domain $\{ x\cdot e_1 \geq \varepsilon/2 \}$ for times larger than $t_b$. The law of reflection chosen here will be the specular reflection, also known as the Snell-Descartes law, which takes here a very simple form, namely:
\begin{equation}
v_i(t_b^+) = v_i(t_b^-) - 2 \big( v_i(t_b^-)\cdot e_1 \big) e_1.
\end{equation}

\begin{figure}[!h]
\centering
\includegraphics[scale=0.2, trim = 0cm 5cm 8cm 5cm]{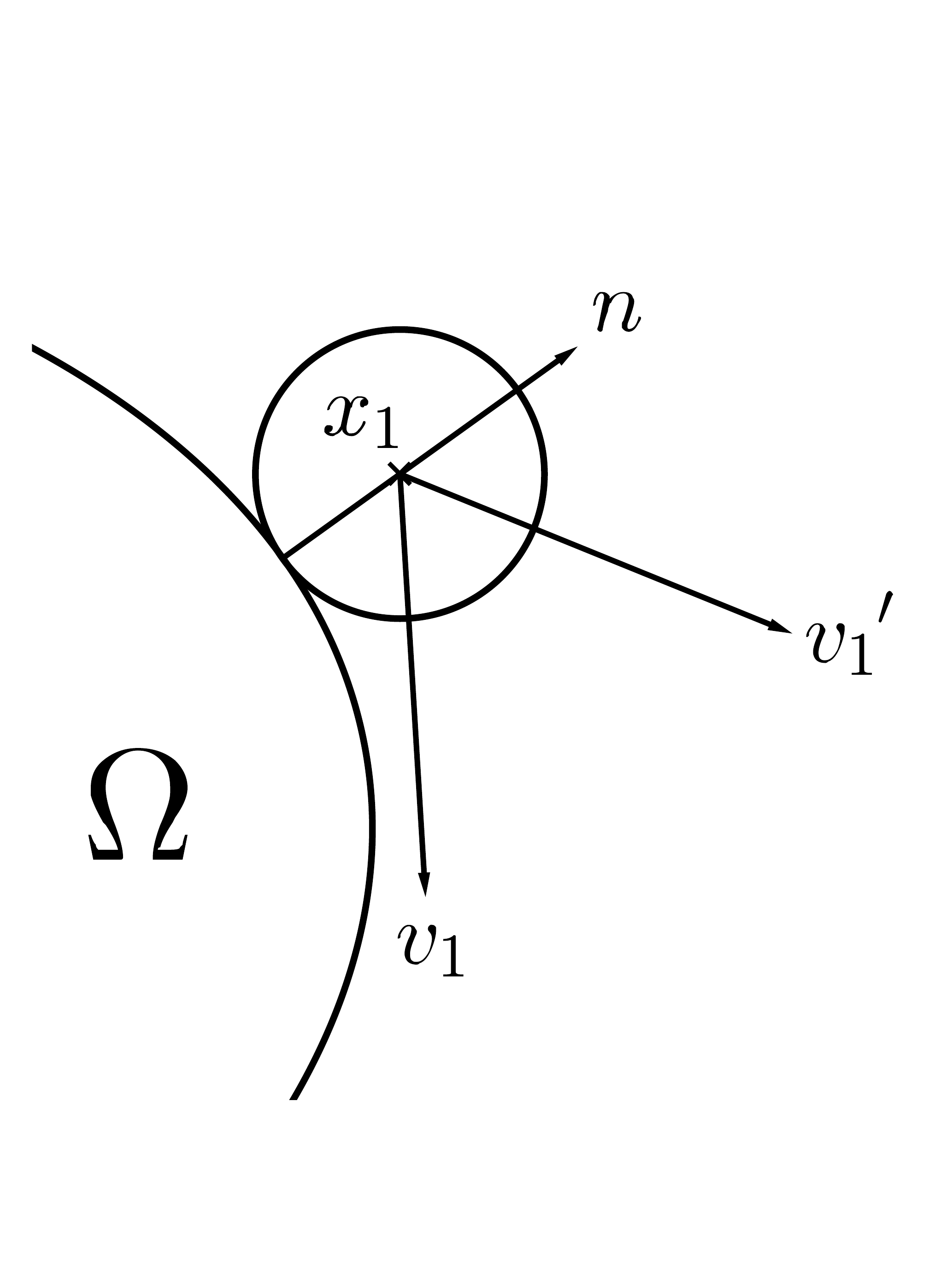}
\caption{A representation of the specular reflection against an obstacle $\Omega$.}
\label{SECT1FigurReflexion_Speculaire}
\end{figure}

\paragraph*{Interaction between the particles: the hard sphere model.}
\label{SSSecInteractioEntreParti}

The particles are assumed to be spherical, with a non zero diameter $\varepsilon$. In addition to the dynamics prescribed by the free flow and the bouncings against the wall, one has to impose another change of velocity when two particles are about to overlap. For two particles at $x_1$ and $x_2$ that are about to overlap (that is such that $x_2=x_1+\varepsilon\omega$ with $\omega \in \mathbb{S}^{d-1}$) with respective pre-collisional velocities $v$ and $v_*$ (that is such that $(x_2-x_1)\cdot(v_*-v) = \omega\cdot(v_*-v) < 0$), one will replace their velocities by post-collisional ones, namely
\begin{equation}
\left\{
\begin{array}{rl}
v' &= v - (v-v_*)\cdot \omega \omega,\\
v_*'&= v_* + (v-v_*)\cdot \omega \omega
\end{array}
\right.
\end{equation}
(that is one has $(x_2-x_1)\cdot(v_*'-v') = \omega\cdot(v_*'-v') > 0$).

\begin{figure}[!h]
\centering
\includegraphics[scale=0.18, trim = 0cm 4cm 8cm 4cm]{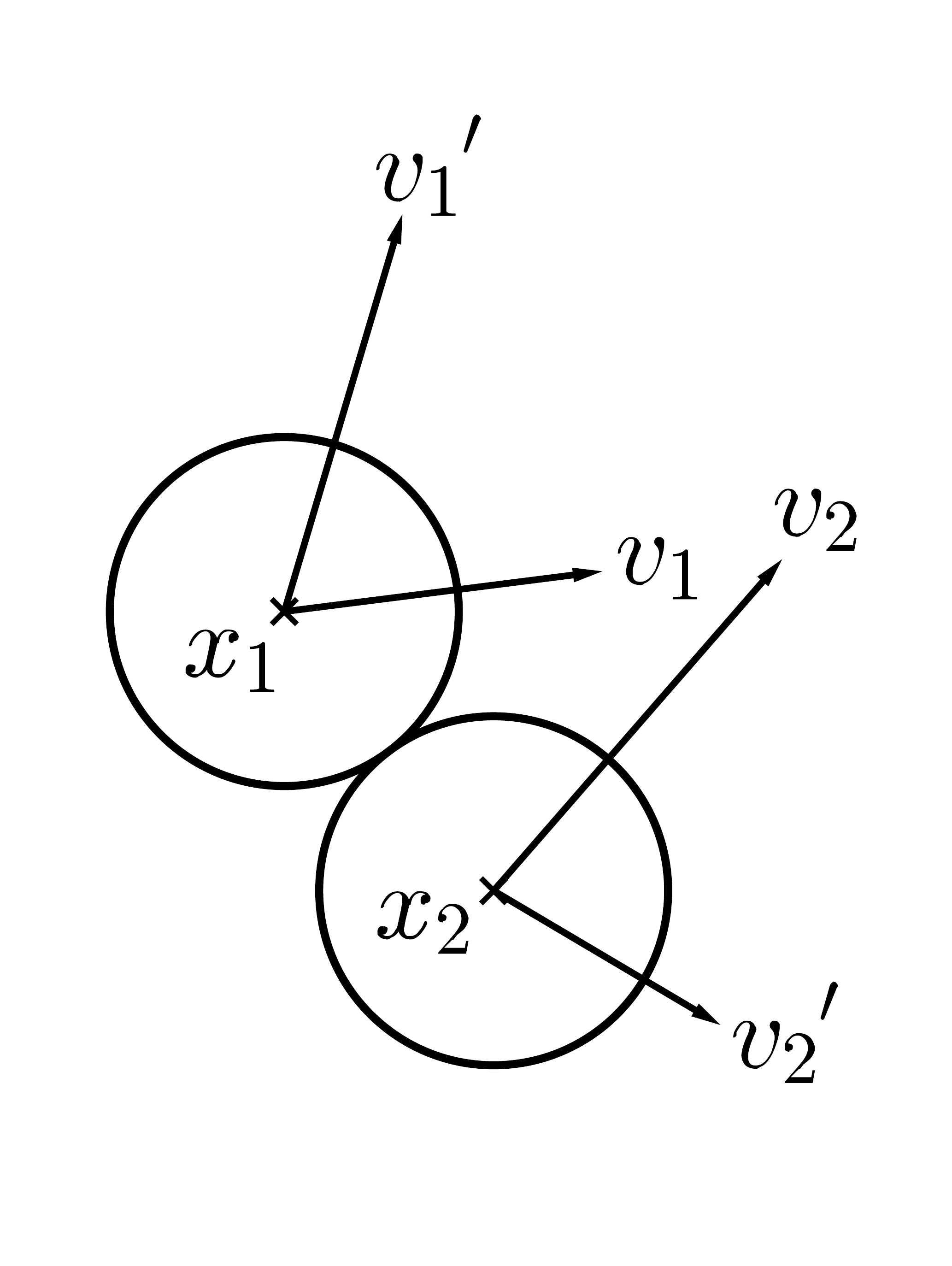}
\caption{A representation of an elastic collision between two hard spheres.}
\label{SECT1FigurReflexion_Speculaire}
\end{figure}

One notices that the kinetic energy and the momentum of the pair of particles is conserved during a collision, that is $\vert v' \vert^2/2 + \vert v_*' \vert^2/2 = \vert v \vert^2/2 + \vert v_* \vert^2/2$
and $v'+v_*' = v+v_*$.
For $\omega \in \mathbb{S}^{d-1}$ fixed, the mapping that associates to the pair of velocities $(v,v_*)$ the new pair
$$
(v',v_*') = \big(v - (v-v_*)\cdot \omega \omega,v_* + (v-v_*)\cdot \omega \omega\big),
$$
which is called the \emph{scattering mapping}, is an involution and sends the pre-collisional velocities $(v,v_*)$ (such that $\omega\cdot(v_*-v) < 0$) onto the post-collisional velocities $(v',v'_*)$ (such that $\omega\cdot(v'_*-v') > 0$), and vice versa.\\
\newline
For more details about the hard sphere dynamics (and the other models of interaction that are usually introduced), the reader may refer to \cite{Cer1} and especially \cite{CeIP}\footnote{See in particular Appendix 4.A ``More About Hard-Sphere Dynamics".}. Let us focus here on an important problem concerning this transport: for a given system of particles, let us call an \emph{event} a time such that a particle bounces against the obstacle or collide with another particle. Depending on the initial configuration, two events involving the same particle can occure at the same time, leading to an ill-defined dynamics. Here are the possible cases.\\
We first deal with the case of a bouncing against the obstacle: thanks to the convexity of the obstacle (in fact any obstacle with a bounded curvature provides the same property up to assume that the size of the particles is small enough), none of the initial configurations could lead to a situation in which a particle touches at the same time the obstacle at two or more different points. However, some initial configurations could lead to a situation in which a particle bounces against the obstacle, and collides with at least another particle at the same time.\\
One should also consider an initial configuration leading to a situation in which a particle collides with two other particles (or more) at the same time.  There is no other kind of simultaneous events involving the same particle when the obstacle is the half-space, but in those cases, two velocities or more are assigned to the same particle at the time of the considered events, we call \emph{pathological} a trajectory for which the dynamics becomes ill-defined due to this phenomenon at some point. Let us then study the initial configurations leading to a pathological trajectory.

\begin{propo}[Rigorous definition of the hard sphere dynamics almost everywhere, globally in time]
\label{SECT1PropoBonneDefinSpherDures}
Let $N$ be an integer larger than $2$ and $0 < \varepsilon \leq 1$ be a positive number. Then the two following assertions hold.
\begin{itemize}
\item The set of initial configurations $Z_N$ in the phase space $\mathcal{D}^\varepsilon_N$ for $N$ hard spheres of radius $\varepsilon/2$ leading to a pathological trajectory during the time interval $\mathbb{R}_+$ is of measure zero.
\item For every initial configuration $Z_N$ in the phase space $\mathcal{D}^\varepsilon_N$, one considers the subset $\mathcal{E}(Z_N)$ of $\mathbb{R}_+$ composed of all the times of the events of this dynamics in the largest time interval where it is well-defined. Then for any initial configuration $Z_N$ outside a subset of the phase space $\mathcal{D}^\varepsilon_N$ of measure zero, $\mathcal{E}(Z_N)$ is a discrete set.
\end{itemize}
\end{propo}

The proof of this result, originally published in \cite{Al75}\footnote{This reference deals with the most general setting possible, taking into account a wide variety of obstacles.}, is presented in a modern way in \cite{GSRT}\footnote{See Proposition 4.1.1 page 28.} for the case without obstacle. The case of the half-space is presented in detail in \cite{PhDTT} but for the sake of completeness, the reader may find a shortened proof in appendix, page \pageref{AppenSectiBonneDefinSpherDures}.

\begin{remar}
This result shows that the hard sphere dynamics is globally defined on time, for almost every initial configurations of particles, and the accumulation of events cannot happen except for a subset of initial data of zero measure.
\end{remar}

\begin{defin}[Hard sphere transport]
\label{SECT1DefinTransport_SpherDures}
For any positive integer $N$ and any positive real number $\varepsilon$, we define the hard sphere transport of $N$ particles of radius $\varepsilon/2$ as the map:
$$
(t,Z_N) \mapsto T^{N,\varepsilon}_t(Z_N),
$$
defined for almost every $Z_N \in \mathcal{D}^\varepsilon_N$ (according to Proposition \ref{SECT1PropoBonneDefinSpherDures}) and any time $t\in\mathbb{R}$, where $T^{N,\varepsilon}_t(Z_N)$ is the configuration starting from $Z_N$ and obtained after following the hard sphere dynamics for a time $t$.
\end{defin}

\subsection{The statistical study of the system of $N$ hard spheres: the BBGKY hierarchy}

In statistical physics the central object turns out to be the density function of the system of particles, which represents the probability, along time, of finding the system of $N$ particles in a given state. The relevant information will be obtained as moments of this function (which corresponds to a "mean information").\\
Another very important object in the following, derived from the density function, will be the family its marginals, that is the integral with respect to some of the variables of the density. In particular, the first marginal, which is obtained as the integral with respect to all of the variables except the position and the velocity of the first particle of the system, represents the mean behaviour of a single particle of the system.\\
One will recall in this section the key observation due to Bogolyubov, Born, Green, Kirkwood and Yvon (see respectively \cite{Bogo}, \cite{BoGr}, \cite{Kirk} and \cite{Yvon}), proving that it is possible to link together those marginals. This link, being known as the \emph{BBGKY hierarchy}, provides a crucial family of equations, deeply linked with the Boltzmann equation (see \cite{Grad}, \cite{Cer1}, \cite{CeIP} and \cite{GSRT}), and which will be the central object of study of the present work.

\paragraph*{The distribution function of a system of $N$ hard spheres.}

We denote
$$
f_N(t,Z_N) = f_N(t,x_1,v_1,\dots,x_N,v_N),
$$
the density function of the system of $N$ hard spheres. In other words, at time $t$, and for $A\subset \mathcal{D}^\varepsilon_N\subset\mathbb{R}^d_{x_1}\times\mathbb{R}^d_{v_1}\times\dots\times\mathbb{R}^d_{x_N}\times\mathbb{R}^d_{v_N}$ measurable, the quantity
$$
\int_A f_N(t,x_1,v_1,\dots,x_N,v_N) \dd x_1 \dd v_1 \dots \dd x_N \dd v_N
$$
represents the probability of finding the system in a configuration belonging to the subpart $A$ of the phase space.\\
We also introduce the following boundary conditions, which incode the hard spheres dynamics, according to the introductive Sections \ref{SSSecDynamDans_EspacPhase} and \ref{SSSecInteractioAvecLeBord}:
\begin{defin}[Boundary condition for the hard sphere dynamics]
\label{SECT1DefinInvolution_CondiBord_}
Let $s$ be a positive integer and $\varepsilon$ be a positive number. One defines the \emph{boundary condition for the hard sphere dynamics of $s$ particles of radius $\varepsilon/2$} as the map defined on the boundary $\mathcal{D}^\varepsilon_s$ of the phase space into itself as
$$
\chi^\varepsilon_s : \left\{
\begin{array}{rl}
\partial \mathcal{D}^\varepsilon_s &\rightarrow \ \ \partial \mathcal{D}^\varepsilon_s,\\
Z_s &\mapsto \ \ \chi^\varepsilon_s(Z_s),
\end{array}
\right.
$$
with $\big( \chi^\varepsilon_s(Z_s) \big)^X = Z_s^X$ (the map does not act on the positions of the configurations) and such that:
\begin{itemize}[leftmargin=*]
\item if for some $1 \leq i < j \leq s$, one has $\vert x_i - x_j \vert = \varepsilon$ and $\vert x_k - x_l \vert > \varepsilon$ for all $(k,l) \neq (i,j)$ (a single collision happens, between the two particles $i$ and $j$), while in addition $x_k\cdot e_1 > \varepsilon/2$ for all $1 \leq k \leq s$, one defines:

$$
\left\{
\begin{array}{rl}
\big(\chi^\varepsilon_s(Z_s)\big)^{V,i} &= v_i - \big(1/\varepsilon^2\big) \big((v_i-v_j) \cdot (x_i-x_j) \big) (x_i-x_j),\\
\big(\chi^\varepsilon_s(Z_s)\big)^{V,j} &= v_j + \big(1/\varepsilon^2\big) \big((v_i-v_j) \cdot (x_i-x_j) \big) (x_i-x_j),\\
\big(\chi^\varepsilon_s(Z_s)\big)^{V,k} &= v_k \text{  for all}\, 1 \leq k \leq s,\ k \neq i,\ k \neq j.
\end{array}
\right.
$$

\item if for some $1 \leq i \leq s$, one has $x_i \cdot e_1 = \varepsilon/2$ and $x_j \cdot e_1 > \varepsilon/2$ for all $j \neq i$ (a single particle bounces against the wall), while in addition $\vert x_k - x_l \vert > \varepsilon$ for all $1 \leq k < l \leq s$, one defines:

$$
\left\{
\begin{array}{rl}
\big(\chi^\varepsilon_s(Z_s)\big)^{V,i} &= v_i - 2 v_i\cdot e_1 e_1, \\
\big(\chi^\varepsilon_s(Z_s)\big)^{V,j} &= v_j \text{  for all  } 1 \leq j \leq s,\ j \neq i.
\end{array}
\right.
$$

\end{itemize}
\end{defin}

The distribution function $f_N$ is a solution of the \emph{Liouville equation}: $\forall\, t \geq 0,\ \forall\, Z_N \in \mathcal{D}^\varepsilon_N$,
\begin{align}
\label{SECT1EquatLiouville_InterieurD}
\partial_t f_N(t,Z_N) + \sum_{i=1}^N v_i\cdot\nabla_{x_i} f_N(t,Z_N) = 0,
\end{align}
with boundary conditions: $\forall\, t \geq 0,\ \forall\, Z_N \in \widehat{\partial \mathcal{D}^\varepsilon_N}^\text{in} = \big( \bigcup_{1 \leq i < j \leq N} B_{i,j}^\text{in} \big) \cup \big( \bigcup_{1 \leq i \leq N} C_i^\text{in} \big)$,
\begin{align}
\label{SECT1EquatLiouville_ConditBord}
f_N(t,Z_N) = f_N\big(t,\chi^\varepsilon_N(Z_N)\big).
\end{align}
For an interesting discussion about the link between the conservation of some physical quantities and the distribution function's boundary conditions, see \cite{Schn}.\\
\newline
The map $\chi^\varepsilon_s$ of the boundary conditions is only well defined on the subset $\widehat{\partial \mathcal{D}^\varepsilon_s} = \big( \bigcup_{1\leq i < j \leq s} B_{i,j} \cup \bigcup_{1 \leq i \leq s} C_i \big)$, with
\begin{align*}
B_{i,j} = \Big\{ \vert x_i - x_j \vert = \varepsilon,\ \vert x_k - x_l \vert > \varepsilon \text{  for all  } (k,l) \neq (i,j), x_k\cdot e_1 > \varepsilon/2 \text{  for all  } k \Big\},
\end{align*}
$$
C_i = \Big\{ x_i \cdot e_1 = \varepsilon/2,\ x_j \cdot e_1 > \varepsilon/2 \text{  for all  } j \neq i \text{  and  } \vert x_k - x_l \vert > \varepsilon \text{  for all  } (k,l) \Big\}.
$$
Nevertheless, and although $\widehat{\partial \mathcal{D}^\varepsilon_s}$ is a strict subset of the boundary $\partial \mathcal{D}^\varepsilon_s$ of the phase space ($\chi^\varepsilon_s$ is not defined for configurations of the boundary such that at least two collisions, two bouncings or a collision and a bouncing happen at the same time), its measure is full, which will be enough for the use we will do of it.\\
One sees that this map $\chi^\varepsilon_s$ is an involution, and if we denote
\begin{align*}
B_{i,j}^\text{in} = \Big\{ \vert x_i - x_j \vert = \varepsilon,\, \vert x_k - x_l \vert > \varepsilon \text{  for all  } (k,l) \neq (i,j),\, x_k \cdot e_1 > \varepsilon/2 \text{  for all  } k \text{  and  } (x_i-x_j)\cdot(v_i-v_j) > 0 \Big\}
\end{align*}
(corresponding to a configuration in which the particles $i$ and $j$ collide and such that for any positive time small enough, the distance between the particles $i$ and $j$, after being transported by the hard sphere dynamics, will be larger than $\varepsilon$: this is an incoming configuration into the phase space), and
\begin{align*}
C_i^\text{in} = \Big\{ x_i \cdot e_1 = \varepsilon/2,\, x_j \cdot e_1 > \varepsilon/2 \text{  for all  } j \neq i,\, \vert x_k - x_l \vert > \varepsilon \text{  for all  } (k,l)
\text{  and  } v_i \cdot e_1 > 0 \Big\}
\end{align*}
(corresponding to a configuration in which the particle $i$ is bouncing against the wall, and such that for any positive time small enough, the distance between the wall and the particle $i$, after being transported by the hard sphere dynamics, will be larger than $\varepsilon/2$: this is also an incoming configuration), then $\chi^\varepsilon_s$ sends the incoming configurations of $B_{i,j}^\text{in}$ and $C_i^\text{in}$ onto the outgoing configurations of $B_{i,j}^\text{out}$ and $C_i^\text{out}$ respectively, and conversely, where
\begin{align*}
B_{i,j}^\text{out} = \Big\{ \vert x_i - x_j \vert = \varepsilon,\, \vert x_k - x_l \vert > \varepsilon \text{  for all  } (k,l) \neq (i,j),\, x_k \cdot e_1 > \varepsilon/2 \text{  for all  } k \text{  and  } (x_i-x_j)\cdot(v_i-v_j) < 0 \Big\}
\end{align*}
and
\begin{align*}
C_i^\text{out} = \Big\{ x_i \cdot e_1 = \varepsilon/2,\, x_j \cdot e_1 > \varepsilon/2 \text{  for all  } j \neq i,\, \vert x_k - x_l \vert > \varepsilon \text{  for all  } (k,l)
\text{  and  } v_i \cdot e_1 < 0 \Big\}.
\end{align*}

\paragraph*{From the Liouville equation to the BBGKY hierarchy.}

Following the computation (which is now a classic) that can be found originally in \cite{Grad}, or in \cite{CeIP} and \cite{GSRT} for a more modern presentation, we can derive an equation verified by the marginals of the distribution function of the hard sphere system. Concerning the particularities appearing due to the presence of the wall, the reader may refer to \cite{PhDTT}.\\
We can show that the marginals $f_N^{(s)}$ of the distribution function solve the following equation on $\mathbb{R}_+\times \mathcal{D}^\varepsilon_s$:
\begin{equation}
\label{SECT1EquatHieraBBGKYFormeDiffe}
\partial_t f^{(s)}_N + \sum_{i=1}^s v_i \cdot \nabla_{x_i} f^{(s)}_N = \mathcal{C}^{N,\varepsilon}_{s,s+1} f^{(s+1)}_N,
\end{equation}
where $\mathcal{C}^{N,\varepsilon}_{s,s+1}$, called the \emph{s-th collision operator}, denotes
\begin{align}
\label{SECT1OperaCollision_HieraBBGKY}
\mathcal{C}^{N,\varepsilon}_{s,s+1} f_N^{(s+1)} (t,Z_s) = \sum_{i=1}^s (N-s)\, \varepsilon^{d-1} \int_{\mathbb{S}^{d-1}\times\mathbb{R}^d} \hspace{-10mm} \omega \cdot (v_{s+1}-v_i) f_N^{(s+1)}(t,Z_s,x_i+\varepsilon\omega,v_{s+1})\dd \omega \dd v_{s+1}.
\end{align}
This generic equation (for $1 \leq s \leq N-1$) constitutes the so-called \emph{BBGKY hierarchy}. Nevertheless, we will not use this version of the BBGKY hierarchy, that has to be considered with the analog of the boundary conditions \eqref{SECT1EquatLiouville_ConditBord}, namely:
\begin{equation}
\label{SECT1CondiBord_HierarchieBBGKY}
\forall \, 1 \leq s \leq N,\forall \, t \geq 0,\ \forall\, Z_s \in \widehat{\partial\mathcal{D}^\varepsilon_s}^\text{in},\ f_N^{(s)}(t,Z_s) = f_N^{(s)}(t,\chi^\varepsilon_s(Z_s)).
\end{equation}
We will use instead an integrated with respect to time version, which is on the one hand more self-contained (since it contains the boundary conditions), and which will be also more convenient to deal with the fixed point argument. This version, equivalent up to assume enough regularity of the solutions, writes
\begin{equation}
\label{SECT1HieraBBGKYVersiInteg/Tmps}
f^{(s)}_N(t,Z_s) ) = \big(\mathcal{T}^{s,\varepsilon}_t f^{(s)}_N(0,\cdot)\big)(Z_s) + \int_0^t \mathcal{T}^{s,\varepsilon}_{t-u} \mathcal{C}^{N,\varepsilon}_{s,s+1} f_N^{(s+1)}(u,Z_s) \dd u,
\end{equation}
where $\mathcal{T}^{s,\varepsilon}_t$ denotes the \emph{backwards} hard sphere flow, defined using the transport introduced in Definition \ref{SECT1DefinTransport_SpherDures}: $\big(\mathcal{T}^{s,\varepsilon}_t f\big)(Z_s) = f\big(T^{s,\varepsilon}_{-t}(Z_s)\big)$.

\paragraph*{The formal limit of the BBGKY hierarchy when $\varepsilon \rightarrow 0$.}

One presents now briefly the bridge built by Grad in \cite{Grad} between the BBGKY hierarchy and the Boltzmann equation, which is now a famous step in the derivation. One can refer to \cite{Cer1}, \cite{CeIP} or again \cite{GSRT} for more details.\\
The first step consists in a change of variable in the pre-collisional configurations in the collision term of the BBGKY hierarchy, that is one will rewrite the integral in order to integrate only over pre-collisional configurations. One writes:
\begin{align*}
\mathcal{C}^{N,\varepsilon}_{s,s+1} f_N^{(s+1)} (t,Z_s) &= \sum_{i=1}^s (N-s)\, \varepsilon^{d-1} \int_{\mathbb{S}^{d-1}\times\mathbb{R}^d} \hspace{-10mm} \omega \cdot (v_{s+1}-v_i) f_N^{(s+1)}(t,Z_s,x_i+\varepsilon\omega,v_{s+1})\dd \omega \dd v_{s+1} \\
&= \sum_{i=1}^s (N-s)\, \varepsilon^{d-1} \Big[\int_{\omega\cdot(v_{s+1}-v_i)\,<\,0} + \int_{\omega\cdot(v_{s+1}-v_i)\,>\,0} \Big] \omega \cdot (v_{s+1}-v_i) \\
& \ \ \ \ \ \ \ \ \ \ \ \ \ \ \ \times f_N^{(s+1)}(t,Z_s,x_i+\varepsilon\omega,v_{s+1})\dd \omega \dd v_{s+1} \\
&= \sum_{i=1}^s (N-s)\, \varepsilon^{d-1} \int_{\omega\cdot(v_{s+1}-v_i)\,>\,0} \hspace{-19mm} \big[\omega \cdot (v_{s+1}-v_i)\big]_+ \\
&\hspace{25mm} \times \Big(f_N^{(s+1)}(t,Z_s,x_i+\varepsilon\omega,v_{s+1}) - f_N^{(s+1)}(t,Z_s,x_i-\varepsilon\omega,v_{s+1}) \Big) \dd \omega \dd v_{s+1},
\end{align*}
where the last line is obtained after performing the change of variables $\omega \rightarrow -\omega$ in the first term constituted of the pre-collisional velocities.\\
Now we can use the boundary conditions verified by the marginal $f_N^{(s+1)}$ in order to remove the post-collisional arguments in the integrand, replacing $f^{(s+1)}_N (t,Z_s,x_i+\varepsilon\omega,v_{s+1})$ by $f^{(s+1)}_N (t,x_1,x_1,\dots,x_i,v_i',\dots,x_i+~\varepsilon\omega,v_{s+1}')$.\\
Finally, taking now formally the limit $\varepsilon \rightarrow 0$, and up to assume that $N\varepsilon^{d-1} \rightarrow 1$, we find the limiting collision operator:
\begin{align*}
\mathcal{C}^0_{s,s+1} f^{(s+1)}(t,Z_s) &= \sum_{i=1}^s \int_{\mathbb{S}^{d-1}\times\mathbb{R}^d} \hspace{-10mm} \big[\omega \cdot (v_{s+1}-v_i)\big]_+ \\
&\times \Big(f^{(s+1)}(t,x_1,v_1,\dots,x_i,v_i',\dots,x_i,v_{s+1}') - f^{(s+1)}(t,Z_s,x_i,v_{s+1}) \Big) \dd \omega \dd v_{s+1},
\end{align*}
where $[x]_+$ denotes the nonnegative part of $x\in\mathbb{R}$, that is $[x]_+ = x$ if $x\geq0$, and $[x]_+ = 0$ if $x<0$.\\
Using this operator, we can define now the limiting hierarchy obtained from the BBGKY one, called the \emph{Boltzmann hierarchy}, which writes
\begin{equation}
\label{SECT1EquatHieraBoltzFormeDiffe}
\partial_t f^{(s)} + \sum_{i=1}^s v_i \cdot \nabla_{x_i} f^{(s)} = \mathcal{C}^0_{s,s+1} f^{(s+1)}.
\end{equation}
As for the BBGKY hierarchy, the integrated version of the Boltzmann hierarchy, which follows, will be the most useful thereafter:
\begin{equation}
\label{SECT1HieraBoltzVersiInteg/Tmps}
f^{(s)}(t,Z_s) ) = \big(\mathcal{T}^{s,0}_t f^{(s)}(0,\cdot)\big)(Z_s) + \int_0^t \mathcal{T}^{s,0}_{t-u} \mathcal{C}^0_{s,s+1} f^{(s+1)}(u,Z_s) \dd u,
\end{equation}
where $\mathcal{T}^{s,0}_t$ denotes the \emph{backwards} free flow with the specular boundary conditions, defined using the associated free transport, that is $\big(\mathcal{T}^{s,0}_t f\big)(Z_s) = f\big(T^{s,0}_{-t}(Z_s)\big)$.

\paragraph*{The Boltzmann equation as the first equation of the Boltzmann hierarchy for tensorized functions.}

The Boltzmann hierarchy is called this way because it is deeply linked with the Boltzmann equation. To be more accurate, if the second unknown $f^{(2)}$ of the sequence of solutions $(f^{(s)})_{s \geq 1}$ of the Boltzmann hierarchy is the tensorization of the first unknown $f^{(1)}$, that is if $f^{(2)}(t,x_1,v_1,x_2,v_2) = f^{(1)}(t,x_1,v_1) f^{(1)}(t,x_2,v_2)$, then $f^{(1)}$ solves the Boltzmann equation.\\ Conversely, if $f$ is a solution of the Boltzmann equation, then the sequence of its tensorizations $(f^{(s)})_{s\geq1} = (f^{\otimes s})_{s \geq 1}$ provides a solution of the Boltzmann hierarchy. This remark, together with the formal derivation of the Boltzmann hierarchy from the BBGKY hierarchy suggests then an interesting plan to obtain a derivation of the Boltzmann equation itself.\\
The assertion that the quantity $N\varepsilon^{d-1}$ stays constant when $N$ goes to infinity is called the \emph{Boltzmann-Grad} limit, introduced by Grad in his pioneering work \cite{Grad} casting for the first time the bridge described between the BBGKY hierarchy and the Boltzmann equation. Physically, it means that the mean free path of a particle remains constant. It also implies that the volume $N\varepsilon^d$ occupied by the particles is going to zero as the number of the particles increases, hence one usually calls this condition the \emph{low density limit}.\\
Concerning now the different steps providing formally the Boltzmann hierarchy, it is quite clear that their order crucially matters: if we had performed the limit $\varepsilon \rightarrow 0$ \emph{before} using the boundary condition
$$
f^{(s+1)}(\dots,v_i',\dots,v_{s+1}') = f^{(s+1)}(\dots,v_i,\dots,v_{s+1}),
$$
then the collision term would have been simply $0$, that is we would have recovered the free transport in the limit. In addition, we decided to remove the post-collisional arguments in the integrand: this can be formally motivated by the fact that the equation involves a \emph{backwards in time} transport, so when two particles collide, it is important to give the pre-collisional velocities associated to a post-collisional pair in order to be able to reconstruct the path of the particles backwards. If we had removed the pre-collisional arguments instead, we would have recover the opposite of the collision term, and then the backwards in time Boltzmann equation.\\
Finally, and even if it was already mentionned several times, all those manipulations and the hierarchies obtained are only formal so far. A first important challenge is to give a rigorous sense to those objects, which is the purpose of the following section.

\section{Introduction of the rigorous functional setting}

\subsection{Definition of the collision operator for the BBGKY hierarchy}
\label{SSect2.1__DefinOperaColliBBGKY}

The free transport with specular reflexion preserves the continuity, and then the Boltzmann hierarchy makes sense for continuous functions. But the hard sphere transport is only defined almost everywhere, so we have to deal with another set of functions. Let us then here study the BBGKY hierarchy, that has to make sense for Lebesgue functions.\\
The general formula \eqref{SECT1HieraBBGKYVersiInteg/Tmps}, using the collision operator described by \eqref{SECT1OperaCollision_HieraBBGKY}, is based on an integration on a manifold with a positive codimension in the phase space. Indeed, if we focus only on the collision term \eqref{SECT1OperaCollision_HieraBBGKY} (and forget for the moment about the integration in time), we see that this collision term is obtained by integrating over the variables $\omega \in \mathbb{S}^{d-1}$ and $v_{s+1}\in\mathbb{R}^d$. Since the trace of a Lebesgue function is not well defined in general, one will need an additional result to give a sense to this term.\\
The problem is for the first time mentionned (and addressed) in \cite{GSRT}, we will here only sketch the main steps of the solution. However, the presentation of the rigorous definition of this term in \cite{GSRT} is quite fast, one may refer to \cite{PhDTT} for a more detailed proof, which leads to the following result. For the sake of completeness, the proof is sketched below.

\begin{theor}[Definition of the collision operator of the BBGKY hierarchy for functions of $\ensuremath{\mathcal{C}\big([0,T],L^\infty(\mathcal{D}^\varepsilon_{s+1})\big)}$ decaying sufficiently fast at infinity in the velocity variables]
\label{SECT2TheorDefinOperaColliBBGKY}
Let $s$ be a positive integer, $\varepsilon$ and $T$ be two positive numbers.\\
Let in addition $g_{s+1} : [0,T] \times \mathbb{R}_+ \rightarrow \mathbb{R}_+$ be a function verifying:
\begin{itemize}[leftmargin=*]
\item $(t,x) \mapsto g_{s+1}(t,x)$ is measurable and almost everywhere positive,
\item for all $x \in \mathbb{R}_+$, the function $t \mapsto g_{s+1}(t,x)$ is increasing,
\item for all $t \in [0,T]$ and almost every $(v_1,\dots,v_s) \in \mathbb{R}^{ds}$, the function \\$v_{s+1} \mapsto \big\vert V_{s+1} \big\vert g_{s+1}\big(t,\big\vert V_{s+1} \big\vert\big)$ is integrable on $\mathbb{R}^d$,
\item for all $t \in [0,T]$, the function $(v_1,\dots,v_s) \mapsto \displaystyle{\int}_{\mathbb{R}^d} \big\vert V_{s+1} \big\vert g_{s+1}\big(t,\big\vert V_{s+1} \big\vert\big) \dd v_{s+1}$ is bounded almost everywhere, and
$$
\vertii{\int_{\mathbb{R}^d} \mathds{1}_{\vert V_{s+1} \vert \geq R} \big\vert V_{s+1} \big\vert g_{s+1}\big(t,\big\vert V_{s+1} \big\vert\big) \dd v_{s+1}}_{L^\infty([0,T],L^\infty(\mathcal{D}^\varepsilon_{s+1}))}
$$
converges to zero as $R$ goes to infinity.
\end{itemize}
Then, for every integer $1 \leq i \leq s$, and for any function $h^{(s+1)} \in \mathcal{C}\big([0,T],L^\infty(\mathcal{D}^\varepsilon_{s+1})\big)$ such that there exists $\lambda \in \mathbb{R}_+$ such that
$$
\vertii{ h^{(s+1)}(t,Z_{s+1}) }_{L^\infty([0,T],L^\infty(\mathcal{D}^\varepsilon_{s+1}))} \leq \lambda \vertii{ g_{s+1}\big(t,\big\vert V_{s+1} \big\vert\big) }_{L^\infty([0,T],L^\infty(\mathcal{D}^\varepsilon_{s+1}))},
$$
the function $\mathcal{C}^\varepsilon_{s,s+1,\pm,i}\mathcal{T}^{s+1,\varepsilon}_t h^{(s+1)}$ is a well defined element of $L^\infty\big([0,T] \times \mathcal{D}^\varepsilon_{s+1} \big)$, and one has almost everywhere on $[0,T] \times \mathcal{D}^\varepsilon_s$:
$$
\Big\vert \mathcal{C}^\varepsilon_{s,s+1,\pm,i}\mathcal{T}^{s+1,\varepsilon}_t h^{(s+1)}(t,Z_s) \Big\vert \leq \lambda \varepsilon^{d-1} \frac{\big\vert \mathbb{S}^{d-1} \big\vert}{2} \int_{\mathbb{R}^d} \big(\vert v_i \vert + \vert v_{s+1} \vert\big)g_{s+1}\big( t,\big\vert V_{s+1} \big\vert \big) \dd v_{s+1}.
$$
\end{theor}

\begin{proof}[Sketch of proof of Theorem \ref{SECT2TheorDefinOperaColliBBGKY}]

Let us assume, in order to simplify the presentation, that the function $f^{(s+1)}_N$ on which the collision term \eqref{SECT1OperaCollision_HieraBBGKY} is acting does not depend on time, and let focus only on the second term of the collision operator, which does not involve scattering.\\
The first ingredient is the Fubini theorem: for a function $f:X\times Y \rightarrow \mathbb{R}$ which is integrable for the product measure $\dd x \otimes \dd y$ on $X \times Y$, one knows that
$$
y \mapsto \int_X f(x,y) \dd x
$$
is defined almost everywhere and is integrable with respect to the measure $\dd y$ on $Y$. In other words, this theorem can be seen as a way to define traces in some particular cases.\\
A function $f^{(s+1)}_N$ defined on the phase space $\mathcal{D}^\varepsilon_{s+1}$ (with $s+1$ particles), which is composed with the map $(Z_s,\omega,v_{s+1}) \mapsto (Z_s,x_i+\varepsilon\omega,v_{s+1})$, defined on $\mathbb{R}^{2ds}\times\mathbb{S}^{d-1}\times\mathbb{R}^d$ and taking its values in $\mathbb{R}^{2d(s+1)}$, and integrated with respect to the variables $\omega$ and $v_{s+1}$, depends in the end on the variable $Z_s$. So in order to apply the Fubini theorem in this case, one should integrate again with respect to this last variable $Z_s$. However it is not enough, since if $f^{(s+1)}_N$ is a measurable function defined on $\mathcal{D}^\varepsilon_{s+1}$, the integration variables $Z_s$, $\omega$ and $v_{s+1}$ cover only a manifold of codimension $1$.\\
The second important idea is then to compose $f^{(s+1)}_N$ with the hard sphere transport for $s+1$ particles, depending of course on the configuration of the system of those $s+1$ hard spheres, but also on an additional parameter: the time $u$. This will play the role of the missing variable in the integration.\\
In the end, this insertion of a transport operator inside the collision term will lead us to consider a "shifted" in time version of the BBGKY hierarchy, called the \emph{conjugated BBGKY hierarchy}, on which the regularity results can be stated\footnote{In \cite{GSRT}, this hierarchy is said to be ``more regular", but it is actually the only one on which one can properly work.}: if one denotes
$$
h_N^{(s)}(Z_{s+1}) = f_N^{(s)}\big(T^{s+1,\varepsilon}_u(Z_{s+1})\big), \text{   that is   } h_N^{(s+1)} = \mathcal{T}^{s+1,\varepsilon}_{-u} f^{(s+1)}_N,
$$
then (formally) the $s$-th equation of the BBGKY hierarchy holds if and only if
\begin{equation}
\label{SECT2HieraBBGKYInteg/TmpsConju}
h^{(s+1)}_N(t,\cdot) = f^{(s+1)}_N(0,\cdot) + \int_0^t \mathcal{T}^{s,\varepsilon}_{-u}\mathcal{C}^{N,\varepsilon}_{s,s+1}\mathcal{T}^{s+1,\varepsilon}_u h^{(s+1)}_N(u,\cdot) \dd u.
\end{equation}
In the end, the collision operator will not be defined alone, but composed with the hard sphere transport. That is why we will talk about the \emph{transport-collision} operator of the BBGKY hierarchy in the sequel, and not about the collision operator only.\\

The very last step, which is of technical order, but which leads to a lot of work, is a series of restrictions that have to be relaxed one by one (see \cite{PhDTT}). To be more explicit, we start by decomposing the collision operator into elementary terms (each one concerning a single particle $1 \leq i \leq s$ chosen to collide, and each one being either in a pre- or in a post-collisional configuration according to the adjunction parameters $(\omega,v_{s+1})$), that is we write
$$
\mathcal{C}^{N,\varepsilon}_{s,s+1} = (N-s)\sum_{i=1}^s \big( \mathcal{C}^\varepsilon_{s,s+1,+,i} -  \mathcal{C}^\varepsilon_{s,s+1,-,i} \big)
$$
where
\begin{align*}
\mathcal{C}^\varepsilon_{s,s+1,\pm,i} h^{(s+1)} = \varepsilon^{d-1} \int_{\substack{\vspace{-2mm}\hspace{-0.3mm}\mathbb{S}^{d-1}_\omega\times\mathbb{R}^d_{v_{s+1}}}} \hspace{-12mm} \big(\omega\cdot(v_{s+1}-v_i\big)_\pm h^{(s+1)}(Z_s,x_i+\varepsilon\omega,v_{s+1}) \dd \omega \dd v_{s+1}
\end{align*}
(here again, this last term is only formally introduced for $h^{(s+1)}$ being a $L^p$ function), and we introduce three cut-off parameters $\delta$, $R_1$ and $R_2$ restricting the domain of integration such that
\begin{itemize}[label=\textbullet]
\item $\vert x_j - x_k \vert > \varepsilon + \sqrt{2}\delta R_2$ for all $1 \leq j < k \leq s+1$ with $(j,k) \neq (i,s+1)$,
\item $x_l\cdot e_1 > \varepsilon/2 + \delta R_2$ for all $1 \leq l \leq s+1$,
\item $\vert X_{s+1} \vert = \vert (x_1,x_2,\dots,x_i,\dots,x_i+\varepsilon\omega) \vert \leq R_1$,
\item $\vert V_{s+1} \vert = \vert (v_1,\dots,v_s,v_{s+1}) \vert \leq R_2$.
\end{itemize}

On this restricted domain of integration, that will be denoted $\mathcal{D}^\varepsilon_{s+1}(\delta,R_1,R_2)$, if the pair of particles $(x_i,v_i)$ and $(x_i+\varepsilon\omega,v_{s+1})$ is in a pre-collisional configuration, for $\delta$ small enough (depending on $R_1$ and $R_2$) the backwards hard sphere transport coincides with the free transport for small times, that is for all $0\leq t \leq \delta$ one has
$$
T^{s+1,\varepsilon}_{-t}(Z_s,x_i+\varepsilon\omega,v_{s+1}) = (X_s-tV_s,V_s,x_i+\varepsilon\omega-tv_{s+1},v_{s+1}),
$$
and since the function
$$
S^-_{s+1}=\left\{
\begin{array}{rl}
\mathcal{D}^\varepsilon_s \times [0,\delta] \times \mathbb{S}^{d-1} \times \mathbb{R}^d &\rightarrow \ \ \mathbb{R}^{2d(s+1)} \\
(Z_s,t,\omega,v_{s+1}) &\mapsto \ \ (X_s-tV_s,x_i+\varepsilon\omega-tv_{s+1},v_{s+1})
\end{array}
\right.
$$
is such that its Jacobian determinant has an absolute value equal to
$$
\varepsilon^{d-1} \big\vert \omega \cdot (v_{s+1}-v_i) \big\vert,
$$
we can at last define the pre-collisional elementary terms of the (truncated) transport-collision operator $\mathcal{C}^\varepsilon_{s,s+1,-,i}(\delta,R_1,R_2) \mathcal{T}^{s+1,\varepsilon}$ using the formula
\begin{align*}
\int_{\substack{\vspace{-2mm}\hspace{-1mm} S^-_{s+1}(\mathcal{D}^\varepsilon_{s+1}(\delta,R_1,R_2))}} \hspace{-24mm} h^{(s+1)}(Z_{s+1}) \dd Z_{s+1} &= \int_{\hspace{-0.4mm}0}^{\delta} \hspace{-2mm} \int_{\mathcal{D}^\varepsilon_s} \int_{\substack{\vspace{-2mm}\hspace{-0.5mm}\mathbb{S}^{d-1}\times\mathbb{R}^d}} \hspace{-10mm} \mathds{1}_{\mathcal{D}^\varepsilon_{s+1}(\delta,R_1,R_2)}\varepsilon^{d-1}\big(\omega\cdot(v_{s+1}-v_i)\big)_- \\
& \ \ \ \ \ \ \ \ \times h^{(s+1)}\big(S^-_{s+1}(Z_s,x_i+\varepsilon\omega,v_{s+1})\big) \dd \omega \dd v_{s+1} \dd Z_s \dd t \\
&= \int_{\hspace{-0.4mm}0}^\delta \hspace{-2mm} \int_{\mathcal{D}^\varepsilon_s} \mathcal{C}^\varepsilon_{s,s+1,-,i}(\delta,R_1,R_2) \mathcal{T}^{s+1,\varepsilon} h^{(s+1)}(t,Z_s) \dd Z_s \dd t,
\end{align*}
for $h^{(s+1)}$ a $L^\infty$ function of the phase space $\mathcal{D}^\varepsilon_{s+1}$.\\
One sees that there is a restriction on the time interval, which can be relaxed (that is one can define the truncated transport-collision operator on any time interval $[0,T]$) thanks to the conservation of the $L^\infty$ norm by the hard sphere transport and a decomposition of any time interval into sub-intervals of length $\leq \delta$. The post-collisional terms $\mathcal{C}^\varepsilon_{s,s+1,+,i}$ are defined in the same way, replacing only the mapping $S^-_{s+1}$ by $S^+_{s+1}$, which is defined as the scattering mapping with $S^-_{s+1}$.\\
\newline
Finally, we relax first the cut-off in the time variable (the parameter $\delta$). We can show that the sequence $\big(\mathcal{C}^\varepsilon_{s,s+1,\pm,i}(\delta,R_1,R_2)\mathcal{T}^{s+1,\varepsilon}_th^{(s+1)}\big)_\delta$ converges strongly in $L^1$ as $\delta\rightarrow 0$ towards a limit, denoted $\mathcal{C}^\varepsilon_{s,s+1,\pm,i}(R_1,R_2)\mathcal{T}^{s+1,\varepsilon}_t h^{(s+1)}$, which is also $L^\infty$, and the convergence holds also in the weak sense in $L^\infty$.\\
The cut-off in the position variable (the parameter $R_1$) can then be relaxed: $\big(\mathcal{C}^\varepsilon_{s,s+1,\pm,i}(R_1,R_2)\mathcal{T}^{s+1,\varepsilon}_t h^{(s+1)}\big)_{R_1}$ converges almost everywhere as $R_1\rightarrow+\infty$ towards a limit, denoted $\mathcal{C}^\varepsilon_{s,s+1,\pm,i}(R_2)\mathcal{T}^{s+1,\varepsilon}_th^{(s+1)}$, which is a $L^\infty$ function (with a supremum which depends on $R_2$).\\
To counter-balance the growth in $R_2$ of the $L^\infty$ norm of $\mathcal{C}^\varepsilon_{s,s+1,\pm,i}(R_2)\mathcal{T}^{s+1,\varepsilon}_th^{(s+1)}$, one has to impose a decrease in the velocity variable for $h^{(s+1)}$. The condition, explicited in Theorem \ref{SECT2TheorDefinOperaColliBBGKY} below, is quite strong, but we can notice that among the few functions verifying this condition can be found the gaussians. With this condition, the sequence $\big(\mathcal{C}^\varepsilon_{s,s+1,\pm,i}(R_2)\mathcal{T}^{s+1,\varepsilon}_th^{(s+1)}\big)_{R_2}$ is a Cauchy sequence in $L^\infty$, and then it is converging as $R_2\rightarrow+\infty$.\\
This long process provides the rigourous definition of the transport-collision operator for the BBGKY hierarchy, up to consider it acting on the set of functions described in the theorem. This concludes the sketch of proof of Theorem \ref{SECT2TheorDefinOperaColliBBGKY}.
\end{proof}

\subsection{A practical functional setting for the hierarchies}

To the best of our knowledge, the previous theorem provides the most general setting in which the collision operator of the BBGKY hierarchy makes sense when it acts on functions (and not on distributions). However, it does not answer the question on the functional setting in which the BBGKY hierarchy is rigourously defined: a solution of the BBGKY hierarchy is a \emph{family} of functions $\big(f^{(s)}_N\big)_s$ such that (formally) $f^{(s)}_N = \mathcal{T}^{s,\varepsilon}_t f^{(s)}_{N,0} + \int_0^t \mathcal{T}^{s,\varepsilon}_{t-u} \mathcal{C}^{N,\varepsilon}_{s,s+1} f^{(s+1)}_N \dd u$, for all $1 \leq s \leq N-1$. In other words, we need to define \emph{families} of functional spaces, that are on the one hand consistent with the action of the collision operator, and on the other hand consistent with the low density limit, in order to be able to compare the solutions of the two hierarchies in the end.\\
Let us then introduce first the relevant functional spaces containing individually the elements $f^{(s)}_N$ and $f^{(s)}$ of the solutions of the two hierarchies, and then complete the introduction of the functional setting with the spaces containing the whole sequences $(f^{(s)}_N)_{1 \leq s \leq N}$ and $(f^{(s)})_{s \geq 0}$.

\paragraph*{Definition of the spaces $X_{\varepsilon,s,\beta}$ and $X_{0,s,\beta}$, the functions of the phase space of $s$ particles bounded by a gaussian in the velocity variables.}

One starts with the definition of the first kind of functional space, in which each marginal will lie. The main difference between the spaces for the BBGKY and the Boltzmann hierarchies, except of course the domain of definition, is the continuity of the functions.

\begin{defin}[Norms $|\cdot|_{\varepsilon,s,\beta}$ and $|\cdot|_{0,s,\beta}$, spaces $X_{\varepsilon,s,\beta}$ and $X_{0,s,\beta}$]
\label{SECT2DefinEspacFonctBornéGauss}
Let $\varepsilon$ and $\beta>0$ be two strictly positive numbers and $s$ be a positive integer. For any function $h^{(s)}$ belonging to $L^\infty\big(\mathcal{D}^\varepsilon_s\big)$, one defines:
$$
|h^{(s)}|_{\varepsilon,s,\beta}=\supess_{Z_s\in\mathcal{D}^\varepsilon_s}\left[\big\vert h^{(s)}(Z_s) \big\vert \exp\left(\frac{\beta}{2}\sum_{i=1}^s|v_i|^2\right)\right],
$$
and the space $X_{\varepsilon,s,\beta}$ as the space of the functions of $L^\infty\big(\mathcal{D}^\varepsilon_s\big)$ with a finite $|\cdot|_{\varepsilon,s,\beta}$ norm, that is:
$$
X_{\varepsilon,s,\beta}=\Big\{ h^{(s)}\in L^\infty\left(\mathcal{D}^\varepsilon_s\right)\ /\ |h^{(s)}|_{\varepsilon,s,\beta}<+\infty\Big\}.
$$
For any function $f^{(s)}$ belonging to $\mathcal{C}_0\big(\big(\overline{\Omega^c}\times\mathbb{R}^d\big)^s\big)$, one defines:
$$
|f^{(s)}|_{0,s,\beta}=\sup_{Z_s\in (\overline{\Omega^c}\times\mathbb{R}^d)^s}\left[\big\vert f^{(s)}(Z_s) \big\vert \exp\left(\frac{\beta}{2}\sum_{i=1}^s|v_i|^2\right)\right],
$$
and the space $X_{0,s,\beta}$ as the space of the continuous functions vanishing at infinity defined on $\big(\overline{\Omega^c}\times\mathbb{R}^d\big)^s$ with a finite $|\cdot|_{0,s,\beta}$ norm, that is:
$$
X_{0,s,\beta}=\big\{ f^{(s)}\in \mathcal{C}_0\big(\big(\overline{\Omega^c}\times\mathbb{R}^d\big)^s\big)\ /\ |f^{(s)}|_{0,s,\beta}<+\infty\big\},
$$
and satisfying the following boundary condition $f^{(s)}(Z_s) = f^{(s)}(\chi^0_s(Z_s))$ for all $Z_s$ belonging to the boundary of $\big(\overline{\Omega^c} \times \mathbb{R}^d \big)^s$\hspace{-1.5mm}, that is such that there exists at least an integer $1 \leq i \leq s$ such that $x_i\cdot e_1 = 0$ and $v_i \cdot e_1 > 0$.
\end{defin}

\paragraph*{Definition of the spaces $\textbf{X}_{\varepsilon,\beta,\mu^\alpha}$ and $\textbf{X}_{0,\beta,\mu^\alpha}$, the sequence of functions of $X_{\cdot,s,\beta}$ with an exponential weight with respect to the number of particles.}

Now that we introduced the functional spaces in which each of the marginals will lie, let us introduce a structure on the sequence of such spaces, with, in addition to a real parameter $\mu$, which is the activity of the solution from a physical point of view, another parameter $\alpha$, strictly positive, and which will be taken equal to $1$ or $2$ in the sequel. The choice of this parameter $\alpha$ plays a role in the definition of the continuity in time introduced in the final spaces, introduced below.

\begin{defin}[Norms $\vertii{\cdot}_{N,\varepsilon,\beta,\mu^\alpha}$ and $\vertii{\cdot}_{0,\beta,\mu^\alpha}$, spaces $\textbf{X}_{N,\varepsilon,\beta,\mu^\alpha}$ and $\textbf{X}_{0,\beta,\mu^\alpha}$]
\label{SECT2DefinEspacSuiteMargiExpon}
Let $N$ be a positive integer. Let $\varepsilon$ an $\beta$ be two strictly positive numbers, $\mu$ be a real number and $\alpha > 0$ be a strictly positive number. For any finite sequence $H_N=\big(h^{(s)}_N\big)_{1\leq s\leq N}$ of functions $h^{(s)}_N$ of $X_{\varepsilon,s,\beta}$, one defines :
$$
\vertii{H_N}_{N,\varepsilon,\beta,\mu^\alpha}=\max_{1\leq s\leq N}\Big(\vert h^{(s)}_N \vert_{\varepsilon,s,\beta}\exp(s^\alpha \mu)\Big),
$$
and the space $\textbf{X}_{N,\varepsilon,\beta,\mu^\alpha}$ as the space of the finite sequences $H_N = \big(h^{(s)}_N\big)_{1\leq s\leq N}$ such that for every $1\leq s\leq N$, $h^{(s)}_N$ belongs to $X_{\varepsilon,s,\beta}$, and such that the sequence $\big( h^{(s)}_N \big)_{1\leq s\leq N}$ has a finite $\vertii{\cdot}_{N,\varepsilon,\beta,\mu^\alpha}$ norm, that is :
$$
\textbf{X}_{N,\varepsilon,\beta,\mu^\alpha}=\Big\{H_N=\big(h^{(s)}_N\big)_{1\leq s\leq N} \in \big(X_{\varepsilon,s,\beta}\big)_{1\leq s\leq N}\ /\ \vertii{H_N}_{N,\varepsilon,\beta,\mu^\alpha}<+\infty\Big\}.
$$
Similarly, for any infinite sequence $F = \big(f^{(s)}\big)_{s\geq 1}$ of functions $f^{(s)}$ of $X_{0,s,\beta}$, one defines :
$$
\vertii{F}_{0,\beta,\mu^\alpha}=\sup_{s\geq 1}\Big(\vert f^{(s)} \vert_{0,s,\beta}\exp(s^\alpha \mu)\Big),
$$
and the space $\textbf{X}_{0,\beta,\mu^\alpha}$ as the space of the infinite sequences $\big( f^{(s)} \big)_{s\geq 1}$ such that for every $s\geq 1$, $f^{(s)}$ belongs to $X_{0,s,\beta}$, and such that the sequence $\big( f^{(s)} \big)_{s\geq 1}$ has a finite $\vertii{\cdot}_{0,\beta,\mu^\alpha}$ norm, that is :
$$
\textbf{X}_{0,\beta,\mu^\alpha}=\Big\{F=\big( f^{(s)} \big)_{s\geq 1} \in \big(X_{0,s,\beta}\big)_{s\geq 1}\ /\ \vertii{F}_{0,\beta,\mu^\alpha}<+\infty\Big\}.
$$
\end{defin}

The following result describes the embeddings that exist between the spaces $X_{\varepsilon,s,\beta}$ for different parameters $\beta$ on the one hand, and on the other hand between the spaces $\textbf{X}_{N,\varepsilon,\beta,\mu^\alpha}$ for different parameters $\beta$ and $\mu$. This will be useful to define functional spaces that are stable under the action of the collision operators.

\begin{propo}[Embeddings of the spaces $X_{\varepsilon,s,\beta}$, $X_{0,s,\beta}$, $\textbf{X}_{N,\varepsilon,\beta,\mu^\alpha}$ and $\textbf{X}_{0,\beta,\mu^\alpha}$]
\label{SECT2PropoIncluEspaXHierarchie}
Let $s$ be a positive integer and $\varepsilon$ be a strictly positive number.
\begin{itemize}[leftmargin=*]
\item For any $\beta\leq\beta'$, one has $X_{\varepsilon,s,\beta'} \subset X_{\varepsilon,s,\beta}$ and $X_{0,s,\beta'} \subset X_{0,s,\beta}$, and if $h^{(s)}$ belongs to $X_{\varepsilon,s,\beta'}$ (respectively $f^{(s)}$ belongs to $X_{0,s,\beta'}$), one has 
$\vert h^{(s)} \vert_{\varepsilon,s,\beta} \leq \vert h^{(s)} \vert_{\varepsilon,s,\beta'}$ (respectively $\vert f^{(s)} \vert_{0,s,\beta} \leq \vert f^{(s)} \vert_{0,s,\beta'}$).
\item For any $\beta \leq \beta'$, any $\mu \leq \mu'$ and any $\alpha > 0$, one has $\textbf{X}_{N,\varepsilon,\beta',\mu'^\alpha} \subset \textbf{X}_{N,\varepsilon,\beta,\mu^\alpha} $ and $\textbf{X}_{0,\beta',\mu'^\alpha} \subset \textbf{X}_{0,\beta,\mu^\alpha}$, and if $\big( h^{(s)}_N \big)_{1\leq s\leq N}$ belongs to $\textbf{X}_{N,\varepsilon,\beta',\mu'^\alpha}$ (respectively $\big( f^{(s)} \big)_{s\geq 1}$ belongs to $\textbf{X}_{0,\beta',\mu'^\alpha}$), one has\\
$\vertii{\big( h^{(s)}_N \big)_{1\leq s\leq N}}_{N,\varepsilon,\beta,\mu^\alpha} \leq\vertii{\big( h^{(s)}_N \big)_{1\leq s\leq N}}_{N,\varepsilon,\beta',\mu'^\alpha}$
(respectively $\vertii{\big( f^{(s)} \big)_{s\geq N}}_{0,\beta,\mu^\alpha} \leq \vertii{\big( f^{(s)} \big)_{s\geq 1}}_{0,\beta',\mu'^\alpha})$.
\end{itemize}
\end{propo}

\paragraph*{Definition of the spaces $\widetilde{\textbf{X}}_{\varepsilon,\widetilde{\beta},\widetilde{\mu}^\alpha}$ and $\widetilde{\textbf{X}}_{0,\widetilde{\beta},\widetilde{\mu}^\alpha}$, the functions of sequences belonging to $\textbf{X}_{\cdot,\widetilde{\beta}(t),\widetilde{\mu}(t)^\alpha}$ at time $t$.}
\label{SSSecDefinFonctSuiteUnifB}

It will be important in the sequel to enable a loss of regularity of the marginals when time grows (translated into a growth of the parameters $\beta$ and $\mu$, for the embeddings of Proposition \ref{SECT2PropoIncluEspaXHierarchie} hold). One will then define spaces of time-dependent functions.\\
From this point, there are mainly two possibilities to define the relevant spaces of time-dependent functions taking their values in $\textbf{X}_{\cdot,\widetilde{\beta}(t),\widetilde{\mu}(t)}$, depending on the value of $\alpha$: this value has to be balanced with the regularity with respect to time, in order to have stable spaces under the action of the collision operator.\\
The motivation of introducting such a parameter is the choice of the continuity in time introduced in \cite{GSRT}, which is uniform in the number of particles $s$. With such a strong continuity in time property, it is possible to show (see \cite{PhDTT}) that the weight $\alpha=1$ is too weak, and has to be replaced by $\alpha > 1$. On the other hand, this value would lead in practice to consider weird initial data, that are not meeting the expected physical properties for marginals of a distribution function. As a consequence, in the sequel we will focus on the practical choice $\alpha = 1$, up to relax the continuity in time property. Only in this section, we will give the proper definition of the spaces  $\widetilde{\textbf{X}}_{\cdot,\widetilde{\beta},\widetilde{\mu}^\alpha}$, for $\alpha = 1$ or $2$ such that the collision operator is stable on those spaces.

\paragraph*{The case of $\alpha = 2$, and uniform continuity in time in the parameter $s$.}
We follow here the definition given in the erratum version of the article \cite{GSRT}\footnote{See the last Definition 5.2.4 of Section 5.2 ``Functional spaces and statement of the results".}, in the sense that we require a uniformly $s$ continuity in time.

\begin{defin}[Norms $\vertiii{\cdot}_{N,\varepsilon,\tilde{\beta},\tilde{\mu}^2}$ and $\vertiii{\cdot}_{0,\tilde{\beta},\tilde{\mu}^2}$, spaces $\widetilde{\textbf{X}}_{N,\varepsilon,\tilde{\beta},\tilde{\mu}^2}$ and $\widetilde{\textbf{X}}_{0,\tilde{\beta},\tilde{\mu}^2}$]
\label{Blablabla}
Let $N$ be a positive integer. Let $\varepsilon$ be a strictly positive number. For any $T>0$, any strictly positive, non increasing function $\tilde{\beta}$, any non increasing function $\tilde{\mu}$, both defined on $[0,T]$, and any function $\widetilde{H}_N:[0,T] \rightarrow\, \bigcup_{t\in[0,T]} \textbf{X}_{N,\varepsilon,\widetilde{\beta}(t),\widetilde{\mu}^2(t)}$, $t \mapsto\, \widetilde{H}_N(t)=\left(h^{(s)}_N(t)\right)_{1\leq s \leq N}$ such that $\widetilde{H}_N(t) \in \textbf{X}_{N,\varepsilon,\widetilde{\beta}(t),\widetilde{\mu}^2(t)}$ for all $t\in[0,T]$, we define
$$
\vertiii{\widetilde{H}_N}_{N,\varepsilon,\tilde{\beta},\tilde{\mu}^2}=\sup_{0\leq t\leq T}\vertii{\widetilde{H}_N(t)}_{N,\varepsilon,\tilde{\beta}(t),\tilde{\mu}^2(t)},
$$
and we define the space $\widetilde{\textbf{X}}_{N,\varepsilon,\tilde{\beta},\tilde{\mu}^2}$ as the space of such functions
$\widetilde{H}_N$ with a finite $\vertiii{\cdot}_{N,\varepsilon,\widetilde{\beta},\widetilde{\mu}^2}$ norm, and verifying the left continuity in time hypothesis:
\begin{align}
\label{SECT2DEFINEspacSuiteTempsContealph2}
\forall t \in\ ]0,T],\ \lim_{u\rightarrow t^-} \vertii{\widetilde{H}_N(t) - \widetilde{H}_N(u)}_{N,\varepsilon,\widetilde{\beta}(t),\widetilde{\mu}^2(t)} = 0.
\end{align}
Similarly, for any $T>0$, any strictly positive, non increasing function $\tilde{\beta}$ and any non increasing function $\tilde{\mu}$, both defined on $[0,T]$, and any function $\widetilde{F}: [0,T] \rightarrow\, \bigcup_{t\in[0,T]} \textbf{X}_{N,\varepsilon,\widetilde{\beta}(t),\widetilde{\mu}^2(t)}$, $t \mapsto\, \widetilde{F}(t)=\left(f^{(s)}(t)\right)_{s\geq 1}$ such that $\widetilde{F}(t) \in \textbf{X}_{0,\widetilde{\beta}(t),\widetilde{\mu}^2(t)}$ for all $t\in[0,T]$, we define
$$
\vertiii{\widetilde{F}}_{0,\tilde{\beta},\tilde{\mu}^2}=\sup_{0\leq t\leq T}\vertii{\widetilde{F}(t)}_{0,\tilde{\beta}(t),\tilde{\mu}^2(t)},
$$
and we define the space $\widetilde{\textbf{X}}_{0,\tilde{\beta},\tilde{\mu}^2}$ as the space of such functions $\widetilde{F}$ with a finite $\vertiii{\cdot}_{0,\widetilde{\beta},\widetilde{\mu}^2}$ norm, and verifying the left continuity in time hypothesis:
\begin{align}
\label{SECT2DEFINEspacSuiteTempsCont0alph2}
\forall t \in\ ]0,T],\ \lim_{u\rightarrow t^-} \vertii{\widetilde{F}(t) - \widetilde{F}(u)}_{0,\widetilde{\beta}(t),\widetilde{\mu}^2(t)} = 0.
\end{align}
\end{defin}

\begin{remar}
To be meaningful, the continuity conditions \eqref{SECT2DEFINEspacSuiteTempsContealph2} and \eqref{SECT2DEFINEspacSuiteTempsCont0alph2} use Proposition \ref{SECT2PropoIncluEspaXHierarchie}, together with the crucial fact that the functions $\widetilde{\beta}$ and $\widetilde{\mu}$ are assumed to be non increasing.
\end{remar}

\paragraph*{The case of $\alpha = 1$, and continuity in time for every integer $s$.}

Let us now introduce the space that will be the most useful for the rest of this work. We choose $\alpha=1$, and require a less restrictive condition of continuity in time than for the case $\alpha=2$: instead of having a continuity condition in the $\vertii{\cdot}_{\cdot,\widetilde{\beta}(t),\widetilde{\mu}(t)^1}$ norm, we will require, for any value of the parameter $s$, a continuity condition in the $\vert \cdot \vert_{\cdot,s,\widetilde{\beta}(t)}$ norm.

\begin{defin}[Norms $\vertiii{\cdot}_{N,\varepsilon,\tilde{\beta},\tilde{\mu}^1}$ and $\vertiii{\cdot}_{0,\tilde{\beta},\tilde{\mu}^1}$, spaces $\widetilde{\textbf{X}}_{N,\varepsilon,\tilde{\beta},\tilde{\mu}^1}$ and $\widetilde{\textbf{X}}_{0,\tilde{\beta},\tilde{\mu}^1}$]
\label{SECT2DefinEspacSuiteMargiTemps}
Let $N$ be a positive integer, $\varepsilon$ be a strictly positive number. For any $T>0$, any strictly positive, non increasing function $\tilde{\beta}$, any non increasing function $\tilde{\mu}$, both defined on $[0,T]$, and any function $\widetilde{H}_N:[0,T] \rightarrow\, \bigcup_{t\in[0,T]} \textbf{X}_{N,\varepsilon,\widetilde{\beta}(t),\widetilde{\mu}^1(t)}$, $
t \mapsto\, \widetilde{H}_N(t)=\left(h^{(s)}_N(t)\right)_{1\leq s \leq N}$ such that $\widetilde{H}_N(t) \in \textbf{X}_{N,\varepsilon,\widetilde{\beta}(t),\widetilde{\mu}^1(t)}$ for all $t\in[0,T]$, we define
$$
\vertiii{\widetilde{H}_N}_{N,\varepsilon,\tilde{\beta},\tilde{\mu}^1}=\sup_{0\leq t\leq T}\vertii{\widetilde{H}_N(t)}_{N,\varepsilon,\tilde{\beta}(t),\tilde{\mu}^1(t)},
$$
and we define the space $\widetilde{\textbf{X}}_{N,\varepsilon,\tilde{\beta},\tilde{\mu}^1}$ as the space of such functions
$\widetilde{H}_N$ with a finite $\vertiii{\cdot}_{N,\varepsilon,\widetilde{\beta},\widetilde{\mu}^1}$ norm, and verifying the left continuity in time hypothesis:
\begin{align}
\label{SECT2DEFINEspacSuiteTempsConte}
\forall t \in\ ]0,T],\ \forall\ 1\leq s\leq N,\ \ \lim_{u\rightarrow t^-} \big\vert h^{(s)}_N(t) - h^{(s)}_N(u) \big\vert_{\varepsilon,s,\widetilde{\beta}(t)} = 0.
\end{align}
Similarly, for any $T>0$, any strictly positive, non increasing function $\tilde{\beta}$ and any non increasing function $\tilde{\mu}$ both defined on $[0,T]$, and any function $\widetilde{F}:[0,T] \rightarrow\, \bigcup_{t\in[0,T]} \textbf{X}_{N,\varepsilon,\widetilde{\beta}(t),\widetilde{\mu}^1(t)}$, $t \mapsto\, \widetilde{F}(t)=\left(f^{(s)}(t)\right)_{s\geq 1}$ such that $\widetilde{F}(t) \in \textbf{X}_{0,\widetilde{\beta}(t),\widetilde{\mu}^1(t)}$ for all $t\in[0,T]$, we define
$$
\vertiii{\widetilde{F}}_{0,\tilde{\beta},\tilde{\mu}^1}=\sup_{0\leq t\leq T}\vertii{\widetilde{F}(t)}_{0,\tilde{\beta}(t),\tilde{\mu}^1(t)},
$$
and we define the space $\widetilde{\textbf{X}}_{0,\tilde{\beta},\tilde{\mu}^1}$ as the space of such functions
$\widetilde{F}$ with a finite $\vertiii{\cdot}_{0,\widetilde{\beta},\widetilde{\mu}^1}$ norm, and verifying the left continuity in time hypothesis:
\begin{align}
\label{SECT2DEFINEspacSuiteTempsCont0}
\forall t \in\ ]0,T],\ \forall\ 1\leq s\leq N,\ \lim_{u\rightarrow t^-} \big\vert h^{(s)}_N(t) - h^{(s)}_N(u) \big\vert_{0,s,\widetilde{\beta}(t)} = 0.
\end{align}
\end{defin}

The spaces $\widetilde{\textbf{X}}_{N,\varepsilon,\widetilde{\beta},\widetilde{\mu}^\alpha}$ and $\widetilde{\textbf{X}}_{0,\widetilde{\beta},\widetilde{\mu}^\alpha}$ defined in the previous section satisfy the following regularity property:

\begin{propo}[Banach space structure of the spaces $\widetilde{\textbf{X}}_{N,\varepsilon,\widetilde{\beta},\widetilde{\mu}^\alpha}$ and $\widetilde{\textbf{X}}_{0,\widetilde{\beta},\widetilde{\mu}^\alpha}$]
\label{THEORComplEspacFonctBBGKY}
Let $N$ be a positive integer. Let $\varepsilon$ be a strictly positive number. For any $T>0$, any strictly positive, non increasing function $\tilde{\beta}$ and any non increasing function $\tilde{\mu}$, both defined on $[0,T]$, and for $\alpha = 1$ or $2$, the spaces $\widetilde{\textbf{X}}_{N,\varepsilon,\widetilde{\beta},\widetilde{\mu}^\alpha}$ and $\widetilde{\textbf{X}}_{0,\widetilde{\beta},\widetilde{\mu}}$ are Banach spaces.
\end{propo}

Since the norm $\vertiii{\cdot}_{\cdot,\tilde{\beta},\tilde{\mu}}$ is defined using suprema, the proof of the previous proposition is very close to the Riesz-Fischer theorem, establishing the completeness of the $L^\infty$ spaces. See \cite{PhDTT} for a proof of Proposition \ref{THEORComplEspacFonctBBGKY}.\\
\newline
From this point, the parameter $\alpha$ will \emph{always} be taken equal to $1$, and will be omitted in the notations in what follows.

\subsection{Existence and uniqueness of the solutions of the hierarchies in the spaces $\widetilde{\textbf{X}}_{N,\varepsilon,\widetilde{\beta},\widetilde{\mu}}$ and $\widetilde{\textbf{X}}_{0,\widetilde{\beta},\widetilde{\mu}}$}

We are now able to state a result of existence and uniqueness for the solutions of the two hierarchies. This result will use the Banach space structure of the spaces $\widetilde{\textbf{X}}_{N,\varepsilon,\widetilde{\beta},\widetilde{\mu}}$ and $\widetilde{\textbf{X}}_{0,\widetilde{\beta},\widetilde{\mu}}$, and the rewriting of the generic equations of the hierarchies as a fixed point problem.

\begin{defin}[BBGKY operator, Boltzmann operator]
\label{SECT2DefinOperateursBBGKYBoltz}
For a sequence of initial data $(f^{(s)}_{N,0})_{1\leq s\leq N} \in \textbf{X}_{N,\varepsilon,\widetilde{\beta}(0),\widetilde{\mu}(0)}$, we introduce the \emph{BBGKY operator}, acting on the sequences $(h^{(s)}_N)_{1 \leq s \leq N}$ of $\widetilde{\textbf{X}}_{N,\varepsilon,\widetilde{\beta},\widetilde{\mu}}$, denoted as $\mathfrak{E}_{N,\varepsilon}\big((f^{(s)}_{N,0})_{1\leq s\leq N},\cdot\,\big)$, and defined as
$$
\mathfrak{E}_{N,\varepsilon}\big((f^{(s)}_{N,0})_{1\leq s\leq N},(h^{(s)}_N)_{1 \leq s \leq N}\big) = \Big(\mathfrak{E}_{N,\varepsilon}^{(s)}\big((f^{(s)}_{N,0})_{1\leq s\leq N},(h^{(s)}_N)_{1 \leq s \leq N}\big)\Big)_{1 \leq s \leq N}
$$
where
$$
\mathfrak{E}_{N,\varepsilon}^{(s)}\big((f^{(s)}_{N,0})_s,(h^{(s)}_N)_s\big)(t,\cdot\,) = f^{(s)}_{N,0}(\cdot) + \int_0^t \mathcal{T}^{s,\varepsilon}_{-u}\mathcal{C}^{N,\varepsilon}_{s,s+1}\mathcal{T}^{s+1,\varepsilon}_u h^{(s+1)}_N(u,\cdot\,)\dd u,
$$
and this, for all $1 \leq s \leq N-1$ and all $t$ (the case $s=N$ is just given by $\mathfrak{E}_{N,\varepsilon}^{(N)}\big((f^{(s)}_{N,0})_s,(h^{(s)}_N)_s\big)(t,\cdot\,) = f^{(N)}_{N,0}(\cdot)$).\\
The same kind of operator can be introduced as well for the Boltzmann hierarchy: for a sequence of initial data $(f^{(s)}_0)_{s\geq 1} \in \textbf{X}_{0,\widetilde{\beta}(0),\widetilde{\mu}(0)}$, we introduce the \emph{Boltzmann operator}, acting on the sequences $(f^{(s)})_{s \geq 1}$ of $\widetilde{\textbf{X}}_{0,\widetilde{\beta},\widetilde{\mu}}$, denoted as $\mathfrak{E}_0\big((f^{(s)}_0)_{s\geq 1},\cdot\,\big)$, and defined as
$$
\mathfrak{E}_0\big((f^{(s)}_0)_{s\geq 1},(f^{(s)})_{s \geq 1}\big) = \Big(\mathfrak{E}_0^{(s)}\big((f^{(s)}_0)_{s\geq1},(f^{(s)})_{s\geq1}\big)\Big)_{s \geq 1}
$$
where
$$
\mathfrak{E}_0^{(s)}\big((f^{(s)}_0)_s,(f^{(s)})_s\big)(t,\cdot\,) = \mathcal{T}^{s,0}_tf^{(s)}_0(\cdot) + \int_0^t \mathcal{T}^{s,0}_{t-u}\mathcal{C}^0_{s,s+1} f^{(s+1)}(u,\cdot\,)\dd u,
$$
for all $s \geq 1$ and all $t$.
\end{defin}

We have then the following reformulation: $(h^{(s)}_N)_{1 \leq s \leq N}$ is a solution of the conjugated BBGKY hierarchy \eqref{SECT2HieraBBGKYInteg/TmpsConju} associated to the initial data $(f^{(s)}_{N,0})_{1\leq s \leq N}$ if and only if $(h^{(s)}_N)_{1 \leq s \leq N} = \mathfrak{E}_{N,\varepsilon}\big((f^{(s)}_{N,0})_{1\leq s\leq N},(h^{(s)}_N)_{1 \leq s \leq N}\big)$ (and of course the same rewritting holds also for the Boltzmann hierarchy).\\
Thanks to this reformulation into a fixed point problem, we can now obtain the following theorem.

\begin{theor}[Joint local in time existence and uniqueness of solutions to the BBGKY and Boltzmann hierarchies]
\label{SECT2TheorExistUniciSolutHiera}
Let $\beta_0$ be a strictly positive real number and $\mu_0$ a real number. There exist a time $T>0$, a strictly positive decreasing function $\tilde{\beta}$ and a decreasing function $\tilde{\mu}$ defined on $[0,T]$ such that $\tilde{\beta}(0)=\beta_0$, $\tilde{\mu}(0)=\mu_0$ and such that for any positive integer $N$ and any strictly positive number $\varepsilon > 0$ verifying the Boltzmann-Grad limit $N\varepsilon^{d-1} = 1$, any pair of sequences of initial data $F_{N,0}=\big(f^{(s)}_{N,0}\big)_{1\leq s\leq N}$ and $F_0 = \big(f^{(s)}_0\big)_{s\geq 1}$ belonging respectively to $\textbf{X}_{N,\varepsilon,\beta_0,\mu_0}$ and $\textbf{X}_{0,\beta_0,\mu_0}$ give rise respectively to a unique solution $H_N=t \mapsto \big(h^{(s)}_N(t,\cdot)\big)_{1\leq s\leq N}$ in $\widetilde{\textbf{X}}_{N,\varepsilon,\tilde{\beta},\tilde{\mu}}$ to the BBGKY hierarchy with initial datum $F_{N,0}$ and $F = t\mapsto \big( f^{(s)}(t,\cdot)\big)_{s\geq 1}$ in $\widetilde{\textbf{X}}_{0,\tilde{\beta},\tilde{\mu}}$ to the Boltzmann hierarchy with initial datum $F_0$, that is there exists a unique pair of elements $H_N$ and $F$ belonging respectively to the spaces $\widetilde{\textbf{X}}_{N,\varepsilon,\widetilde{\beta},\widetilde{\mu}}$ and $\widetilde{\textbf{X}}_{0,\widetilde{\beta},\widetilde{\mu}}$ such that, for every $t\in [0,T]$:
$$
H_N(t) = \mathfrak{E}_{N,\varepsilon} \big( F_{N,0}, H_N \big)(t),
$$
and
$$
F(t) = \mathfrak{E}_0 \big( F_0, F \big)(t).
$$
Moreover, the decreasing functions $\widetilde{\beta} = \widetilde{\beta}_\lambda$ and $\widetilde{\mu} = \widetilde{\mu}_\lambda$ are affine, given by the expressions
$\widetilde{\beta}_\lambda = [0,T] \rightarrow\, \mathbb{R}_+^*$, $t \mapsto\, \widetilde{\beta}_\lambda(t) = \beta_0 - \lambda t$ and $\widetilde{\mu}_\lambda = [0,T] \rightarrow\, \mathbb{R}$, $t \mapsto\, \widetilde{\mu}_\lambda(t) = \mu_0 - \lambda t$, for $\lambda$ depending only on $\beta_0$ and $\mu_0$, and for $T$ depending only on $\beta_0,\mu_0$ and $\lambda$.
\end{theor}

\begin{remar}
Here it is important to notice that, since in the end the goal is to obtain a convergence result of the solutions of the BBGKY hierarchy towards the solution of the Boltzmann hierarchy, the time of existence of solutions has to be the same for the two hierarchies (this comes from the fact that the upper bound in the control \eqref{SECT2InegaContractioOperaColli} in Section \ref{AppenSSectContractioOperaHiera} is exactly the same for $\mathfrak{E}_{N,\varepsilon}\big((f^{(s)}_{N,0})_{1\leq s\leq N},\cdot\,\big)$ and $\mathfrak{E}_0\big((f^{(s)}_0)_{s\geq 1},\cdot\,\big)$), and it has also to be same for all $N$ concerning the BBGKY hierarchy (this comes from the fact that the dependency on $N$ and $\varepsilon$ for $C_3(d,N,\varepsilon)$ in \eqref{SECT2InegaContractioOperaColli} is exactly through the term $N\varepsilon^{d-1}$, as one may see by tracking the constants along the sketch of proof that was presented, which explains why we have to work in the Boltzmann-Grad limit $N\varepsilon^{d-1}=1$).
\end{remar}

The proof of Theorem \ref{SECT2TheorExistUniciSolutHiera}, which is presented in \cite{GSRT}, and of course in \cite{PhDTT} with much details, relies on a crucial contracting inequality due to Nishida \cite{Nish}, Uchiyama \cite{Uchi} and Ukai \cite{Ukai}. A shortened version is also presented at the end of this work, in appendix \ref{AppenSectiOperaColliContractio}: first we adress the problem of the stability of the spaces $\widetilde{\textbf{X}}_{\cdot,\widetilde{\beta},\widetilde{\mu}}$ by the operators of Definition \ref{SECT2DefinOperateursBBGKYBoltz}, and second we study the norm of those operators. In particular, it is possible to show that, up to choose wisely the weights $\widetilde{\beta}$ and $\widetilde{\mu}$, they are contracting mappings. Nevertheless, the choice of those weights comes with a (serious) time restriction on the validity of the existence of the solutions.

\section{The main result: the convergence of the solutions of the BBGKY hierarchy towards the solutions of the Boltzmann hierarchy}

This section will be devoted to prove the main result of this work, which can be presented as follows.

\begin{theor}[Lanford's theorem: convergence of the BBGKY hierarchy towards the Boltzmann hierarchy]
\label{SECT3TheorLanford___Vers_Quali}
Let $\beta_0>0$ and $\mu_0$ be two real numbers. Then there exists a time $T>0$ such that the following holds:\\
let $F_0 = \big(f^{(s)}_0\big)_{s\geq 1}$ be a sequence of initial data of the Boltzmann hierarchy belonging to $\textbf{X}_{0,\beta_0,\mu_0}$, and for any positive integer $N$, let $F_{N,0} = \big(f^{(s)}_{N,0}\big)_{1 \leq s \leq N}$ be a sequence of initial data of the BBGKY hierarchy belonging to $\textbf{X}_{N,\varepsilon,\beta_0,\mu_0}$. We assume that for any $s \in \mathbb{N}^*$, $f^{(s)}_{N,0}$ converges locally uniformly towards $f^{(s)}_0$ on the phase space of $s$ particles, with in addition $\sup_{N \geq 1} \vertii{ F_{N,0} }_{N,\varepsilon,\beta_0,\mu_0} < +\infty$.\\
Then, in the Boltzmann-Grad limit $N \rightarrow +\infty$, $N\varepsilon^{d-1} = 1$, if one denotes $F=\big(f^{(s)}\big)_{s \geq 1}$ the solution on $[0,T]$ of the Boltzmann hierarchy with initial data $F_0$, and $F_N$ the solution on $[0,T]$ of the BBGKY hierarchy with initial data $F_{N,0}$, one has that, for any positive integer $s$, the locally uniform convergence on the domain of local uniform convergence $\Omega_s$ (see Definition \ref{SECT3DefinDomaiConveLocalUnifo} page \pageref{SECT3DefinDomaiConveLocalUnifo} below), uniformly on $[0,T]$, of $f^{(s)}_N$ towards $f^{(s)}$.
\end{theor}

\begin{remar}
This result is the analog of the Lanford's theorem (see \cite{Lanf}) when the particles evolve in the half-space, in its qualitative version. A modern proof of the original theorem, stated for domains without boundary ($\mathbb{R}^d$ or $\mathbb{T}^d$), can be found in \cite{GSRT}. In this reference, the authors were able to perform an important breakthrough by achieving the most detailed proof of Lanford's result, with in addition an explicit rate of convergence.\\
At this step, the presence of the obstacle does not essentially change the statement of \cite{GSRT}. However, we will see along the proof that this obstacle complicates the argument, and although we can refine Theorem \ref{SECT3TheorLanford___Vers_Quali} and provide a quantitative convergence (as in \cite{GSRT}), stated at the very end of this work in Theorem \ref{SECT4TheorPrincipal_deLanford_} page \pageref{SECT4TheorPrincipal_deLanford_}, the rate of convergence is less sharp when there is an obstacle.\\
We note the significant fact that the theorems presented here provide a convergence in a strong sense, implying in particular the one obtained in \cite{GSRT} (which was the convergence in the sense of the observables, that is a uniform convergence in the time and the position variables, but only in the sense of the distributions for the velocity variable). The counterpart is the domain on which this convergence holds: we are forced to consider compact sets in the phase space that are in particular not crossing the wall, nor containing grazing velocities.\\
The convergence between the hierarchies implies two important results: first, for $s=1$, we obtain a rigorous derivation of the Boltzmann equation from finite systems of hard spheres, and second, since the result holds for any integer $s$, when $f^{(s)}$ is tensorized we also recover the propagation of chaos, since the $s$-th marginal of the hard sphere system converges towards a chaotic distribution function.
\end{remar}

\subsection{An explicit formula for the solutions of the hierarchies}

Now that the problem of existence and uniqueness has been addressed, let us see how the solutions of the hierarchies can be rewritten explicitely in terms of the initial data, and iterations of the integrated in time (transport)-collision-transport operator.\\
To describe this result, which can be seen as an analog of the Duhamel formula, we introduces the following notations: the \emph{integrated in time transport-collision-transport operator of the BBGKY hierarchy} $t \mapsto \displaystyle{\int}_{\hspace{-1.75mm}0}^t \hspace{-1mm} \mathcal{T}^{s,\varepsilon}_{-u} \mathcal{C}^{N,\varepsilon}_{s,s+1} \mathcal{T}^{s+1,\varepsilon}_u f^{(s+1)}_{N,0} \dd u$ will be denoted as $t \mapsto \big( \mathcal{I}^{N,\varepsilon}_s f^{(s)}_{N,0} \big) (t,\cdot)$, while the $k$-th iterate of this operator, that is $t \mapsto \displaystyle{\int}_{\hspace{-1.75mm}0}^t \hspace{-1mm} \mathcal{T}^{s,\varepsilon}_{-t_1} \mathcal{C}^{N,\varepsilon}_{s,s+1} \mathcal{T}^{s+1,\varepsilon}_{t_1}\displaystyle{\int}_{\hspace{-1.75mm}0}^{t_1} \hspace{-2mm} \mathcal{T}^{s+1,\varepsilon}_{-t_2} \mathcal{C}^{N,\varepsilon}_{s+1,s+2} \mathcal{T}^{s+2,\varepsilon}_{t_2} \dots$\\ 
\hspace*{35mm}$\displaystyle{\int}_{\hspace{-1.75mm}0}^{t_{k-1}} \hspace{-2mm} \mathcal{T}^{s+k-1,\varepsilon}_{-t_k} \mathcal{C}^{N,\varepsilon}_{s+k-1,s+k} \mathcal{T}^{s+k,\varepsilon}_{t_k} f^{(s+k)}_{N,0} \dd t_k\dots \dd t_2 \dd t_1$, \\
which can be denoted as $t \mapsto \Big( \mathcal{I}^{N,\varepsilon}_s \circ \mathcal{I}^{N,\varepsilon}_{s+1} \circ \dots \circ \mathcal{I}^{N,\varepsilon}_{s+k-1} f^{(s+k)}_{N,0} \Big) (t,\cdot)$, thanks to the new notations, will be in fact denoted as $t \mapsto \big( \mathcal{I}^{N,\varepsilon}_{s,s+k-1} f^{(s+k)}_{N,0} \big) (t,\cdot)$ (the second subscript index describes the number of iterations). Similarly, the integrated in time collision-transport operator of the Boltzmann hierarchy\\
$t \mapsto \displaystyle{\int}_{\hspace{-1.75mm}0}^t \mathcal{T}^{s,0}_{t-u} \mathcal{C}^0_{s,s+1} f^{(s+1)}(u,\cdot) \dd u$ will be denoted as $t \mapsto \big(\mathcal{I}^0_s f^{(s)}\big) (t,\cdot)$.\\
The $k$-th iterate of this operator, that is: $t \mapsto \displaystyle{\int}_{\hspace{-1.75mm}0}^t \mathcal{T}^{s,0}_{t-t_1} \mathcal{C}^{N,\varepsilon}_{s,s+1} \displaystyle{\int}_{\hspace{-1.75mm}0}^{t_1} \mathcal{T}^{s,0}_{t_1-t_2} \mathcal{C}^0_{s+1,s+2} \dots$\\
$\displaystyle{\int}_{\hspace{-1.75mm}0}^{t_{k-1}} \mathcal{T}^{s+k-1,0}_{t_{k-1}-t_k} \mathcal{C}^0_{s+k-1,s+k} f^{(s+k)}(t_k,\cdot) \dd t_k\dots \dd t_2 \dd t_1$ will be denoted as \\$t \mapsto \big( \mathcal{I}^0_{s,s+k-1} f^{(s+k)} \big) (t,\cdot)$.

\begin{propo}[Iterated Duhamel formula for the solution of the hierarchies]
\label{SECT3PropoFormuDuhamelIterée__}
Let $N$ be a positive integer and $\varepsilon$ be a strictly positive number. In the Boltzmann-Grad limit $N\varepsilon^{d-1} = 1$, for any strictly positive number $\beta_0$, any real number $\mu_0$, and for any sequence of initial data $F_{N,0} = \big( f^{(s)}_{N,0} \big)_{1\leq s\leq N}$ belonging to the space $\textbf{X}_{N,\varepsilon,\beta_0,\mu_0}$, the unique solution of the integrated form of the conjugated BBGKY hierarchy with initial datum $F_{N,0}$ on the time interval $[0,T]$ ($T$ given by Theorem \ref{SECT2TheorDefinOperaColliBBGKY}) is given by
\begin{align}
\label{SECT3PROPODuhamelIterée__BBGKY}
H_N = t \mapsto \Big( f^{(s)}_{N,0} + \sum_{k=1}^{N-s} \mathds{1}_{s\leq N-k} \big( \mathcal{I}^{N,\varepsilon}_{s,s+k-1} f^{(s+k)}_{N,0} \big)(t,\cdot) \Big)_{1\leq s\leq N}.
\end{align}
Similarly for any sequence of initial data $F_0 = \big( f^{(s)}_0 \big)_{s\geq 1}$ belonging to the space $\textbf{X}_{0,\beta_0,\mu_0}$, the unique solution of the integrated form of the Boltzmann hierarchy with initial datum $F_0$ on the same time interval $[0,T]$ is given by
\begin{align}
\label{SECT3PROPODuhamelIterée__Boltz}
F = t \mapsto \Big( \mathcal{T}^{s,0}_t f^{(s)}_0(\cdot) + \sum_{k=1}^{+\infty} \Big( \mathcal{I}^0_{s,s+k-1} \big(u \mapsto \mathcal{T}^{s+k,0}_u f^{(s+k)}_0\big) \Big) (t,\cdot) \Big)_{s\geq 1}.
\end{align}
\end{propo}

\begin{proof}[Sketch of proof of Proposition \ref{SECT3PropoFormuDuhamelIterée__}]
The proof of Proposition \ref{SECT3PropoFormuDuhamelIterée__} comes from the fact that the two series in \eqref{SECT3PROPODuhamelIterée__BBGKY} and \eqref{SECT3PROPODuhamelIterée__Boltz} (which are well defined as limits of Cauchy sequences in the two respective Banach spaces $\widetilde{\textbf{X}}_{N,\varepsilon,\widetilde{\beta},\widetilde{\mu}}$ and $\widetilde{\textbf{X}}_{0,\widetilde{\beta},\widetilde{\mu}}$, thanks to the contracting property of the integrated in time collision operators) are fixed points, respectively of the BBGKY and the Boltzmann operators, which concludes the proof of Proposition \ref{SECT3PropoFormuDuhamelIterée__} thanks to the uniqueness of the solutions for the two hierarchies. See \cite{PhDTT} for more details.
\end{proof}

\subsection{A geometrical interpretation of the Duhamel formula}

The Duhamel formula will guide us to a geometrical interpretation, suggesting a somehow natural proof for the convergence. Let us discuss here the formula \eqref{SECT3PROPODuhamelIterée__Boltz}, for the Boltzmann hierarchy, in order to avoid the questions about the meaning of the collision operator for the BBGKY hierarchy.\\
First, the $s$-th marginal is the sum of the terms $\mathcal{I}^0_{s,s+k-1}\big(\mathcal{T}^{s+k,0}_u f^{(s+k)}_0\big)$, for all $k \geq 1$, where $k$ represents the number of iterations of the integrated in time collision-transport operator.\\
Now, each collision operator $\mathcal{C}^0_{s,s+1}$ being a sum of the $2s$ terms $\mathcal{C}^0_{s,s+1,\pm,i}$ (for $1 \leq i \leq s$ and $\pm = +$ or $-$), each term $\mathcal{I}^0_{s,s+k-1}$ can be decomposed as a sum of $2^ks(s+1)\dots(s+k-1)$ terms. Each of those terms corresponds to, first, choosing a first particle $j_1$ among $s$ particles, and a configuration $\pm_1$, being either pre-collisional ($\pm_1 = -$) or post-collisional ($\pm_1 = +$), then choosing a second particle $j_2$ among $s+1$ particles, and a configuration $\pm_2$, and so on, $k$ times. If we write explicitely such a term, for example the one with $k=2$ (two iterations of the integrated in time collision-transport operator), $\pm_1=-$ and $\pm_2=+$, we obtain:
\begin{align*}
- \int_0^t \smashoperator{\bigintssss\limits_{\hspace{14mm}\mathbb{S}^{d-1}_{\omega_1}\times
\mathbb{R}^d_{v_{s+1}}}} \Big[ \omega_1\cdot \big(v_{s+1} - v^{s,0}_{j_1}(t_1)\big) \Big]_- \int_0^{t_1} \smashoperator{\bigintssss\limits_{\hspace{14mm}\mathbb{S}^{d-1}_{\omega_2}\times
\mathbb{R}^d_{v_{s+2}}}} \bigg[ \omega_2 \cdot \big(v_{s+2} - v^{0,j_2}_{s,1}(t_2)\big) \bigg]_+ f_0^{(s+2)}\big( Z^0_{s,2}(0) \big) \dd \omega_2 \dd v_{s+2} \dd t_2 \dd \omega_1 \dd v_{s+1} \dd t_1,
\end{align*}
where
\begin{align*}
v^{0,j_1}_{s,0}(t_1) &= \big(T^{s,0}_{t_1-t}(Z_s)\big)^{V,j_1},\\
v^{0,j_2}_{s,1}(t_2) &= \Big( T^{s+1,0}_{t_2-t_1}\big(
T^{s,0}_{t_1-t}(Z_s),\big( T^{s,0}_{t_1-t}(Z_s) \big)^{X,j_1},v_{s+1}
\big) \Big)^{V,j_2},\\
Z^0_{s,2}(0) &= T^{s+2,0}_{-t_2}\Bigg(
\Bigg(
T^{s+1,0}_{t_2-t_1}\Big(
T^{s,0}_{t_1-t}(Z_s),\big(T^{s,0}_{t_1-t}(Z_s)\big)^{X,j_1},v_{s+1}\Big), \\
&\hspace{50mm}\Big(T^{s+1,0}_{t_2-t_1}\Big(
T^{s,0}_{t_1-t}(Z_s),\big(T^{s,0}_{t_1-t}(Z_s)\big)^{X,j_1},v_{s+1}
\Big)\Big)^{X,j_2},v_{s+2}
\Bigg)'_{j_2,s+2}
\Bigg).
\end{align*}
Considering the expressions from the left to the right, one sees that the iterations of the operators are complicating more and more the arguments below the integrals.\\
Let us investigate the structure of those arguments: in the expression of this term of the solution of the Boltzmann hierarchy at time $t$, the first argument $v^{0,j_1}_{s,0}(t_1)$ corresponds to the velocity of the $j_1$-th particle of the system starting from a configuration $Z_s$, and after following the backward free flow with boundary condition for a time $t-t_1$ (where $t_1$ is the integration variable of the first integrated in time transport-collision operator). Then, a particle is added to this system of $s$ particles at time $t_1-t$, next to the $j_1$-th particle, with $v_{s+1}$ as its initial velocity. Then, this new system with $s+1$ particles follows again the backward free flow for a time $t_1-t_2$. Here we select the velocity of the $j_2$-th particles of this new configuration: this is the argument $v^{0,j_2}_{s,1}(t_2)$. Finally, to obtain the last argument $Z^0_{s,2}(0)$, we restart from the last configuration which was described, picks the $j_2$-th particle and add just next to it (at time $(t-t_1)+(t_1-t_2)$ its position is $\Big(T^{s+1,0}_{t_2-t_1}\Big(
T^{s,0}_{t_1-t}(Z_s),\big(T^{s,0}_{t_1-t}(Z_s)\big)^{X,j_1},v_{s+1}
\Big)\Big)^{X,j_2}$) another particle, with an initial velocity $v_{s+2}$. Since this time $\pm_2$ was chosen to be $+$, the particle which is added starts in a post-collisional configuration, so one applies in addition the scattering operator, and in the end this new system of $s+2$ particles undergoes the action of the backward free transport for a time $t_2$: this is the final argument $Z^0_{s,2}(0)$.\\
\newline
This process, complicated at first glance, can be pictured as in Figure \ref{SECT3FigurConstPseudTraje} below. Such process, mixing transports and adjunctions of particles, builds what is usually called in the literature \emph{pseudo-trajectories} (pseudo because the number of particles of the system changes along time).\\
\begin{figure}[!h]
\centering
\includegraphics[scale=0.4, trim = 3cm 1cm 0cm 0cm]{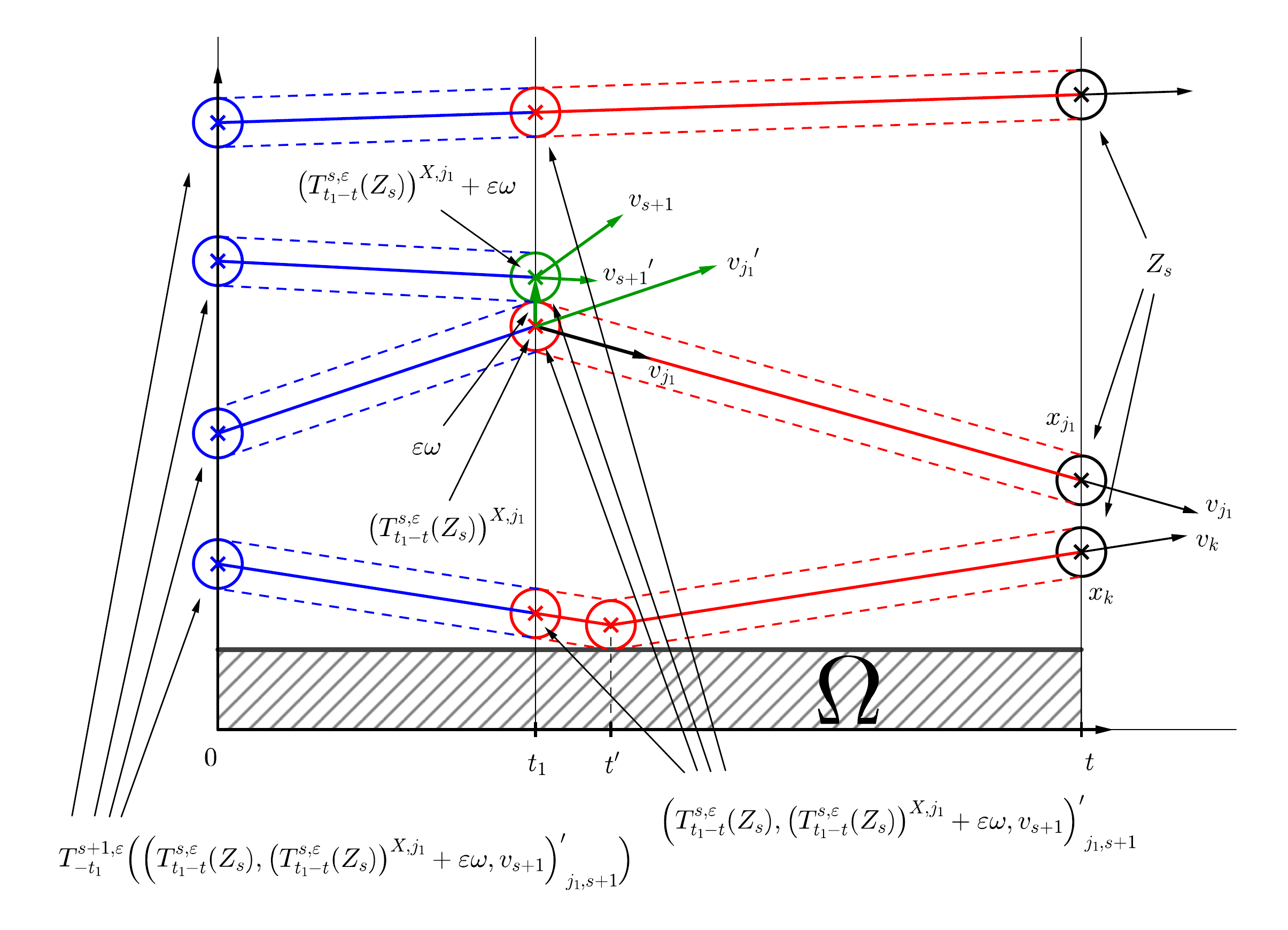}
\caption{Construction of a pseudo-trajectory for the BBGKY hierarchy.}
\label{SECT3FigurConstPseudTraje}
\end{figure}
Formally, the same decomposition, and the same process to describe the elementary terms can be done for the BBGKY hierarchy as well. Of course, there is an important difference: here the particles have a non-zero radius, and the backward free transport is replaced by the backward hard sphere transport, allowing possible interactions between the particles.\\
Let us introduce here some notations. A pseudo-trajectory is entirely determined by its initial configuration $Z_s$, its number of adjunctions $k$, and finally by the choice of its adjunction parameters (the time of adjunction, the particle chosen for it, and the angular parameter and the velocity of the particle added). We introduce then the sets $\mathfrak{T}_k = \big\{(t_1,\dots,t_k) \in [0,t]^k\ /\ t_1 < \dots < t_k \big\}$, $\mathfrak{J}^s_k = \big\{(j_1,\dots,j_k) \in \mathbb{N}^k\ /\ \forall\, i, 1 \leq j_i \leq s+i-1 \big\}$ and $\mathfrak{A}_k = \big\{\big((\omega_1,v_{s+1}),\dots,(\omega_k,v_{s+k})\big)\big\} = \big(\mathbb{S}^{d-1}\times\mathbb{R}^d\big)^k$. Note that the difference between an adjunction in a pre- or post-collisional configuration lies in the sign of $\omega_i\cdot\big(v_{s+i}-v^{\cdot,i-1}_{s,j_i}(t_i)\big)$ (where $v^{\cdot,i-1}_{s,j_i}(t_i)$ is the velocity of the particle $j_i$ undergoing the adjunction, at the time $t_i$ of this adjunction), this sign being $\pm_i$.\\
We call then a \emph{pseudo-trajectory} the collection of all the configurations, along time, of the system starting from $Z_s$, and undergoing the adjunctions described by $T_k \in \mathfrak{T}_k$, $J_k\in\mathfrak{J}^s_k$ and $A_k \in \mathfrak{A}_k$, and denoted $Z^0(Z_s,T_k,J_k,A_k)$ (or $Z^\varepsilon$ for the BBGKY hierarchy), or more simply $Z^0$. To specify that we consider the pseudo-trajectory at a fixed time $\tau$, we denote its configuration $Z^0(\tau)$. A pseudo-trajectory is then a time-dependent function taking its values in the configurations, with an increasing number of particles along time. The position of a generic particle $j_1$ at time $\tau$ of this pseudo-trajectory will be denoted $x^{0,j_1}_{s,i}(\tau)$ (or $x^{\varepsilon,j_1}$ for the BBGKY hierarchy), where the subscripts $s$ and $i$ represent respectively the initial number of particles, and the number of adjunctions performed before $\tau$. For a velocity, we simply replace $x$ by $v$.\\
The main idea of the proof of the convergence can be then described easily: one expects that for particles of small radius, the process described above produces pseudo-trajectories for the BBGKY hierarchy that are uniformly close to the pseudo-trajectories for the Boltzmann hierarchy, as pictured in Figure \ref{SECT3FigurCompaPseudTrajectoir}.
\begin{figure}[!h]
\centering
\includegraphics[scale=0.4, trim = 3cm 2.5cm 0cm 0cm]{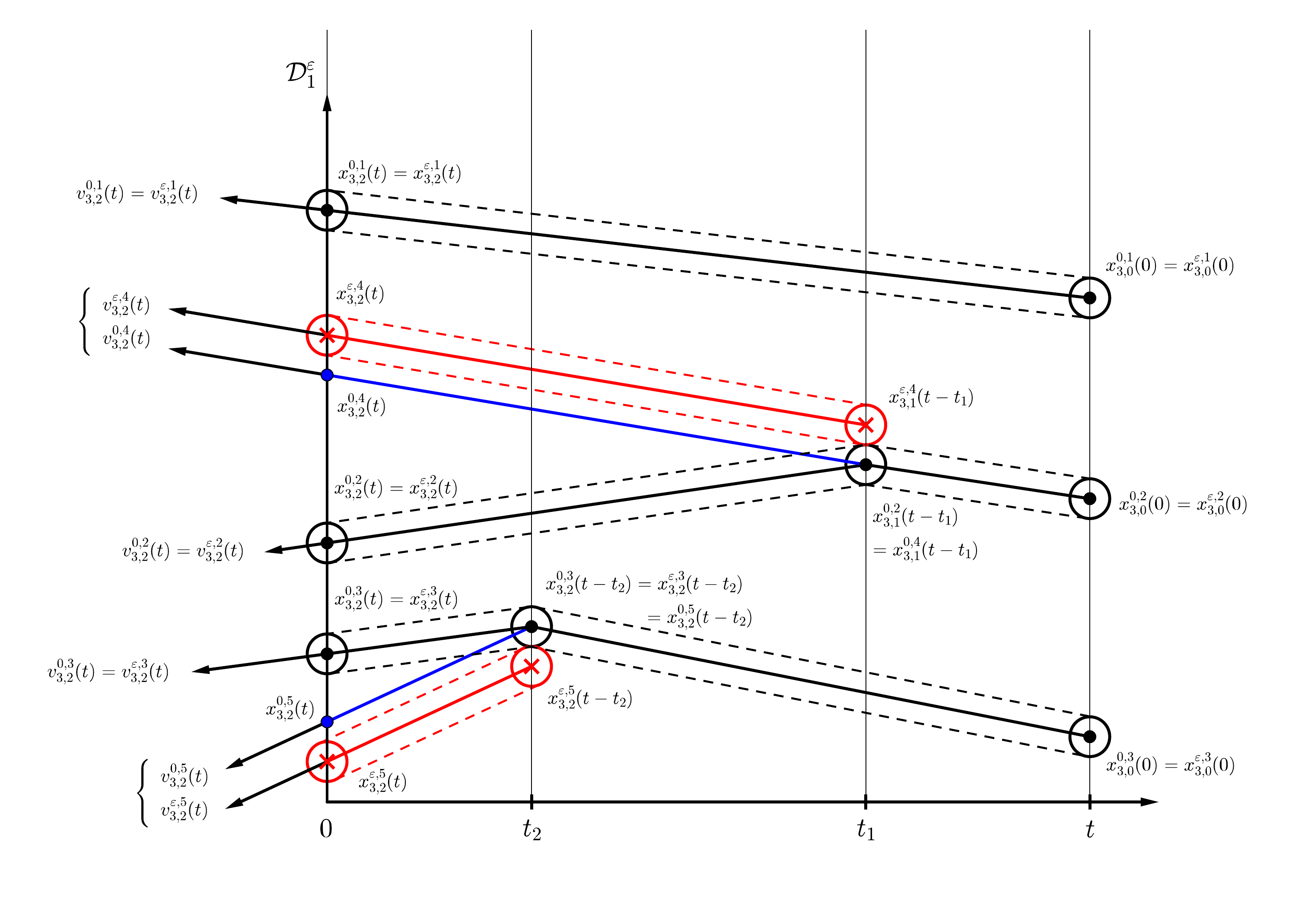}
\caption{Comparison of the pseudo-trajectories of the two hierarchies.}
\label{SECT3FigurCompaPseudTrajectoir}
\end{figure}
Then, if the pseudo-trajectories can be compared, a continuity argument will show that the integrands in the expression of the Duhamel formula will converge, and a dominated convergence argument would conclude the proof.\\
\newline
Nevertheless, the uniform comparison of the pseudo-trajectories is not always possible, for initial configurations of the system may lead to drastically diverging pseudo-trajectories. This is the case when two particles, evolving according to the hard sphere transport, collide one with another. Since this behaviour cannot happen for the corresponding Boltzmann pseudo-trajectory (that is, with the same initial configuration, the same choice of particles for the adjunctions, and the same choice for the pre- or post-collisional settings), a radical difference between the positions and the velocities may suddenly appear, as it is pictured in Figure \ref{SECT3FigurArbrePathoCasdeRecol}. This is a well-known obstruction in Lanford's proof, called \emph{recollision} (see \cite{CeIP}). Those recollisions were studied in much details in \cite{GSRT} in the case of the Euclidean space (without any obstacle in the domain).

\begin{figure}[!h]
\centering
\includegraphics[scale=0.1, trim = 2cm 14.5cm 0cm 3.5cm]{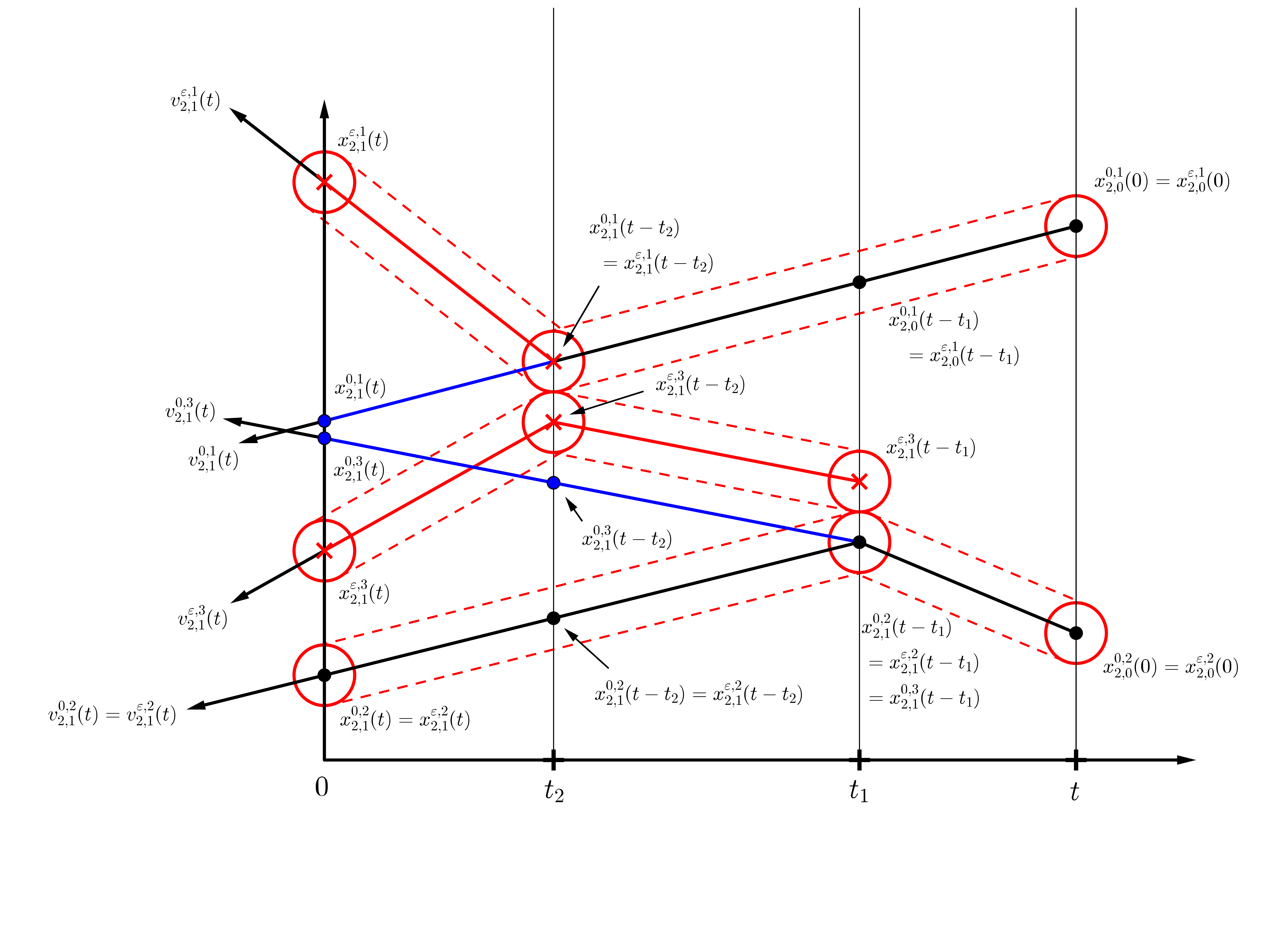}
\caption{Case of a recollision in the dynamics of the hard spheres, and divergence from the dynamics of particles of radius zero, following the free flow.}
\label{SECT3FigurArbrePathoCasdeRecol}
\end{figure}

In the case of our work, the presence of the wall may also produce divergence between the pseudo-trajectories. If a particle hits the wall, starting from a given point, with a given velocity, the point of impact will depend on the radius of the particle, as well as the trajectory of the particle after the bouncing. This variation can be uniformly controlled in terms of the radii of the particles though. A much more serious problem comes from the divergence between the times of bouncing of two particles of different radii: between the two bouncings, one has already bounced and has already a reflected velocity, while the other particle has still its pre-bouncing velocity, leading to an important difference. Those phenomena are pictured in Figure \ref{SECT3FigurDiffeDynamApresRebon} below.

\begin{figure}[!h]
\centering
\includegraphics[scale=0.4, trim = 0.95cm 2cm 0cm 1cm]{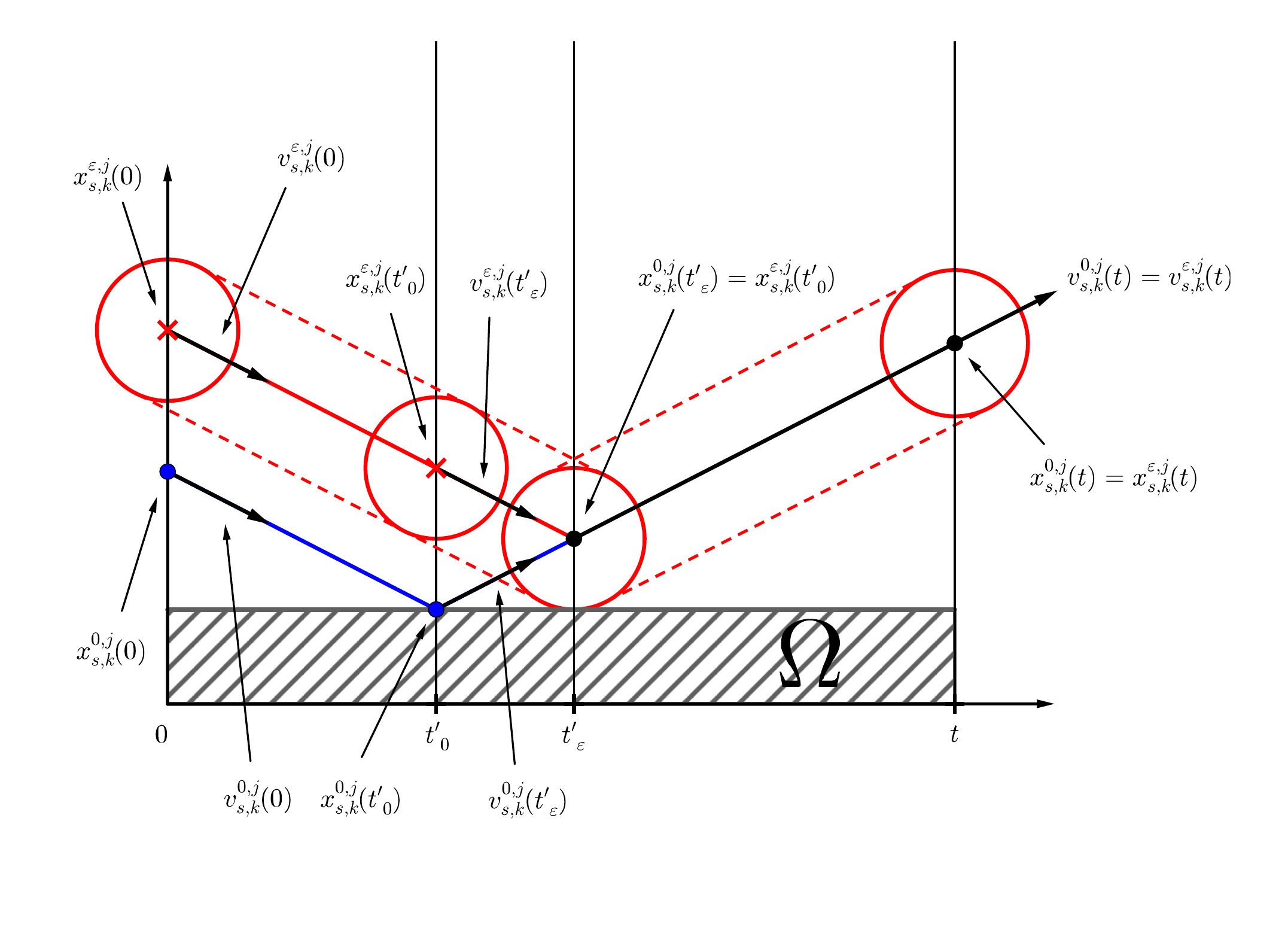}
\caption{The two phenomena of divergence of the pseudo-trajectories appearing during a bouncing against the obstacle.}
\label{SECT3FigurDiffeDynamApresRebon}
\end{figure}

The task will be now to prepare the solutions for those geometrical comparisons: some cut-offs will simplify the proof, while we will have to control the size of the configurations leading to pathological pseudo-trajectories preventing the comparison.

\subsection{Simplification of the pseudo-trajectories thanks to cut-offs}

As in \cite{GSRT}\footnote{See Chapter 7.}, we start with considering only pseudo-trajectories with a finite number of adjunctions, that is obtained with a finite number of iterations of the collision operators. In addition, we will remove the pseudo-trajectories with particles travelling too fast: we will perfom a cut-off in large velocities. Finally, an important tool to simplify the study of the geometry of the pseudo-trajectories will be to prevent the adjunctions to be to close in time: the time difference between two adjunctions will be bounded from below. All those simplifications have a cost, that is described in the three lemmas below.

\begin{lemma}[Cut-off in high number of collisions]
\label{PROPOCutofGrandNombrColli}
Let $\beta_0$ be a strictly positive number and $\mu_0$ be a real number. For any positive integer $n$, any positive integer $N$ and any strictly positive number $\varepsilon > 0$ verifying the Boltzmann-Grad limit $N\varepsilon^{d-1} = 1$, and any couple of sequences of initial data $F_{N,0} = \big( f^{(s)}_{N,0}\big)_{1\leq s\leq N}$ and $F_0 = \big(f^{(s)}_0\big)_{s\geq 1}$ belonging respectively to $\textbf{X}_{N,\varepsilon,\beta_0,\mu_0}$ and $\textbf{X}_{0,\beta_0,\mu_0}$, the respective unique solutions $H_N \in \widetilde{\textbf{X}}_{N,\varepsilon,\widetilde{\beta}_\lambda,\widetilde{\mu}_\lambda}$ to the BBGKY hierarchy with initial datum $F_{N,0}$ and $F \in \widetilde{\textbf{X}}_{0,\widetilde{\beta}_\lambda,\widetilde{\mu}_\lambda}$ to the Boltzmann hierarchy with inital datum $F_0$ (where $\widetilde{\beta}_\lambda$, $\widetilde{\mu}_\lambda$, $H_N$ and $F$ are given by Theorem \ref{SECT2TheorExistUniciSolutHiera}) satisfy
\begin{align}
\label{CUTOFGrandNombrColliBBGKY}
\vertiii{ H_N - \Big( f^{(s)}_{N,0} + \sum_{k=1}^n \mathds{1}_{s\leq N-k} \big( \mathcal{I}^{N,\varepsilon}_{s,s+k-1} f^{(s+k)}_{N,0}\big) \Big)_{1\leq s\leq N}}_{\substack{\vspace{-2.5mm}N,\varepsilon,\widetilde{\beta}_\lambda,\widetilde{\mu}_\lambda}} \hspace{-12.5mm} \leq \Big(\frac{1}{2}\Big)^n \vertii{ F_{N,0} }_{N,\varepsilon,\beta_0,\mu_0},
\end{align}
and
\begin{align}
\label{CUTOFGrandNombrColliBoltz}
\hspace{-3mm}\vertiii{ F - \Big( t\mapsto \mathcal{T}^{s,0}_t f^{(s)}_0 + \sum_{k=1}^n \big( \mathcal{I}^0_{s,s+k-1} \big( u\mapsto \mathcal{T}^{s+k,0}_u f^{(s+k)}_0\big)\big)(t,\cdot)\Big)_{s\geq 1}}_{\substack{\vspace{-2.5mm}0,\widetilde{\beta}_\lambda,\widetilde{\mu}_\lambda}} \hspace{-9mm} \leq \Big(\frac{1}{2}\Big)^n \vertii{ F_0 }_{0,\beta_0,\mu_0}.
\end{align}
\end{lemma}

\begin{proof}[Proof of Lemma \ref{PROPOCutofGrandNombrColli}]
The proof is a simple consequence of the contracting property of the collision operators on the time interval $[0,T]$ given by Theorem \ref{SECT2TheorExistUniciSolutHiera}. The proof is presented in details in \cite{PhDTT}.
\end{proof}
In the following, we will denote:
$$
H_N^n = t \mapsto \Big( f^{(s)}_{N,0}(\cdot) + \sum_{k=1}^n \mathds{1}_{s\leq N-k} \big( \mathcal{I}^{N,\varepsilon}_{s,s+k-1} f^{(s+k)}_{N,0} \big)(t,\cdot) \Big)_{1\leq s\leq N}.
$$
and
$$
F^n =  t \mapsto \Big( \mathcal{T}^{s,0}_t f^{(s)}_0 + \sum_{k=1}^n \big( \mathcal{I}^0_{s,s+k-1}\big( u\mapsto \mathcal{T}^{s+k,0}_u f^{(s+k)}_0\big) \big)(t,\cdot) \Big)_{s\geq 1}.
$$
In addition we introduce for any parameter $R>0$:
\begin{align*}
H_N^{n,R} = t\mapsto \Big( f^{(s)}_{N,0}(\cdot)\mathds{1}_{\vert V_s \vert \leq R} + \sum_{k=1}^n \mathds{1}_{s\leq N-k} \big( \mathcal{I}^{N,\varepsilon}_{s,s+k-1} \big(f^{(s+k)}_{N,0}\mathds{1}_{\vert V_{s+k} \vert \leq R}\big) \big)(t,\cdot)\Big)_{1\leq s\leq N}
\end{align*}
and
\begin{align*}
F^{n,R} = t\mapsto \Big( \mathcal{T}^{s,0}_t f^{(s)}_0(\cdot)\mathds{1}_{\vert V_s \vert \leq R} +\sum_{k=1}^n \big( \mathcal{I}^{N,\varepsilon}_{s,s+k-1} \big( u\mapsto \mathcal{T}^{s+k,0}_u f^{(s+k)}_{N,0}\mathds{1}_{\vert V_{s+k} \vert \leq R}\big) \big)(t,\cdot)\Big)_{s\geq 1}. 
\end{align*}

\begin{lemma}[Cut-off in large energy configurations]
\label{PROPOCutofConfigHauteEnerg}
Let $\beta_0$ be a strictly positive number and $\mu_0$ be a real number. There exists an affine, strictly positive, decreasing function $\widetilde{\beta}' < \widetilde{\beta}$ defined on $[0,T]$ (where $\widetilde{\beta} = \widetilde{\beta}_\lambda$ is given by Theorem \ref{SECT2TheorExistUniciSolutHiera}) and two constants $C_1(d,\beta_0,\mu_0)$ and $C_2(d,\beta_0,\mu_0)$, depending only on the dimension $d$ and on the numbers $\beta_0$ and $\mu_0$, such that for any positive integer $n$, any strictly positive number $R$, any positive integer $N$ and any strictly positive number $\varepsilon > 0$ verifying the Boltzmann-Grad limit $N\varepsilon^{d-1} = 1$, and any pair of sequences of initial data $F_{N,0} = \big( f^{(s)}_{N,0}\big)_{1\leq s\leq N}$ and $F_0 = \big(f^{(s)}_0\big)_{s\geq 1}$ belonging respectively to $\textbf{X}_{N,\varepsilon,\beta_0,\mu_0}$ and $\textbf{X}_{0,\beta_0,\mu_0}$, the truncated in high number of collisions solutions $H_N^n \in \widetilde{\textbf{X}}_{N,\varepsilon,\widetilde{\beta}_\lambda,\widetilde{\mu}_\lambda}$ and $F^n \in \widetilde{\textbf{X}}_{0,\widetilde{\beta}_\lambda,\widetilde{\mu}_\lambda}$ satisfy
\begin{align}
\label{CUTOFHauteEnergColliBBGKY}
\vertiii{ H_N^n - H_N^{n,R}}_{N,\varepsilon,\widetilde{\beta}',\widetilde{\mu}_\lambda}\leq C_1 \exp\big( -C_2 R^2 \big) \vertii{ F_{N,0} }_{N,\varepsilon,\widetilde{\beta}'(0),\mu_0},
\end{align}
and
\begin{align}
\label{CUTOFHauteEnergColliBoltz}
\vertiii{ F^n - F^{n,R}}_{0,\widetilde{\beta}',\widetilde{\mu}_\lambda}\leq C_1 \exp\big( -C_2 R^2 \big) \vertii{ F_0 }_{0,\widetilde{\beta}'(0),\mu_0}.
\end{align}
\end{lemma}

We introduce here the notations for the \emph{integrated in time collision operators, truncated in small time difference between the collisions}, that is we consider
\begin{align*}
&\mathds{1}_{t\geq(k-1)\delta}\int_{(k-1)\delta}^t \mathcal{T}^{s,\varepsilon}_{-t_1} \mathcal{C}^{N,\varepsilon}_{s,s+1} \mathcal{T}^{s+1,\varepsilon}_{t_1}\Big( \mathds{1}_{t_1\geq(k-1)\delta} \int_{(k-2)\delta}^{t_1-\delta} \mathcal{T}^{s+1,\varepsilon}_{-t_2} \mathcal{C}^{N,\varepsilon}_{s+1,s+2} \mathcal{T}^{s+2,\varepsilon}_{t_2}\\
&\hspace{40mm} \dots \Big(\mathds{1}_{t_{k-1}\geq \delta} \int_0^{t_{k-1}-\delta} \mathcal{T}^{s+k-1,\varepsilon}_{-t_k} \mathcal{C}^{N,\varepsilon}_{s+k-1,s+k} \mathcal{T}^{s+k,\varepsilon}_{t_k} f^{(s+k)}_{N,0} (t_k,\cdot) \dd t_k \Big) \dots \Big) \dd t_1,
\end{align*}
so that we have $t_{j-1}-t_j \geq \delta$ for all $2 \leq j \leq k$. It is indeed a cut-off in small differences between the adjunctions, since they are performed at the times $t_j$.\\
Those truncated in small differences between the adjunctions, iterated collision operators will be denoted $\mathcal{I}_{s,s+k-1}^{N,\varepsilon,\delta}$ and $\mathcal{I}_{s,s+k-1}^{0,\delta}$ respectively for the BBGKY hierarchy and the Boltzmann hierarchy. We introduce then
\begin{align*}
H_N^{n,R,\delta} = t\mapsto \Big( &f^{(s)}_{N,0}(\cdot)\mathds{1}_{\vert V_s \vert \leq R} + \sum_{k=1}^n \mathds{1}_{s\leq N-k} \Big( \mathcal{I}^{N,\varepsilon,\delta}_{s,s+k-1} \big(f^{(s+k)}_{N,0}\mathds{1}_{\vert V_{s+k} \vert \leq R}\big) \Big) (t,\cdot)\Big)_{1\leq s\leq N}
\end{align*}
and
\begin{align*}
F^{n,R,\delta} = t\mapsto \Big( &\mathcal{T}^{s,0}_t f^{(s)}_0(\cdot)\mathds{1}_{\vert V_s \vert \leq R} + \sum_{k=1}^n \Big( \mathcal{I}^{N,\varepsilon,\delta}_{s,s+k-1} \big( u\mapsto \mathcal{T}^{s+k,0}_u f^{(s+k)}_{N,0}\mathds{1}_{\vert V_{s+k} \vert \leq R}\big) \Big) (t,\cdot)\Big)_{s\geq 1}. 
\end{align*}

\begin{lemma}[Cut-off in small time difference between the collisions]
\label{PROPOCutofPetitTempsSepa1}
Let $\beta_0$ be a strictly positive number and $\mu_0$ be a real number. There exists a constant $C_3(d,\beta_0,\mu_0)$, depending only on the dimension $d$ and on the numbers $\beta_0$ and $\mu_0$, such that for any positive integer $n$, any strictly positive numbers $R$ and $\delta$, any positive integer $N$ and any strictly positive number $\varepsilon > 0$ verifying the Boltzmann-Grad limit $N\varepsilon^{d-1} = 1$, and any couple of sequences of initial data $F_{N,0} = \big( f^{(s)}_{N,0}\big)_{1\leq s\leq N}$ and $F_0 = \big(f^{(s)}_0\big)_{s\geq 1}$ belonging respectively to $\textbf{X}_{N,\varepsilon,\beta_0,\mu_0}$ and $\textbf{X}_{0,\beta_0,\mu_0}$, the respective truncated in high number of collisions and in large energy solutions $H^{n,R}_N \in \widetilde{\textbf{X}}_{N,\varepsilon,\widetilde{\beta},\widetilde{\mu}}$ and $F \in \widetilde{\textbf{X}}_{0,\widetilde{\beta},\widetilde{\mu}}$ verify, for all integer $1\leq s\leq N$ and time $t\in [0,T]$:
\begin{align}
\label{CUTOFPetitTempsSeparBBGK1}
&\Big\vert \big(H^{n,R}_N\big)^{(s)}(t,\cdot) - \big(H^{n,R,\delta}_N\big)^{(s)}(t,\cdot) \Big\vert_{\varepsilon,s,\widetilde{\beta}(t)} \leq C_3 \vertii{ F_{N,0} }_{N,\varepsilon,\beta_0,\mu_0^1}\sqrt{s} n^{3/2}\sqrt{\delta},
\end{align}
and
\begin{align}
\label{CUTOFPetitTempsSeparBolt1}
&\Big\vert \big(F^{n,R}\big)^{(s)}(t,\cdot) - \big(F^{n,R,\delta}\big)^{(s)}(t,\cdot) \Big\vert_{0,s,\widetilde{\beta}(t)} \leq C_3 \vertii{ F_0 }_{0,\beta_0,\mu_0^1}\sqrt{s} n^{3/2}\sqrt{\delta}.
\end{align}
\end{lemma}

The proofs of Lemmas \ref{PROPOCutofConfigHauteEnerg} and \ref{PROPOCutofPetitTempsSepa1} are not specific to the case of the half-plane, and use only the properties of the collision operators. See \cite{GSRT} or \cite{PhDTT}, where the details are presented.

\subsection{The stability of the good configurations}

We will control here the configurations leading to recollisions. To do so, we will introduce the concept of \emph{good configurations}, that is, the initial data for the system of particles such that all the particles remain at a certain fixed distance one from another, for all time. The goal is to show that, except for a few adjunction parameters, the good configurations are stable under adjunctions, following \cite{GSRT}\footnote{See Proposition 12.1.1 page 94, their key geometrical result allowing to compare the pseudo-trajectories.}. Here the main difference, due of course to the presence of the wall, is the requirement of an additional condition on the particle undergoing the adjunction: it has to be far enough from the obstacle. The reason of this restriction is easy to understand, since if a particle $k+1$ is added next to a particle $k$ which is close to the boundary of the domain, there are many possible velocities to choose for the new particle such that a recollision will happen between the particles $k$ and $k+1$, after a bouncing of the new particle $k+1$ (see Figure \ref{SECT3FigurDiffeVitesRecol/dist} below).

\begin{figure}[!h]
\centering
\includegraphics[scale=0.5, trim = 0cm 9cm 0cm 9cm]{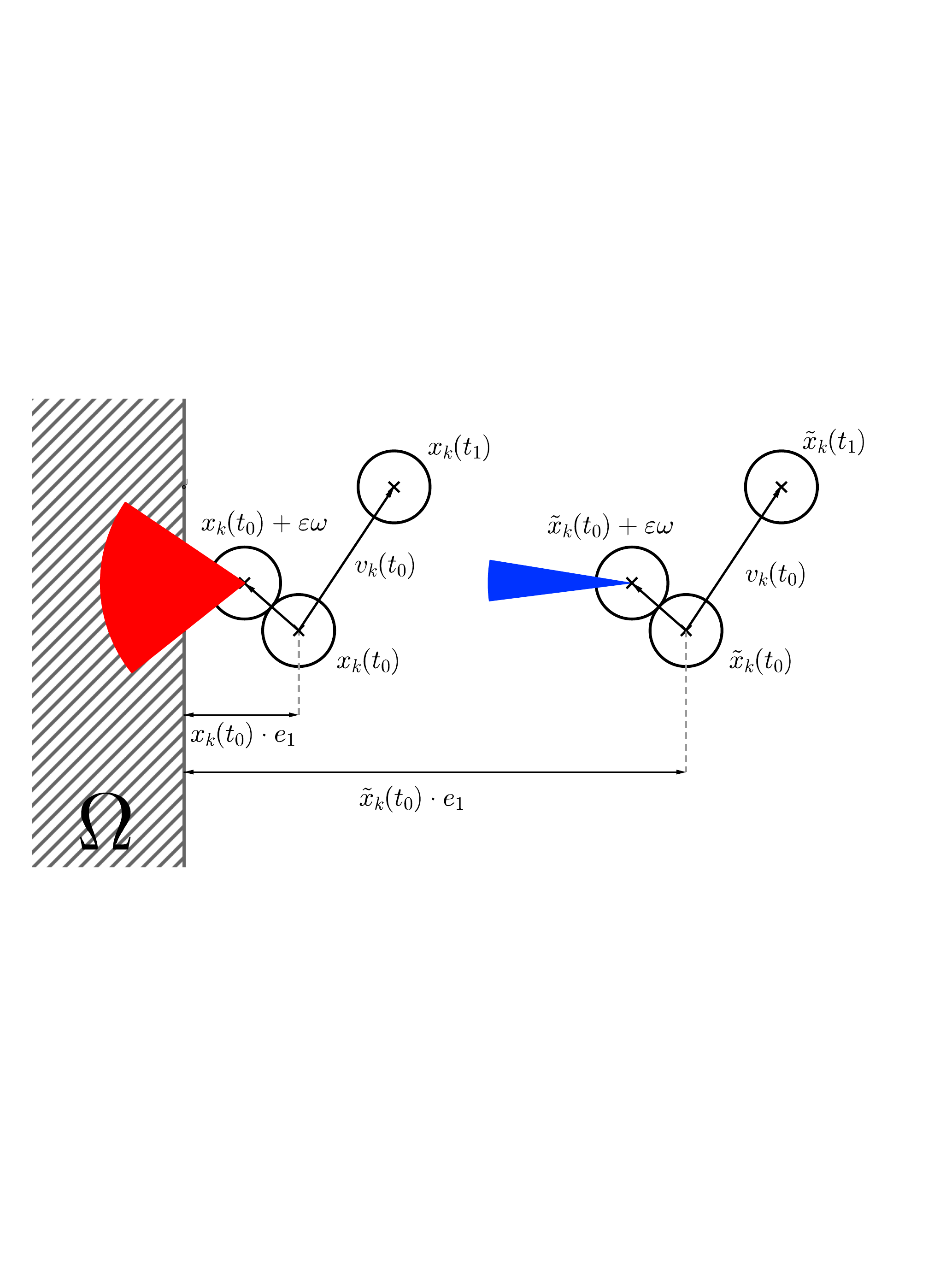}
\caption{The possible directions (in red and blue) of the velocity of the new particle $k+1$ leading to a recollision between $t_0$ and $t_1$ with the particle $k$ undergoing the adjunction (the particle $k$ undergoes the collision at time $t_0$, and travels then until $t_1$ on the figure), depending on the distance between this particle and the wall.}
\label{SECT3FigurDiffeVitesRecol/dist}
\end{figure}

\begin{defin}[Good configuration]
Let $\varepsilon$ and $c$ be two strictly positive numbers and $k$ be a positive integer. One defines the \emph{set of good configurations for $k$ hard spheres separated by at least $c$}, respectively the \emph{set of good configurations for $k$ particles following the free flow separated by at least $c$}, as the subset of $\mathcal{D}^\varepsilon_k$ of the configurations $Z_s$ such that
$$
\Big\vert \left(T^{k,\varepsilon}_{-\tau}(Z_k)\right)^{X,i}-\left(T^{k,\varepsilon}_{-\tau}(Z_k)\right)^{X,j} \Big\vert > c,
$$
respectively as the subset of $\big(\Omega^c\times\mathbb{R}^d\big)^k$ of the configuration $Z_s$ such that
$$
\Big\vert \left(T^{k,0}_{-\tau}(Z_k)\right)^{X,i}-\left(T^{k,0}_{-\tau}(Z_k)\right)^{X,j} \Big\vert > c,
$$
for all $\tau > 0$ and $1 \leq i \neq j \leq k$.\\
Those sets will be denoted respectively $\mathcal{G}^\varepsilon_k(c)$ and $\mathcal{G}^0_k(c)$.
\end{defin}

Let us introduce first an accurate definition of the stability under the adjunctions. We will also use the following notations for the orthogonal symmetries: $\mathcal{S}_0(x)$ will denote $x - 2 (x \cdot e_1) e_1$ and $\mathcal{S}_\varepsilon(x)$ will denote $x - 2(x\cdot e_1)e_1 + \varepsilon e_1$.

\begin{defin}[Stability by adjunction of the good configurations]
\label{SECT3DefinVitesExcluRecolBinai}
Let $k$ be a positive integer, and $R,\delta,\varepsilon$, $a$, $\varepsilon_0$ be five strictly positive numbers. For $\overline{Z}_k\in\mathcal{G}^0_k(\varepsilon_0)$, we define
$$
\mathcal{B}_k(R,\delta,\varepsilon,a,\varepsilon_0)(\overline{Z}_k)\subset \mathbb{S}^{d-1}\times B(0,R)
$$
as the \emph{complement} of the set of the elements $(\omega,v)$ of $\mathbb{S}^{d-1} \times B(0,R)$ such that, for all $Z_k \in \mathcal{G}^\varepsilon_k(\varepsilon)$ with, for all $1\leq i \leq k$, we have $\big\vert \overline{x}_i-x_i \big\vert = \Big\vert \big(\overline{Z}_k\big)^{X,i} - \big(Z_k\big)^{X,i} \Big\vert \leq a$ and $\overline{v}_k = v_k$, that is $\big(\overline{Z}_k\big)^{V,k} = \big(Z_k\big)^{V,k}$, and for all $1\leq i\leq k-1$, $v_i = \overline{v}_i$ or $v_i = \mathcal{S}_0(\overline{v}_i)$, that is $(Z_k)^{V,i} = \big(\overline{Z}_k\big)^{V,i}$ or $(Z_k)^{V,i} = \mathcal{S}_0\big((\overline{Z}_k)^{V,i}\big)$, then
\begin{itemize}[leftmargin=*]
\item if $\omega\cdot(v-\overline{v}_k)<0$:
	\begin{itemize}[leftmargin=0cm]
	\item the configuration $(Z_k,x_k+\varepsilon\omega,v)$ does not lead to a further recollision, that is $\big(Z_k,x_k+\varepsilon\omega,v\big) \in \mathcal{G}^\varepsilon_{k+1}(\varepsilon)$,
	\item the configuration $\big(\overline{Z}_k,\overline{x}_k,v\big)$ is a good configuration (for the free-flow) separated by at least $\varepsilon_0$ after a time $\delta$, that is $T^{k+1,0}_{-\delta}\big(\overline{Z}_k,\overline{x}_k,v\big)\in\mathcal{G}^0_{k+1}(\varepsilon_0)$,
	\end{itemize}
\item if $\omega\cdot(v-\overline{v}_k)>0$:
	\begin{itemize}[leftmargin=0cm]
	\item the configuration $\big(Z_s,x_k+\varepsilon\omega,v\big)'_{k,k+1}$ does not lead to a further recollision, that is $\big(Z_s,x_k+\varepsilon\omega,v\big)'_{k,k+1} \in \mathcal{G}^\varepsilon_{k+1}(\varepsilon)$,
	\item the configuration $\big(\overline{Z}_k,\overline{x}_k,v\big)$ is a good configuration (for the free flow) separated by at least $\varepsilon_0$ after a time $\delta$, that is $T^{k+1,0}_{-\delta}\Big(\big(\overline{Z}_k,\overline{x}_k,v\big)'_{k,k+1}\Big)\in\mathcal{G}^0_{k+1}(\varepsilon_0)$.
	\end{itemize}
\end{itemize}
\end{defin}

\begin{remar}
Here, we wants to find a small upper bound on the size of the set $\big\vert \mathcal{B}_k\big(R,\delta,\varepsilon,a,\varepsilon_0\big)\big(\overline{Z}_k\big) \big\vert$, to show that there are a lot of ways to add a particle to a system of $k$ particles in a good configuration, which lead to a new system of $k+1$ particles in a good configuration.\\
There are two important degrees of freedom introduced in Definition \ref{SECT3DefinVitesExcluRecolBinai}: one authorizes a possible small difference between the positions of the configurations $\overline{Z}_k \in \big( \Omega^c \times \mathbb{R}^d\big)^k$ and the positions of the vector $Z_k \in \mathcal{D}^\varepsilon_k$, and the velocities of those two configurations may differ by a symmetry (that is $v_k=\overline{v}_k$ or $\mathcal{S}_0(\overline{v}_k)$).\\
The necessity of those degrees of freedom is actually clear: differences between the pseudo-trajectories of the BBGKY and the Boltzmann hierarchies will appear, due to the radius of the particles and the interaction with the obstacle (see Figure \ref{SECT3FigurDiffeDynamApresRebon}).
\end{remar}

\begin{theor}[Control of the size of the good configurations by adjunction of particles]
\label{SECT3TheorBonneConfiAdjonParti}
There exists a strictly positive constant $C(d)$ depending only on the dimension $d$ such that for any positive integer $k$, and for all $R,\delta,\varepsilon,a,\varepsilon_0,\rho,\eta$ strictly positive real numbers such that $2 \varepsilon \leq a, 4\sqrt{3} a \leq \varepsilon_0, \varepsilon_0 \leq \eta\delta, 3a \leq \rho$, $R \geq 1$ and $\varepsilon_0/\delta \leq 1$, and for all $\overline{Z}_k\in\mathcal{G}^0_k(\varepsilon_0)$ such that
\begin{align}
\label{PropoBnCfiAdjonMinimDista}
\overline{x}_k\cdot e_1 = \big(\overline{Z}_k\big)^{X,k}\cdot e_1 \geq \rho,
\end{align}
there exists a measurable subset $\widetilde{\mathcal{B}_k}(R,\delta,\varepsilon,a,\varepsilon_0,\rho,\eta)(\overline{Z}_k) \subset {S}^{d-1}\times B(0,R)$ such that :
$$
\mathcal{B}_k(R,\delta,\varepsilon,a,\varepsilon_0)(\overline{Z}_k) \subset \widetilde{\mathcal{B}_k}(R,\delta,\varepsilon,a,\varepsilon_0,\rho,\eta)(\overline{Z}_k)
$$
and
$$
\bigcup_{ \substack{\overline{Z}_k \in\ \mathcal{G}_k(\varepsilon_0), \\ \overline{x}_k\cdot e_1 \geq \rho} } \{ \overline{Z}_k \} \times \widetilde{\mathcal{B}_k}(R,\delta,\varepsilon,a,\varepsilon_0,\rho,\eta)(\overline{Z}_k)
$$
is measurable. Moreover, one has:
\begin{align}
\label{PropoStabiBoCfgTaillPatho}
\Big\vert \widetilde{\mathcal{B}_k}(R,\delta,\varepsilon,a,\varepsilon_0,\rho,\eta)(\overline{Z}_k) \Big\vert \leq \ C(d) \bigg( \eta^d + R^d\Big(\frac{a}{\rho}\Big)^{d-1} + k R^{2d-1}\Big(\frac{a}{\varepsilon_0}\Big)^{d-3/2} + k R^{d+1/2}\Big(\frac{\varepsilon_0}{\delta}\Big)^{d-3/2} \bigg).
\end{align}
\end{theor}

The proof of such a theorem lies on two cornerstones, as in \cite{GSRT}: on the one hand, a ``Shooting Lemma", proving that for a particle $k$ which starts from a ball, with a given velocity, there is only a small amount of velocities, lying in a cylinder with a small radius, for another particle $k+1$ starting from another ball such that the particles $k$ and $k+1$ will collide (or, more generally, will be close at some time). We also have to show that the scattering operator maps small cylinders into small subsets of the adjunction parameters, in the case when the new particle is added in a post-collisional configuration.\\
Let us introduce here the notation for the cylinders. For two vectors $v,w \in \mathbb{R}^d$ and a positive real number $\rho$, we set:
$$
K(v,w,\rho) = \Big\{ x \in \mathbb{R}^d\ /\ \forall u \in \mathbb{S}^{d-1} \text{  such that  } u\cdot w = 0,\, (x-v)\cdot u \leq \rho \Big\}.
$$

\begin{lemma}[Shooting Lemma with fixed axes]
\label{SECT3LemmedeTirPlanAxes_Fixes}
Let $R$, $\delta$, $\varepsilon$, $a$ and $\varepsilon_0$ be five strictly positive numbers, such that $
\varepsilon \leq a, 2\sqrt{3} a \leq \varepsilon_0$. We consider two points $\overline{x}_1,\overline{x}_2\in\big\{x\in\mathbb{R}^d\ /\ x\cdot e_1>0\big\}$ such that $|\overline{x}_1-\overline{x}_2|\geq\varepsilon_0$, and $v_1\in B(0,R)$.\\
Then for all $x_1\in B(\overline{x}_1,a)$, $x_2\in B(\overline{x}_2,a)$, and $v_2\in B(0,R)$:
\begin{enumerate}
\item if for some $\delta > 0$
\begin{align*}
v_2\notin&\ K(v_1,\overline{x}_1-\overline{x}_2,\varepsilon_0/\delta)\cup K(\mathcal{S}_0(v_1),\mathcal{S}_0(\overline{x}_1) - \overline{x}_2,\varepsilon_0/\delta),
\end{align*}
for all $\tau\geq\delta$, we have: $\Big\vert \Big(T^{2,0}_{-\tau}\big(\overline{Z_2}\big)\Big)^{X,1}-\Big(T^{2,0}_{-\tau}
\big(\overline{Z_2}\big)\Big)^{X,2} \Big\vert > \varepsilon_0$,
\item if in addition
\begin{align*}
v_2\notin&\ K(v_1,\overline{x}_1-\overline{x}_2,12Ra/\varepsilon_0)\cup K(\mathcal{S}_0(v_1),\mathcal{S}_0(\overline{x}_1) - \overline{x}_2,16Ra/\varepsilon_0),
\end{align*}
for all $\tau > 0$, we have: $\Big\vert\left(T^{2,\varepsilon}_{-\tau}(Z_2)\right)^{X,1}-\left(T^{2,\varepsilon}_{-\tau}(Z_2)\right)^{X,2}\Big\vert > \varepsilon$.
\end{enumerate}
\end{lemma}

\begin{proof}[Proof of Lemma \ref{SECT3LemmedeTirPlanAxes_Fixes}]
Let us start with the second point, and assume that the two particles (of radius $\varepsilon$) collide, and we consider the smallest time $\tau_0$ such that it happens. Then by definition, the particles follow the free flow, with boundary conditions, before $\tau_0$. From this point, there are two possibilities for the expression of the position of the two particles, depending on wether they have already bounced against the obstacle, or not, before $\tau_0$. If none of the particles have bounced, the situation is already studied in \cite{GSRT}\footnote{See Lemma 12.2.1 page 96.}. The conclusion in that case is that $v_2$ has to lie in the cylinder $K(v_1,\overline{x}_1-\overline{x}_2,12Ra/\varepsilon_0)$, as soon as $\varepsilon_0 \geq 2 \sqrt{3}a$.\\
It remains then three cases to study, specific to the presence of the wall. Before the collision at $\tau_0$:\\
\textbullet\, when only the first particle has bounced against the wall, one has\\
$\big(T^{2,\varepsilon}_{-\tau_0}(Z_2)\big)^{X,1} = \mathcal{S}_\varepsilon(x_1-\tau_0 v_1)$ and $\big(T^{2,\varepsilon}_{-\tau_0}(Z_2)\big)^{X,2} = x_2 - \tau_0 v_2$. We have
$$
\big\vert \big(T^{2,\varepsilon}_{-\tau_0}(Z_2)\big)^{X,1} - \big(T^{2,\varepsilon}_{-\tau_0}(Z_2)\big)^{X,2} \big\vert = \big\vert \big(\mathcal{S}_\varepsilon(x_1) - x_2\big) - \tau_0\big(\mathcal{S}_0(v_1) - v_2\big) \big\vert,
$$
which is equal to $\varepsilon$ by definition of $\tau_0$. We now get rid of the vectors $x_1$ and $x_2$ using the fact that $\mathcal{S}_\varepsilon(u) - \mathcal{S}_\varepsilon(v) = \mathcal{S}_0(u) - \mathcal{S}_0(v)$ and that $\mathcal{S}_0$ is linear, writing:
$$
\mathcal{S}_\varepsilon(x_1) - x_2 = \mathcal{S}_0(x_1-\overline{x}_1) + \mathcal{S}_\varepsilon(\overline{x}_1) - \overline{x}_2 + (\overline{x}_2 - x_2),
$$
which provides, since $\vert x_i - \overline{x}_i \vert \leq a$ (for $i=1,2$) and $\varepsilon \leq a$, that
\begin{equation}
\label{SECT3PREUVLemmeTir__Dista1Rebo}
\big\vert \mathcal{S}_\varepsilon(\overline{x}_1) - \overline{x}_2 - \tau_0\big(\mathcal{S}_0(v_1)-v_2\big) \big\vert \leq 3a.
\end{equation}
Repeating the proof of \cite{GSRT}, we find that $\mathcal{S}_0(v_1)-v_2$ belongs to the cone $\mathcal{C}$ of vertex $0$ (in $\mathbb{R}^d$), based on the ball centered on $\mathcal{S}_\varepsilon(\overline{x}_1)-\overline{x}_2$ and of radius $3a$. Now, since $\big\vert \mathcal{S}_\varepsilon(\overline{x}_1) - \overline{x}_2 \big\vert \geq \vert \overline{x}_1 - \overline{x}_2 \vert \geq \varepsilon_0$ (the first inequality is an easy consequence of the fact that, if we denote $\overline{x}_\theta$ the element of the segment $\big[\overline{x}_1,\mathcal{S}_\varepsilon(\overline{x}_2)\big]$, one has on the one hand $\big\vert \mathcal{S}_\varepsilon(\overline{x}_2) - \overline{x}_\theta \big\vert = \vert \overline{x}_2 - \overline{x}_\theta \vert$, and on the other hand $\big\vert \mathcal{S}_\varepsilon(\overline{x}_2) - \overline{x}_1 \big\vert = \big\vert \mathcal{S}_\varepsilon(\overline{x}_2) - \overline{x}_\theta \big\vert + \vert \overline{x}_\theta - \overline{x}_1 \vert$), if $\varepsilon_0 \geq 2\sqrt{3}a$, the cylinder $K(0,\mathcal{S}_\varepsilon(\overline{x}_1)-\overline{x}_2,12Ra/\varepsilon_0)$ contains the intersection of the cone $\mathcal{C}$ with the ball $B(0,2R)$, so this cylinder contains $v_2 - \mathcal{S}_0(v_1)$.\\
However, this result is not entirely satisfactory, since the axis of the cylinder depends on $\varepsilon$ (which would cause trouble in the preparation of the pseudo-trajectories, see Proposition \ref{SECT3PropoControParamAdjonPatho} below). To eliminate this dependency, we simply write $\mathcal{S}_\varepsilon(\overline{x}_1) = \mathcal{S}_0(\overline{x}_1) + \varepsilon e_1$, so that we replace \eqref{SECT3PREUVLemmeTir__Dista1Rebo} by
$$
\big\vert \mathcal{S}_0(\overline{x}_1) - \overline{x}_2 - \tau_0\big(\mathcal{S}_0(v_1)-v_2\big) \big\vert \leq 3a + \varepsilon \leq 4a,
$$
leading to the conclusion that $v_1 \in K(\mathcal{S}_0(v_1),\mathcal{S}_0(\overline{x}_1)-\overline{x}_2,16Ra/\varepsilon_0)$.\\
\textbullet\, When only the second particle has bounced, we obtain here the condition $\big\vert \big(x_1 - \mathcal{S}_0(x_2)\big) - \tau_0\big(v_1 - \mathcal{S}_0(v_2)\big) \big\vert = \varepsilon$. We use now the fact that $\mathcal{S}_0$ is a linear, involutive isometry, so that the condition can be rewritten as
$\big\vert \mathcal{S}_0\big(x_1 - \mathcal{S}_\varepsilon(x_2)\big) - \tau_0\big(\mathcal{S}_0(v_1) - v_2\big) \big\vert = \varepsilon$. Taking care again of removing the dependency on $\varepsilon$ in $\mathcal{S}_0(x_1-\mathcal{S}_\varepsilon(x_2))$ and replacing $x_i$ by $\overline{x}_i$ (for $i=1,2$), we find here:
$$
\big\vert \mathcal{S}_0(\overline{x}_1 - \mathcal{S}_0(\overline{x}_2)) - \tau_0 \big(\mathcal{S}_0(v_1)-v_2\big) \big\vert \leq 4a
$$
after writing $\mathcal{S}_0\big(\overline{x}_1 - \mathcal{S}_\varepsilon(\overline{x}_2)\big) = \mathcal{S}_0\big(\overline{x}_1 - \mathcal{S}_0(\overline{x}_2) - \varepsilon e_1\big) = \mathcal{S}_0\big(\overline{x}_1 - \mathcal{S}_0(\overline{x}_2)\big) - \varepsilon\mathcal{S}_0(e_1)$. We deduce again that $v_2$ has to lie in $K\big(\mathcal{S}_0(v_1),\mathcal{S}_0(\overline{x}_1-\mathcal{S}_0(\overline{x}_2)),16Ra/\varepsilon_0\big)$.\\
\textbullet\, And finally when the two particles have bounced, we have in that case $\big(T^{2,\varepsilon}_{-\tau_0}(Z_2)\big)^{X,i} = \mathcal{S}_\varepsilon(x_i - \tau_0 v_i)$ for $i=1,2$, so that, using again the identity $\mathcal{S}_\varepsilon(u) - \mathcal{S}_\varepsilon(v) = \mathcal{S}_0(u) - \mathcal{S}_0(v)$, we obtain that
\begin{align*}
\big\vert \big(T^{2,\varepsilon}_{-\tau_0}(Z_2)\big)^{X,1} - \big(T^{2,\varepsilon}_{-\tau_0}(Z_2)\big)^{X,2} \big\vert &= \big\vert \mathcal{S}_0\big((x_1-\tau_0 v_1) - (x_2-\tau_0v_2)\big) \big\vert = \big\vert (x_1-x_2) - \tau_0(v_1-v_2) \big\vert = \varepsilon,
\end{align*}
which turns out to be exactly the condition obtained without any bouncing before $\tau_0$, corresponding to the case studied in \cite{GSRT}.\\
For the first point of the lemma, we notice that it is the free flow which is involved, so there is no concern about collisions between the particles to have. Let us start to say that the same four sub-cases as for the previous point have to be considered, and the one without bouncing is already adressed in \cite{GSRT}. Let us present the proof when only the first particle has bounced against the wall before $\tau_0$, a time such that:
$$
\big\vert \big(T^{2,0}_{-\tau_0}(\overline{Z}_2)\big)^{X,1} - \big(T^{2,0}_{-\tau_0}(\overline{Z}_2)\big)^{X,1} \big\vert \leq \varepsilon_0.
$$
In that case, the condition writes explicitely $\big\vert \mathcal{S}_0(\overline{x}_1) - \overline{x}_2 - \tau_0\big(\mathcal{S}_0(v_1) - v_2\big) \big\vert \leq \varepsilon_0$. Then, for all unit vector $n$ orthogonal to $\mathcal{S}_0(\overline{x}_1) - \overline{x}_2$, we have thanks to the Cauchy-Schwarz inequality that $\tau_0 \big\vert n \cdot \big(\mathcal{S}_0(v_1)-v_2\big) \big\vert \leq \varepsilon_0$, which means that $\mathcal{S}_0(v_1)-v_2$ belongs to $K(0,\mathcal{S}_0(\overline{x}_1) - \overline{x}_2,\varepsilon_0/\delta)$, hence the conclusion. The other cases are obtained in the same way. This concludes the proof of Lemma \ref{SECT3LemmedeTirPlanAxes_Fixes}.
\end{proof}

The second lemma studies the effect of the scattering of the cylinders, already stated and proved in \cite{GSRT}:

\begin{lemma}[Scattering Lemma for cylinders]
\label{SECT3LemmeScateCylinVitesPatho}
There exists a strictly positive constant $C_5(d)$ depending only on the dimension $d$ such that for all strictly positive numbers $\rho$ and $R$, all vectors $(y,w)\in\mathbb{R}^d\times B(0,R)$, and $v_1\in B(0,R)$, one has :
$$
\left|\mathcal{N}^*(w,y,\rho)(v_1)\right|\leq C_5R^{d+1/2}\rho^{d-3/2}
$$
with
\begin{align*}
&\mathcal{N}^*(w,y,\rho)(v_1)=\Big\{(\omega,v_2)\in\mathbb{S}^{d-1}\times B(0,R)\ /\ (v_2-v_1)\cdot\omega>0,\, v_1'\in K(w,y,\rho)\text{  or  }v_2'\in K(w,y,\rho)\Big\}.
\end{align*}
\end{lemma}

\begin{proof}[Proof of Theorem \ref{SECT3TheorBonneConfiAdjonParti}]
There are two cases that have to be considered separately: whether the particles $k$ and $k+1$ are in a pre-collisional configuration or not.\\
\newline
\textbullet\ Let us start with the pre-collisional case. It implies that the velocities of the pair of particles $(k,k+1)$ are not modified by the scattering.\\
Considering now the recollisions (that is, if $Z^\varepsilon_{k+1}$ belongs to $\mathcal{G}^\varepsilon_{k+1}(\varepsilon)$ or not, with $Z^\varepsilon_{k+1}=(Z_k,x_k+\varepsilon\omega,v)$), they cannot happen between the particles $i$ and $j$ with $1 \leq i < j \leq k$ by hypothesis. If now $i < k$ and $j = k+1$, a simple application of the Shooting Lemma \ref{SECT3LemmedeTirPlanAxes_Fixes} provides $4(k-1)$ cylinders (for the velocity $v$) of respective volume $C(d)(4R)\big(24Ra/\varepsilon_0\big)^{d-1}$ to exclude in order to prevent recollision. The most interesting case is then when $i=k$ and $j=k+1$. In that case, except if the velocities of the two particles are the same, no recollision can occur if none or both of the two particles have bounced against the wall. If only one of the two particles has already bounced against the wall at $\tau_0$, a time such that $\big\vert \big(T^{k+1,\varepsilon}_{-\tau_0}(Z^\varepsilon_{k+1})\big)^{X,k} - \big(T^{k+1,\varepsilon}_{-\tau_0}(Z^\varepsilon_{k+1})\big)^{X,k+1} \big\vert = \varepsilon$, this condition can be rewritten as
\begin{equation}
\label{SECT3PREUVTheorAdjonRecolkk+1e}
\big\vert -2x_k\cdot e_1 e_1 - \varepsilon \omega + \varepsilon e_1 - \tau_0\big(\mathcal{S}_0(\overline{v}_k)-v\big) \big\vert = \varepsilon.
\end{equation}
To obtain a control on $v - \mathcal{S}_0(\overline{v}_k)$ which does not depend on the positions of the particles nor on the angular parameter $\omega$, we write:
\begin{align*}
\varepsilon &\geq \big\vert 2\overline{x}_k\cdot e_1 e_1 - \tau_0\big(v - \mathcal{S}_0(\overline{v}_k)\big) \big\vert - \big\vert 2 x_k \cdot e_1 e_1 - 2 \overline{x}_k \cdot e_1 e_1 \big\vert - \varepsilon \vert \omega \vert - \varepsilon \vert e_1 \vert \\
&\geq \big\vert 2\overline{x}_k\cdot e_1 e_1 - \tau_0\big(v - \mathcal{S}_0(\overline{v}_k)\big) \big\vert - 2a - 2\varepsilon,
\end{align*}
so that $\big\vert 2\overline{x}_k\cdot e_1 e_1 - \tau_0\big(v - \mathcal{S}_0(\overline{v}_k)\big) \big\vert \leq 5a$. Following the same proof as for Lemma \ref{SECT3LemmedeTirPlanAxes_Fixes}, we deduce that $v - \mathcal{S}_0(\overline{v}_k)$ belongs to the cone of vertex $0$ and based on the ball $B(2\overline{x}_k\cdot e_1 e_1,5a)$. We see here why we need to assume that the particle $k$ is far enough from the boundary: if $\big\vert 2\overline{x}_k\cdot e_1 e_1 \big\vert \leq 5a$, the condition on the cone is empty, since we would describe the whole space with a cone centered on $0$ and based on a ball which contains $0$. So here, since $\rho \leq \overline{x}_k\cdot e_1$ and $5a/\sqrt{3} \leq \rho$, we can deduce that the condition \eqref{SECT3PREUVTheorAdjonRecolkk+1e} implies that $v \in K\big(\mathcal{S}_0(\overline{v}_k),2\overline{x}_k\cdot e_1 e_1,10Ra/\rho\big)$. The measure of this cylinder, intersected with $B(0,R)$, is then controlled by $C(d)(4R)\big(10Ra/\rho\big)^{d-1}$.\\
Let us consider now the problem of the good configurations for the free flow (that is, if $T^{k+1,0}_{-\delta}\big(\overline{Z}^0_{k+1}\big)$ belongs to $\mathcal{G}^0_{k+1}(\varepsilon_0)$ or not, with $\overline{Z}^0_{k+1} = (\overline{Z}_k,\overline{x}_k,v)$). Again, if the two particles $i$ and $j$ are such that $1 \leq i<j \leq k$, by hypothesis the distance between those particles will always be larger than $\varepsilon_0$. If $i<k$ and $j=k+1$, we can again use the Shooting Lemma \ref{SECT3LemmedeTirPlanAxes_Fixes}, so that up to exclude $2(k-1)$ cylinders (for the velocity $v$) of respective volume $C(d)R\big(\varepsilon_0/\delta\big)^{d-1}$, the pair of particles $(i,k+1)$ will stay at a distance larger than $\varepsilon_0$ after a time $\delta$. Now if $i=k$ and $j=k+1$, the condition
\begin{equation}
\label{SECT3PREUVTheorAdjonRecolkk+10}
\big\vert \big(T^{k+1,0}_{-\tau_0}(\overline{Z}^0_{k+1})\big)^{X,k} - \big(T^{k+1,0}_{-\tau_0}(\overline{Z}^0_{k+1}\big)^{X,k+1} \big\vert \leq \varepsilon_0
\end{equation}
may happen either if only one of the two particles has bounced against the wall at time $\tau_0$, or if none or both have. In the last case, \eqref{SECT3PREUVTheorAdjonRecolkk+10} writes $\big\vert (\overline{x}_k-\tau_0\overline{v}_k) - (\overline{x}_k-\tau_0v)\big\vert = \tau_0\vert v - \overline{v}_k \vert \leq \varepsilon_0$. If $v \notin B(\overline{v}_k,\eta)$ with $\varepsilon_0 \leq \delta\eta$, this can never happen. In the first case, \eqref{SECT3PREUVTheorAdjonRecolkk+10} writes $\big\vert \tau_0\big(v - \mathcal{S}_0(\overline{v}_k)\big) - \big(\mathcal{S}_0(\overline{x}_k) - \overline{x}_k\big) \big\vert \leq \varepsilon_0$. We deduce then that for all unitary vector $u$ orthogonal to $e_1$ we have $\tau_0 \cdot \big\vert u\cdot\big(v - \mathcal{S}_0(\overline{v}_k)\big) \big\vert \leq \varepsilon_0$. Then, for $\tau_0 \geq \delta$ and $v \notin K\big(\mathcal{S}_0(\overline{v}_k),e_1,\varepsilon_0/\delta\big)$, the condition \eqref{SECT3PREUVTheorAdjonRecolkk+10} cannot hold if only one of the particles has bounced. In both cases, up to exclude a subset of velocities of size $C(d)\big(\eta^d + R\big(\varepsilon_0/\delta\big)^{d-1}\big)$, \eqref{SECT3PREUVTheorAdjonRecolkk+10} cannot hold.\\
\textbullet\ It remains the post-collisional case to be investigated. Let us start with the recollisions for the hard sphere dynamics. First, since the scattering does not modify the velocities of the particles $i < k$, the first recollision obtained from $Z^{\varepsilon,'}_{k+1}$ (with $Z^{\varepsilon,'}_{k+1} = (\dots,x_k,\overline{v}_k\hspace{-1.5mm}',x_k+\varepsilon\omega,v')$) cannot be between the particles $i$ and $j$ with $1 \leq i < j \leq k-1$ (by hypothesis on the pre-collisional configuration $Z^\varepsilon_{k+1}$). In other words, if we get rid of the recollisions with the other pairs, it would imply that we will get rid also of the recollisions for the pairs $1 \leq i < j \leq k-1$. If now $i \leq k-1$ and $j=k$, thanks to Lemma \ref{SECT3LemmedeTirPlanAxes_Fixes}, up to exclude $4(k-1)$ cylinders of radius $12Ra/\varepsilon_0$ for the velocity $\overline{v}_k\hspace{-1.5mm}'$, we would make sure of the absence of recollision for those pairs. Similarly, for $1 \leq i \leq k-1$ and $j=k+1$, by excluding the same $4(k-1)$ cylinders as in the pre-collisional case (for $v'$ this time, and not $v$), we would eliminate the recollisions for that case. However the conditions described here concern the post-collisonal velocities  $\overline{v}_k\hspace{-1.5mm}'$ and $v'$, not the adjunction parameters $v$ and $\omega$. To convert those conditions and obtain a control on those parameters, we will simply use Lemma \ref{SECT3LemmeScateCylinVitesPatho}. In summary, this lemma applied to the excluded cylinders together provides a set of adjunction parameters $(\omega,v)$ of size $C(d) (k-1) R^{2d-1} \big(a/\varepsilon_0\big)^{d-3/2}$ to exclude. It remains only the most delicate case to consider, when $i=k$ and $j=k+1$. Here one cannot repeat the argument of the pre-collisional case, for both post-collisional velocities $\overline{v}_k\hspace{-1.5mm}'$ and $v'$ depend on $v$ (which prevent to define cylinders depending only on $\overline{Z}_k$), and also on $\omega$, which is another source of difficulty. Now for that pair, since the configuration $(x_k,\overline{v}_k\hspace{-1.5mm}',x_k+\varepsilon\omega,v')$ is pre-collisional by definition of the scattering, no recollision can happen between those particles if none or both have bounced against the wall. In the case when only one particle has already bounced at time $\tau_0$, the condition of the recollision between the particles $k$ and $k+1$ writes $\big\vert 2x_k\cdot e_1 e_1 + \varepsilon\omega - \varepsilon e_1 - \tau_0\big(v' - \mathcal{S}_0(\overline{v}_k\hspace{-1.5mm}')\big) \big\vert = \varepsilon$. The naive consideration of the cylinder, as for the pre-collisional case, would provide a condition on $v' - \mathcal{S}_0(\overline{v}_k\hspace{-1.5mm}')$, which prevents to conclude directly. But writing $v' - \mathcal{S}_0(\overline{v}_k\hspace{-1.5mm}') = v' - \overline{v}_k\hspace{-1.5mm}' - \big(\mathcal{S}_0(\overline{v}_k\hspace{-1.5mm}') - \overline{v}_k\hspace{-1.5mm}'\big)$, and using the definition of the post-collisional velocities, we have $v' - \overline{v}_k\hspace{-1.5mm}' = \big(v - \overline{v}_k\big) - 2\big(v - \overline{v}_k\big) \cdot \omega \omega$, showing that $v' - \mathcal{S}_0(\overline{v}_k\hspace{-1.5mm}')$ belongs to $K(0,e_1,10Ra/\rho)$ if and only if $\big(v - \overline{v}_k\big) - 2\big(v - \overline{v}_k\big) \cdot \omega \omega$ belongs to the same cylinder. This is a condition on $v$ (instead of $v'$), but depending also on $\omega$. This condition is equivalent to that $v - \overline{v}_k$ belongs to $K(0,e_1 - 2e_1\cdot\omega \omega,10Ra/\rho)$. As a conclusion, up to exclude among the adjunction parameters $(\omega,v)$ the subset
$$
\big( \hspace{-5mm} \bigcup_{\hspace{5mm} \omega \, \in \, \mathbb{S}^{d-1}} \hspace{-5mm} \big\{\omega\big\} \times K(0,e_1 - 2e_1\cdot\omega \omega,10Ra/\rho) \hspace{0.5mm} \big),
$$
which has a size controlled by $C(d)R\big(10Ra/\rho\big)^{d-1}$, we can claim that there will be no recollision between the particles $k$ and $k+1$, which concludes the study of the recollisions in the post-collisional case.\\
Let us finally consider the problem of the good configurations for the free flow in the post-collisional case. First, since only the velocities of the two last particles are modified with the scattering, no pair $(i,j)$ of particles, with $1 \leq i < j < k$, can be closer than $\varepsilon_0$ by hypothesis on $\overline{Z}_k$. Concerning the pairs $(i,k)$ and $(i,k+1)$ for $i < k$, thanks to the Shooting Lemma \ref{SECT3LemmedeTirPlanAxes_Fixes}, up to exclude $2(k-1)$ cylinders of radius $\varepsilon_0/\delta$ among the velocities $\overline{v}_k\hspace{-1.5mm}'$ and $v'$, we can be sure that the particles will stay at a distance larger than $\varepsilon_0$ one from another. The Scattering Lemma \ref{SECT3LemmeScateCylinVitesPatho} enables to translate this condition on $\omega$ and $v$: it means that a subset of size $C(d)(k-1)R^{d+1/2}\big(\varepsilon_0/\delta\big)^{d-3/2}$ has to be excluded to keep the distance larger than $\varepsilon_0$ between those particles. Finally, for the particles $k$ and $k+1$, we have $\big\vert \big(T^{k+1,0}_{-\tau_0}(\overline{Z}^{0,'}_{k+1})\big)^{X,k} - \big(T^{k+1,0}_{-\tau_0}(\overline{Z}^{0,'}_{k+1})\big)^{X,k+1} \big\vert \leq \varepsilon_0$ (with $\overline{Z}^{0,'}_{k+1} = (\dots,\overline{x}_k,\overline{v}_k\hspace{-1.5mm}',\overline{x}_k,v')$) if and only if $\tau_0 \vert v' - \overline{v}_k\hspace{-1.5mm}' \vert \leq \varepsilon_0$ or $\big\vert \tau_0\big(v' - \mathcal{S}_0(\overline{v}_k\hspace{-1.5mm}')\big) - \big(\mathcal{S}_0(\overline{x}_k)-\overline{x}_k\big) \big\vert \leq \varepsilon_0$. If $\vert v' - \overline{v}_k\hspace{-1.5mm}' \vert > \eta$ with $\varepsilon_0 \leq \delta \eta$, that is, thanks to the conservation of the kinetic energy along the collisions implying $\vert v' - \overline{v}_k\hspace{-1.5mm}' \vert = \vert v - \overline{v}_k \vert$, if $\vert v - \overline{v}_k \vert \geq \eta$, the first condition cannot hold. The last condition $\big\vert \tau_0\big(v' - \mathcal{S}_0(\overline{v}_k\hspace{-1.5mm}')\big) - \big(\mathcal{S}_0(\overline{x}_k)-\overline{x}_k\big) \big\vert \leq \varepsilon_0$ is finally studied in the same fashion as for the post-collisional case of the hard sphere flow: up to exclude among the adjunction parameters $(\omega,v)$ the subset
$$
\big( \hspace{-5mm} \bigcup_{\hspace{5mm} \omega \, \in \, \mathbb{S}^{d-1}} \hspace{-5mm} \big\{\omega\big\} \times K(0,e_1 - 2e_1\cdot\omega \omega,\varepsilon_0/\delta) \hspace{0.5mm} \big),
$$
of size $C(d)R\big(\varepsilon_0/\delta\big)^{d-1}$, the second condition cannot hold.\\
\newline
The subset $\mathcal{B}_k(\overline{Z}_k)$ of Theorem \ref{SECT3TheorBonneConfiAdjonParti} is now simply the collection of all the excluded subsets of the adjunction parameters $(\omega,v) \in \mathbb{S}^{d-1} \times \mathbb{R}$ described above: on its complement the adjunction parameters provide new configurations of $k+1$ particles that are in a good configuration, and the size of $\mathcal{B}_k(\overline{Z}_k)$ is controlled by the sum of the size of the previous excluded subsets, which concludes the proof of Theorem \ref{SECT3TheorBonneConfiAdjonParti}.
\end{proof}

\subsection{The inclusion of the cut-off in proximity with the obstacle}

In Theorem \ref{SECT3TheorBonneConfiAdjonParti}, it was important that the particle undergoing the adjunction is far from the wall. Sadly, there is no direct way to fulfill this condition, since the positions are not integration variables in the integrated in time collision operators, that is this condition cannot be obtained with a naive cut-off.\\
However, there is a way to use the adjunction parameters to reach that goal. If a particle has not a grazing velocity, it will not stay close to the wall for a long time. We will then remove the pseudo-trajectories with grazing collisions, and remove also the time intervals during which the particle chosen to undergo the adjunction is too close to the wall. To do so, an important (and quite technical) step is to control the effect of the scattering mapping on the grazing collisions.

\begin{defin}[Adjunction parameters inducing grazing collisions after scattering]
\label{DEFINParamAdjonColliRasan}
For any $v_1 \in B(0,R)$, we will call the subset of the \emph{adjunction parameters inducing grazing collisions after adding a particle to another one with velocity $v_1$ and after scattering}, and we will denote $\mathcal{N}^*(R,\alpha)(v_1) \subset \mathbb{S}^{d-1}\times B(0,R)$ for the set defined by:
\begin{align*}
\mathcal{N}^*(R,\alpha)(v_1)=&\ \bigg\{(\omega,v_2)\in\mathbb{S}^{d-1}\times B(0,R)\ /\ (v_2-v_1)\cdot\omega>0,\, v_1'\in\{|v\cdot e_1|\leq\alpha\}\text{  or  }v_2'\in\{|v\cdot e_1|\leq\alpha\}\bigg\}.
\end{align*}
\end{defin}

It is possible then to obtain the following result.

\begin{lemma}
\label{SECT3LemmeScatePavé_VitesRasan}
There exist two strictly positive constants $C(d)$ and $c(d)$ depending only on the dimension $d$ such that for all strictly positive numbers $\alpha \leq c(d)$ and $R \geq 1$, and all $v\in B(0,R)$, one has:
\begin{align}
\label{LemmeScatePavé_Vitesd=2__}
\Big\vert \mathcal{N}^*(R,\alpha)(v) \Big\vert \leq C(d) R^2\alpha^{1/8}
\end{align}
in the case $d=2$, and
\begin{align}
\label{LemmeScatePavé_Vitesdgeq3}
\Big\vert \mathcal{N}^*(R,\alpha)(v) \Big\vert \leq C(d)R^{d-1}\alpha
\end{align}
in the case $d\geq3$.
\end{lemma}

\begin{proof}[Proof of Lemma \ref{SECT3LemmeScatePavé_VitesRasan}]
First, one needs an intermediate result about the measure of some subspace of a sphere, which follows.\\
Let $\alpha$, $r$ and $R$ be three strictly positive numbers. In the case $d=2$, if $R \geq 1$, $\alpha \leq \min \Big\{1,R^2/36,\big(\frac{1+\sqrt{2}-\sqrt{3}}{\sqrt{3}/2+\sqrt{2}}\big)^2\Big\}$,  $2\sqrt{\alpha}\leq r \leq 2R$, and if $\big\vert x\cdot e_1 \big\vert \leq r - \sqrt{\alpha}$, then
\begin{align}
\label{SECT3PREUVSpherEntr2Plansdim=2}
\Big\vert \big\{ y \in \mathbb{R}^2\ /\ \big\vert x - y \big\vert = r, \big\vert y \cdot e_1 \big\vert \leq \alpha \big\} \Big\vert \leq \sqrt{R}\alpha^{1/4},
\end{align}
and in the case $d \geq 3$, there exists a constant $C(d)$ depending only on the dimension $d$ such that if $r \leq 2R$, then
\begin{align}
\label{SECT3PREUVSpherEntr2Plansdim>2}
\Big\vert \big\{ y \in \mathbb{R}^2\ /\ \big\vert x - y \big\vert = r, \big\vert y \cdot e_1 \big\vert \leq \alpha \big\} \Big\vert \leq C(d)r^{d-2}\alpha.
\end{align}
This intermediate result is obtained after studying the three possible subcases concerning the position of the point $x$ with respect to the hyperplane $\{ x \cdot e_1 = 0 \}$. For this purpose, let us introduce the quantity $p = x\cdot e_1$. Without loss of generality, let us assume that $p\geq 0$.\\
\newline
\textbullet\ If $r < p - \alpha$ (the trivial case when the sphere does not cross $\{-\alpha \leq y \cdot e_1 \leq \alpha\}$), then here of course $\Big\vert \big\{ y \in \mathbb{R}^2\ /\ \big\vert x - y \big\vert = r, \big\vert y \cdot e_1 \big\vert \leq \alpha \big\} \Big\vert = 0$.\\
\newline
\textbullet\ If $r \geq p - \alpha$, and ($r \leq p + \alpha$ or $r \leq \alpha$) (the case when the sphere crosses only a single plane delimiting $\{-\alpha \leq y \cdot e_1 \leq \alpha\}$), then either $r \leq \alpha$, and then here
$\{ y \in \mathbb{R}^2\ /\ \vert x - y \vert = r, \vert y \cdot e_1 \vert \leq \alpha \} \subset \partial B(x,r)$, so that:\\
$\Big\vert \big\{ y \in \mathbb{R}^2\ /\ \big\vert x - y \big\vert = r, \big\vert y \cdot e_1 \big\vert \leq \alpha \big\} \Big\vert \leq C(d)r^{d-1}$, or $r \geq p - \alpha$ and $r < p+\alpha$, and in that case only one of the two apices (along the diameter parallel to $e_1$) of the sphere is strictly in between the two planes delimiting $\{-\alpha \leq y \cdot e_1 \leq \alpha\}$. It is clear that in this particular case, the surface of the sphere is maximized when the apex is tangent to one of the two planes, so that one has here:\\
$\Big\vert \big\{ y \in \mathbb{R}^d\ /\ \vert y-x \vert \leq r, \vert y \cdot e_1 \vert \leq \alpha \big\} \Big\vert \leq \Big\vert \big\{ y \in \mathbb{R}^d\ /\ \vert y-(r-\alpha)e_1 \vert \leq r, \vert y \cdot e_1 \vert \leq \alpha \big\} \Big\vert$.\\
\newline
\textbullet\ Finally, the only remaining case is when $r > p + \alpha$ and $r > \alpha$, that is when the sphere is large enough to have both its apices (along the diameter parallel to $e_1$) outside $\{-\alpha \leq y \cdot e_1 \leq \alpha\}$. The question now is then when the measure of the surface contained between the two planes $y \cdot e_1 = \pm\alpha$ is maximal. We will investigate this using an explicit computation, and separating the two cases $d = 2$ and $d \geq 3$.\\
Let us start with the simplest case, when $d \geq 3$. One has
\begin{align*}
\Big\vert \big\{ y \in \mathbb{R}^d\ /\ \big\vert x-y \big\vert \leq r, \big\vert y \cdot e_1 \big\vert \leq \alpha \big\} \Big\vert &= \int_{\hspace{-0.5mm}\substack{\vspace{-1mm}(x + r \mathbb{S}^{d-1}) \cap \{ \vert y \cdot e_1 \vert \leq \alpha \} }} \hspace{-26mm} \dd S \\
&= \int_{\hspace{-0.5mm}\substack{\vspace{-3.5mm}[p-r,p+r] \cap [-\alpha,\alpha]}} \hspace{-20mm} \Big\vert \Big( \sqrt{r^2 - (p-z)^2} \Big) \mathbb{S}^{d-2} \Big\vert \dd z.
\end{align*}
The hypotheses about $p$, $r$ and $\alpha$ imply that $[-\alpha,\alpha] \subset [p-\alpha,p+r]$, so that
\begin{align*}
\Big\vert \big\{ y \in \mathbb{R}^d\ /\ \big\vert x-y \big\vert \leq r, \big\vert y \cdot e_1 \big\vert \leq \alpha \big\} \Big\vert &= \int_{-\alpha}^\alpha \hspace{-2mm} \big(r^2-(p-y)^2\big)^{(d-2)/2} \big\vert \mathbb{S}^{d-2} \big\vert \dd y \\
&= C(d) r^{d-1} \hspace{-1mm} \int_{-(\alpha+p)/r}^{(\alpha-p)/r} \hspace{-5mm} (1-u^2)^{(d-2)/2} \dd u.
\end{align*}
The function $p \mapsto \int_{-(\alpha+p)/r}^{(\alpha-p)/r} (1-u^2)^{(d-2)/2}$ is decreasing for $p\geq 0$ (it means that the surface is maximal when the center of the sphere is exactly between the two planes delimiting $\{-\alpha \leq y \cdot e_1 \leq \alpha\}$), one has
$$
\int_{\substack{\vspace{-0.5mm}-(\alpha+p)/r}}^{\substack{\vspace{0.5mm}(\alpha-p)/r}} \hspace{-9mm}(1-u^2)^{(d-2)/2} \dd u \leq 2 \int_0^{\alpha/r} \hspace{-4mm}(1-u^2)^{(d-2)/2} \dd u \leq 2 \frac{\alpha}{r}.
$$
We obtained in the end for the case $d\geq 3$ that
\begin{align*}
\Big\vert \big\{ y \in \mathbb{R}^d\ /\ \big\vert x-y \big\vert \leq r, \big\vert y \cdot e_1 \big\vert \leq \alpha \big\} \Big\vert \leq C(d) r^{d-2} \alpha,
\end{align*}
so that \eqref{SECT3PREUVSpherEntr2Plansdim>2} is proved.\\
When $d=2$, the set $\big\{ y \in \mathbb{R}^2\ /\ \big\vert x-y \big\vert \leq r, \big\vert y \cdot e_1 \big\vert \leq \alpha \big\}$ corresponds to $\big\{ (y_1,y_2) \in \mathbb{R}^2\ /\ (y_1-p)^2+y_2^2 = r^2,\ -\alpha \leq y_1 \leq \alpha \big\}$, so that the explicit computation leads to
\begin{align}
\label{SECT3PREUVSpherEntr2Plansdim21}
\Big\vert \big\{ y \in \mathbb{R}^2\ /\ \big\vert x-y \big\vert \leq r, \big\vert y \cdot e_1 \big\vert \leq \alpha \big\} \Big\vert &= 2 \int_{-\alpha}^\alpha \sqrt{1 + \frac{(y_1-p)^2}{r^2 - (y_1-p)^2}} \dd y_1 \nonumber\\
&= 2r \Big( \arccos\Big(\hspace{-1mm}-\frac{(\alpha+p)}{r}\Big) - \arccos\Big(\frac{\alpha-p}{r}\Big) \Big)
\end{align}
The goal now is to simplify the quantity \eqref{SECT3PREUVSpherEntr2Plansdim21}, firstly by removing the dependency with respect to $p$, then to $r$, and lastly by obtaining a more convenient expression.\\
First, the quantity \eqref{SECT3PREUVSpherEntr2Plansdim21}, seen as a function of $p$, is increasing (as a simple computation of the derivative shows it), so after introducing a cut-off on the values of $p$ that are close to the upper bound $p=r$, namely:
\begin{equation}
\label{}
p \leq r - \alpha^a,
\end{equation}
with $a\in\,]0,1[$ (so that $\alpha^a-\alpha > 0$ for $\alpha$ small enough), we find
\begin{align}
\label{SECT3PREUVSpherEntr2Plansdim22}
2r \Big( \arccos\Big(\hspace{-1mm}-\frac{(\alpha+p)}{r}\Big) - \arccos\Big(\frac{\alpha-p}{r}\Big) \Big) \leq 2r \Big( \arccos\Big(\frac{\alpha^a-\alpha}{r}-1\Big) - \arccos\Big(\frac{\alpha^a+\alpha}{r}-1\Big) \Big).
\end{align}
Now, we consider the right hand side of the inequality \eqref{SECT3PREUVSpherEntr2Plansdim22} as a function of $r$, and again the computation of the derivative proves that it is increasing if and only if
\begin{align}
\label{SECT3PREUVSpherEntr2Plansdim23}
\arccos\Big(\frac{\alpha^a-\alpha}{r}-1\Big) - \arccos\Big(\frac{\alpha^a-\alpha}{r}-1\Big) \geq \frac{\alpha^a + \alpha}{\sqrt{2r(\alpha^a+\alpha)-(\alpha^a+\alpha)^2}} - \frac{\alpha^a-\alpha}{\sqrt{2r(\alpha^a-\alpha)-(\alpha^a-\alpha)^2}}\cdotp
\end{align}
If one assumes in addition that $\alpha$ is small compared to $r$, that is, explicitely, such that $2\alpha^a \leq r$ and $\alpha \leq 1$, then $\alpha^a-\alpha -r \leq \alpha^a + \alpha - r \leq 2\alpha^a - r \leq 0$, so that the arguments inside the two arccosine functions in \eqref{SECT3PREUVSpherEntr2Plansdim23} are negative. Since the arccosine function is convex on $[-1,0]$, we can then state that
\begin{align*}
\arccos\Big(\frac{\alpha^a-\alpha}{r}-1\Big) - \arccos\Big(\frac{\alpha^a+\alpha}{r}-1\Big) &\geq \Big(\Big(\frac{\alpha^a+\alpha}{r}-1\Big) - \Big(\frac{\alpha^a-\alpha}{r}-1\Big)\Big)\frac{1}{\sqrt{1-\Big(\frac{\alpha^a+\alpha}{r}-1\Big)^2}} \\
&\geq \frac{2\alpha}{\sqrt{2r(\alpha^a+\alpha)-(\alpha^a+\alpha)^2}}\cdotp
\end{align*}
As a consequence, the following inequality
\begin{align}
\label{SECT3PREUVSpherEntr2Plansdim24}
\frac{2\alpha}{\sqrt{2r(\alpha^a+\alpha)-(\alpha^a+\alpha)^2}} \geq \frac{\alpha^a + \alpha}{\sqrt{2r(\alpha^a+\alpha)-(\alpha^a+\alpha)^2}} - \frac{\alpha^a-\alpha}{\sqrt{2r(\alpha^a-\alpha)-(\alpha^a-\alpha)^2}} 
\end{align}
would imply \eqref{SECT3PREUVSpherEntr2Plansdim23}. But since $\alpha^a - \alpha > 0$, \eqref{SECT3PREUVSpherEntr2Plansdim24} is equivalent to
$$
2r(\alpha^a-\alpha) - (\alpha^a-\alpha)^2 \leq 2r(\alpha^a+\alpha) - (\alpha^a+\alpha)^2
$$
or again $\alpha^a \leq r$, which was assumed, so \eqref{SECT3PREUVSpherEntr2Plansdim24}, and then \eqref{SECT3PREUVSpherEntr2Plansdim23} hold. As a consequence, for $2\alpha^a \leq r \leq R$, one has
\begin{align}
\label{SECT3PREUVSpherEntr2Plansdim25}
r\Big(\arccos\Big(&\frac{\alpha^a-\alpha}{r}-1\Big) - \arccos\Big(\frac{\alpha^a+\alpha}{r}-1\Big)\Big) \leq R\Big(\arccos\Big(\frac{\alpha^a-\alpha}{R}-1\Big) - \arccos\Big(\frac{\alpha^a+\alpha}{R}-1\Big)\Big).
\end{align}
Finally, the upper bound of \eqref{SECT3PREUVSpherEntr2Plansdim25} can be simplified when $\alpha$ is small, using basically the idea that $\arccos(x) - \pi \underset{-1}{\sim} - \sqrt{2(x+1)}$. More precisely, using the identity:
$$
\pi - \sqrt{3(x+1)} \leq \arccos(x) \leq \pi - \sqrt{2(x+1)},
$$ holding for all $x \in [-1,-2/3]$, one finds
\begin{align*}
\arccos\Big(\frac{\alpha^a-\alpha}{R}-1\Big) - \arccos\Big(\frac{\alpha^a+\alpha}{R}-1\Big) &\leq \sqrt{3\frac{(\alpha^a+\alpha)}{R}} - \sqrt{2\frac{(\alpha^a-\alpha)}{R}} \\
&\leq \frac{\alpha^{a/2}}{\sqrt{R}}\Big( \sqrt{3}\sqrt{1+\alpha^{1-a}} - \sqrt{2}\sqrt{1-\alpha^{1-a}} \Big) \\
&\leq \frac{\alpha^{a/2}}{\sqrt{R}}\big( \sqrt{3}(1+\alpha^{1-a}/2) - \sqrt{2}(1-\alpha^{1-a}) \big),
\end{align*}
as soon as $\alpha \leq (R/6)^{1/a}$, implying in particular $0 \leq (\alpha^a + \alpha)/R \leq 1/3$. If one has in addition that $(\sqrt{3}/2+\sqrt{2})\alpha^{1-a} \leq 1 - (\sqrt{3} - \sqrt{2})$, that is $\alpha \leq \big((1+\sqrt{2}-\sqrt{3})/(\sqrt{3}/2+\sqrt{2})\big)^{1/(1-a)}$, then in the end
\begin{align*}
\arccos\Big(\frac{\alpha^a-\alpha}{R}-1\Big) &- \arccos\Big(\frac{\alpha^a+\alpha}{R}-1\Big) \leq \frac{\alpha^{a/2}}{\sqrt{R}}.
\end{align*}
Multiplying the difference of the two arccosines by $2r$ and keeping in mind that $r$ is bounded by $R$, we recover the result \eqref{SECT3PREUVSpherEntr2Plansdim=2} for the dimension $d=2$, with the restrictions that were described concerning $\alpha$, when we choose $a = 1/2$.\\
\newline
Back to the control of the size of the set $\mathcal{N}^*(R,\alpha)(v_1)$, we recall that, by definition, this set is composed of the adjunction parameters $\omega$ and $v_2$ leading to at least one grazing post-collisional velocity $v'$ (that is such that $-\alpha \leq v'\cdot e_1 \leq \alpha$) between $v'_1$ and $v'_2$. Here it is important to recall the elementary geometrical propetries fulfilled by those post-collisional velocities: by definition of the scattering operator, for $v_1$ and $v_2$ fixed, the two velocities $v'_1$ and $v'_2$ lie in the boundary of the ball centered on $(v_1+v_2)/2$ and of radius $\vert v_1 - v_2 \vert/2$, $v_1'$ and $v_2'$ delimiting a diameter of this ball. In addition, the three vectors $v'_1-v_1$, $v'_2-v_2$ and $\omega$ have always the same direction.\\
Now that the general geometrical setting has been set, let us now describe the main argument for the control of the size of $\mathcal{N}^*(R,\alpha)(v_1)$. Assuming that $v_1$ and $v_2$ are both fixed, if $v'_1$ is prescribed to lie in a given part of the ball centered on $(v_1+v_2)/2$ and of radius $\vert v_1 - v_2 \vert/2$, what can be said about the part in which $\omega$ lies? In the case of $v'_1$, it is also important to notice that $v'_1-v_1$ and $\omega$ have, in addition to the same direction, the same orientation (since $(v_2-v_1)\cdot\omega > 0$: we are in the situation in which the scattering has to be applied). The situation can then be pictured as in Figure \ref{SECT3FigurLien_Entrev'_1&Omega} below.
\begin{figure}[!h]
\centering
\includegraphics[scale=0.06, trim = 0.95cm 2cm 0cm 1cm]{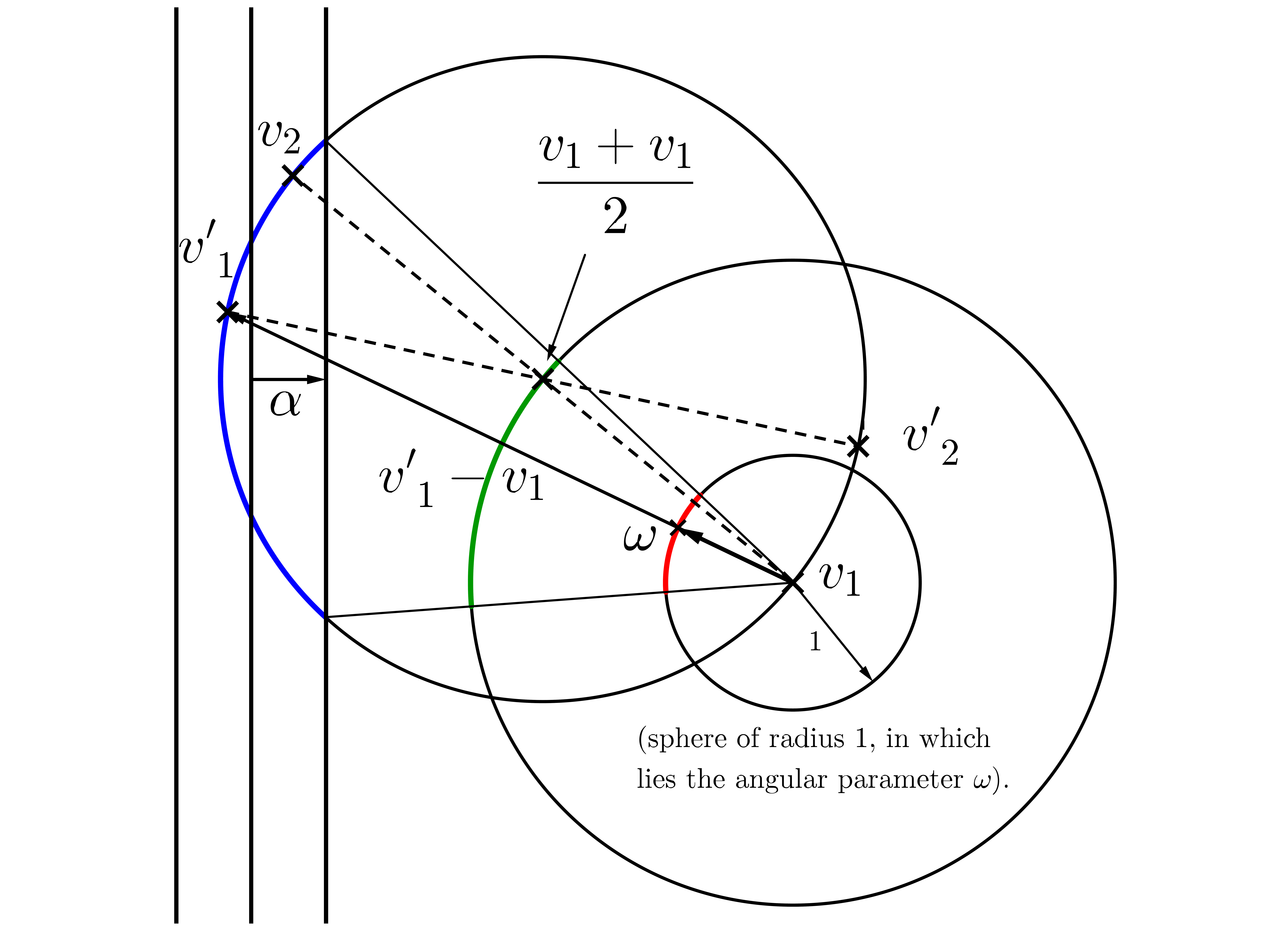}
\caption{The link between the constraints over the position of $v'_1$ and the position of $\omega$.}
\label{SECT3FigurLien_Entrev'_1&Omega}
\end{figure}
The key observation is that, in the two dimensional case, if the vector $v'_1-(v_1+v_2)/2$ covers some angle, then, thanks to the inscribed angle theorem, the vector $v'_1-v_1$ will cover half of this angle, since $v'_1-(v_1+v_2)/2$ has its origin at the center of the ball centered on $(v_1+v_2)/2$ and of radius $\vert v_1 - v_2 \vert/2$, while the origin of $v'_1-v_1$ lies on the boundary of this ball. From this point, the link between the surface covered by $v'_1-v_1$ and the one covered by $\omega$ is a simple question of scaling. Let us now make this argument rigorous.\\
We saw that the idea is to solve the problem when $v_2$ is fixed, and then deduce the general result by releasing the constraint on $v_2$. We introduce then
$$
\mathcal{N}^*(\alpha)(v_1,v_2) = \big\{ \omega \in \mathbb{S}^{d-1}\ /-\alpha \leq v'_1 \cdot e_1 \leq \alpha \text{     or   } -\alpha \leq v'_2 \cdot e_1 \leq \alpha \big\},
$$
so that $\vert \mathcal{N}^*(R,\alpha)(v_1) \vert = \int_{v_2 \in B(0,R)} \vert \mathcal{N}^*(R,\alpha)(v_1,v_2) \vert \dd v_2$.

For the dimension $d\geq 3$, we have to parametrize properly the problem. Denoting $\tilde{r}$ the norm of $v_2-v_1$ and $u \in \mathbb{S}^{d-2}$ its normalized projection on the hyperplane $v \cdot e_1 = 0$, we have $v_2 = v_2(u,\tilde{r},\theta) = v_1 + \tilde{r}(\cos\theta e_1 + \sin \theta u)$, with $\theta \in [0,\pi]$. The polar coordinates provide then
$$
\big\vert \mathcal{N}^*(R,\alpha)(v_1) \big\vert \leq C(d) \int_{\mathbb{S}^{d-1}}\int_0^{2R}\int_0^\pi \tilde{r}^{d-1} \big\vert \mathcal{N}^*(\alpha)\big(v_1,v_2(u,\tilde{r},\theta)\big) \big\vert \dd \theta \dd \tilde{r} \dd u.
$$
Therefore, considering $x = (v_1+v_2)/2$ and $r = \vert v_2-v_1 \vert/2$, $\omega$ belongs to $\mathcal{N}^*(\alpha)(v_1,v_2)$ if and only if $v'_1$ or $v'_2$ belong to $\partial B(x,r) \cap \{ \vert y\cdot e_1 \vert \leq \alpha \}$.\\
We can now use the argument based on the inscribed angle theorem, since for any two-dimensional plane $\mathcal{P}$ through $v_1$, the intersection of this plane with $\partial B(x,r)$ is a circle which contains $v_1$, while the intersection of $\mathcal{P}$ with the ball centered on $v_1$ and of radius $\vert v_1-v_2 \vert/2 = r$ is a circle of the same radius (so that we are exactly in the situation described in Figure \ref{SECT3FigurLien_Entrev'_1&Omega}). We have then
\begin{align*}
\Big\vert \mathcal{P} \cap \partial B(v_1,r) &\cap \big\{z\ /\ \exists\, \lambda \in \mathbb{R}\ /\ \lambda(z-v_1) \in \{-\alpha \leq y\cdot e_1 \leq \alpha\}\cap \partial B(x,r) \big\} \Big\vert \\
&= \frac{1}{2}\big\vert \mathcal{P} \cap \big( \partial B(x,r) \cap \{-\alpha \leq y \cdot e_1 \leq \alpha\}\big) \big\vert,
\end{align*}
(it means that the green line on Figure \ref{SECT3FigurLien_Entrev'_1&Omega} measures half the length of the blue line), so that
\begin{align*}
\Big\vert \partial B(v_1,r) \cap \big\{z\ /\ \exists\, \lambda& \in \mathbb{R}\ /\ \lambda(z-v_1) \in \{-\alpha \leq y\cdot e_1 \leq \alpha\}\cap \partial B(x,r) \big\} \Big\vert \\
&= \frac{1}{2}\big\vert\partial B(x,r) \cap \{-\alpha \leq y \cdot e_1 \leq \alpha\} \big\vert.
\end{align*}
Finally, and it is an important point, the elements $z$ of $\partial B(v_1,r) \cap \big\{z\ /\ \exists\, \lambda \in \mathbb{R}\ /\ \lambda(z-v_1) \in \{-\alpha \leq y\cdot e_1 \leq \alpha\}\cap \partial B(x,r) \big\}$ are such that $z-v_1$ has a norm equal to $r$, and is colinear to $\omega$ (since $z-v_1$ is colinear to $v'_1-v_1$ by construction). So we can state in the end that
\begin{align*}
\big\vert \mathcal{N}^*(\alpha)\big(v_1,v_2(u,\tilde{r},\theta)\big) \big\vert = \frac{1}{r^{d-1}}\big\vert\partial B(x,r) \cap \{-\alpha \leq y \cdot e_1 \leq \alpha\} \big\vert.
\end{align*}
One will now conclude using the intermediate results \eqref{SECT3PREUVSpherEntr2Plansdim=2} and \eqref{SECT3PREUVSpherEntr2Plansdim>2}. For the case $d\geq 3$, one has
\begin{align*}
\big\vert \mathcal{N}^*(R,\alpha)(v_1) \big\vert &\leq C(d) \int_{\mathbb{S}^{d-2}}\int_0^{2R}\int_0^\pi \tilde{r}^{d-1}\Big(\frac{1}{\tilde{r}^{d-1}}\tilde{r}^{d-2}\alpha\Big) \dd \theta \dd \tilde{r} \dd u\\
&\leq C(d)R^{d-1}\alpha,
\end{align*}
which is the sought result for the case when the dimension is higher than $2$.\\
When $d=2$, one has to be more careful, due to the restrictions specific to this case. We start by cutting off the small difference between the velocities $v_1$ and $v_2$ (to take into account the constraint $2\sqrt{\alpha} \leq r \leq 2R$): to do so, let us consider a parameter $b \in ]0,1[$, and let us write
\begin{align*}
\big\vert \mathcal{N}^*(R,\alpha)(v_1) \big\vert \leq \int_{\substack{\vspace{-1.5mm}B(v_1,\alpha^b)}} \hspace{-8mm} \big\vert \mathcal{N}^*(\alpha)(v_1,v_2) \big\vert \dd v_2 + \int_{\substack{\vspace{-1.5mm}B(v_1,2R)\backslash B(v_1,\alpha^b)}} \hspace{-21mm} \big\vert \mathcal{N}^*(\alpha)(v_1,v_2) \big\vert \dd v_2.
\end{align*}
Let us rewrite and bound from above the second term using again the polar coordinates, which take the much simpler expression in the case $d=2$:
\begin{align*}
\big\vert \mathcal{N}^*(R,\alpha)(v_1) \big\vert \leq \int_{\substack{\vspace{-1.5mm}B(v_1,\alpha^b)}} \hspace{-8mm} \big\vert \mathcal{N}^*(\alpha)(v_1,v_2) \big\vert \dd v_2 + \int_{\alpha^b}^{2R} \int_0^\pi \tilde{r} \big\vert \mathcal{N}^*(\alpha)\big(v_1,v_2(\tilde{r},\theta)\big) \big\vert \dd \theta \dd \tilde{r}.
\end{align*}
In order to apply the control \eqref{SECT3PREUVSpherEntr2Plansdim=2}, we need now to take into account the constraint $\vert x \cdot e_1 \vert \leq r - \sqrt{\alpha}$, that is here $\vert v_1\cdot e_1 + (\tilde{r}\cos\theta)/2 \vert \leq \tilde{r}/2 - \sqrt{\alpha}$. Remembering also that if $r < \vert x \cdot e_1 \vert - \alpha$, then $\vert \{ y \in \partial B(x,r)\ /\ \vert y \cdot e_1 \vert \leq \alpha \} \vert = 0$, we decompose then the integral over $\theta$ as:
\begin{align*}
\int_{\alpha^b}^{2R} \hspace{-2mm}\int_0^\pi \tilde{r} \big\vert \mathcal{N}^*(\alpha)\big(v_1,v_2(\tilde{r},\theta)\big) \big\vert \dd \theta \dd \tilde{r} \leq \int_{\alpha^b}^{2R}\hspace{-2mm}& \int_{\substack{\vspace{-1.5mm}\tilde{r}/2-\sqrt{\alpha} < \vert v_1\cdot e_1 + (\tilde{r}\cos\theta)/2 \vert \leq \tilde{r}/2 + \alpha}} \hspace{-43mm} \tilde{r} \big\vert \mathcal{N}^*(\alpha)\big(v_1,v_2(\tilde{r},\theta)\big) \big\vert \dd \theta \dd \tilde{r} \\
&\hspace{20mm} + \int_{\alpha^b}^{2R}\hspace{-2mm} \int_{\substack{\vspace{-1.5mm}\vert v_1\cdot e_1 + (\tilde{r}\cos\theta)/2 \vert \leq \tilde{r}/2-\sqrt{\alpha})}} \hspace{-34mm} \tilde{r} \big\vert \mathcal{N}^*(\alpha)\big(v_1,v_2(\tilde{r},\theta)\big) \big\vert \dd \theta \dd \tilde{r}.
\end{align*}
The first term will be small since the domain of integration is small, and \eqref{SECT3PREUVSpherEntr2Plansdim=2} will hold for the second term. More precisely, the condition $\tilde{r}/2-\sqrt{\alpha} < \vert v_1\cdot e_1 + (\tilde{r}\cos\theta)/2 \vert \leq \tilde{r}/2 + \alpha$ contains actually the two conditions
\begin{align*}
1- \frac{2}{\tilde{r}}\big(\sqrt{\alpha}+v_1\cdot e_1\big) < \cos \theta \leq 1 + \frac{2}{\tilde{r}}\big(\alpha-v_1\cdot e_1) \text{   or   } -1-\frac{2}{\tilde{r}}(\alpha+v_1\cdot e_1) < \cos\theta \leq -1 + \frac{2}{\tilde{r}}\big(\sqrt{\alpha}-v_1\cdot e_1\big).
\end{align*}
Using now the fact that the second derivative of the arccosine function is non negative on $[-1,0]$, we can deduce that for $-1\leq x \leq y \leq 0$ holds $\arccos(x) - \arccos(y) \leq \arccos(-1) - \arccos(-1+y-x)$. This provides then that
\begin{align*}
\arccos\Big(\min\big(1,1-\frac{1}{\tilde{r}}(\sqrt{\alpha}+v_1\cdot e_1\big)\Big) - \arccos\Big(\min\big(1,1+&\frac{2}{\tilde{r}}(\alpha-v_1\cdot e_1\big)\Big) \\
&\leq \arccos(-1) - \arccos\big(-1+\frac{2}{\tilde{r}}(\alpha+\sqrt{\alpha})\big) \\
&\leq \arccos(-1) - \arccos\big(-1+\frac{4}{\tilde{r}}\sqrt{\alpha}\big)
\end{align*}
and similarly
\begin{align*}
\arccos\Big(\max\big(-1,-1-\frac{1}{\tilde{r}}(\alpha+v_1\cdot e_1\big)\Big) - \arccos\Big(\min\big(-1,-1&+\frac{2}{\tilde{r}}(\sqrt{\alpha}-v_1\cdot e_1\big)\Big) \\
&\leq \arccos(-1) - \arccos\big(-1+\frac{4}{\tilde{r}}\sqrt{\alpha}\big).
\end{align*}
To get rid of the arccosines in the expression, we use in the end that $\pi - \sqrt{3(x+1)} \leq \arccos(x)$ for all $x \in [-1,-2/3]$, providing that
$$
\arccos(-1) - \arccos\big(-1+\frac{4}{\tilde{r}}\sqrt{\alpha}\big) \leq \sqrt{3\frac{4}{\tilde{r}}\sqrt{\alpha}},
$$
up to have $\sqrt{\alpha}/\tilde{r} \leq 1/12$, hence 
$$\big\vert \big\{ \theta \in [0,\pi]\ /\ \tilde{r}/2-\sqrt{\alpha} < \vert v_1\cdot e_1 + (\tilde{r}\cos\theta)/2 \vert \leq \tilde{r}/2 + \alpha \big\} \big\vert \leq 4\sqrt{3}\alpha^{1/4-b/2}.
$$
Back to the decomposition of $\big\vert \mathcal{N}^*(R,\alpha)(v_1) \big\vert$, we have, bounding roughly the integrands of the two first terms (keeping simply in mind that $\mathcal{N}^*(\alpha)(v_1,v_2)$ is by definition a part of the sphere $\mathbb{S}^{d-1}$) and applying at last \eqref{SECT3PREUVSpherEntr2Plansdim=2} in the last one:
\begin{align*}
\hspace{-20mm}\big\vert \mathcal{N}^*(R,\alpha)(v_1) \big\vert \leq \int_{B(v_1,\alpha^b)} \big\vert \mathbb{S}^{d-1} \big\vert &\dd v_2 + C \int_{\alpha^b}^{2R}\hspace{-2mm} \int_{\substack{\vspace{-1.5mm}\tilde{r}/2-\sqrt{\alpha} < \vert v_1\cdot e_1 + (\tilde{r}\cos\theta)/2 \vert \leq \tilde{r}/2 + \alpha}} \hspace{-43mm} 2R \big\vert \mathbb{S}^{d-1} \big\vert \dd \theta \dd \tilde{r} \\
&+ C \int_{\alpha^b}^{2R}\hspace{-2mm} \int_{\substack{\vspace{-1.5mm}\vert v_1\cdot e_1 + (\tilde{r}\cos\theta)/2 \vert \leq \tilde{r}/2-\sqrt{\alpha})}} \hspace{-34mm} \tilde{r}\Big( \frac{1}{\tilde{r}} \sqrt{R}\alpha^{1/4}\Big) \dd \theta \dd \tilde{r} \\
& \leq C \big( \alpha^{2b} + R^2 \alpha^{1/4-b/2} + R^{3/2} \alpha^{1/4} \big).
\end{align*}
It is important to notice that \eqref{SECT3PREUVSpherEntr2Plansdim=2} holds for $R\geq 1$ and $\alpha$ small enough, with a condition which does not depend on $R$, and such that $\sqrt{\alpha} \leq \tilde{r}/4$. If one takes $b < 1/2$, and if $\tilde{r} \geq \alpha^b$, then $\tilde{r}/(4\sqrt{\alpha}) \geq \alpha^{b-1/2}/4 \rightarrow + \infty$ when $\alpha \rightarrow 0$, so that there exists a constant $c(b)$ such that for $\alpha \leq c(b)$, one has $\tilde{r}/4 \geq \sqrt{\alpha}$ and $\sqrt{\alpha}/\tilde{r} \leq 1/12$ for all $\tilde{r} \in [\alpha^b,2R]$, which enables indeed to use \eqref{SECT3PREUVSpherEntr2Plansdim=2}, and in addition to control the size of the domain of the second term.\\
As a conclusion, since $\big\vert \mathcal{N}^*(R,\alpha)(v_1) \big\vert \leq CR^2\big(\alpha^{2b}+\alpha^{1/4-b/2}\big)$, choosing $b=1/4$ provides the conclusion of the lemma, and the size of $\mathcal{N}^*(R,\alpha)(v_1)$ is going to zero for $R$ fixed and as $\alpha$ goes to zero. This concludes the proof of Lemma \ref{SECT3LemmeScatePavé_VitesRasan}.
\end{proof}

\subsection{The convergence of the pseudo-trajectories}

All the tools are now in place to obtain pseudo-trajectories without recollisions. Now, we can see on which set of initial configurations the geometrical results may be used to construct such pseudo-trajectories. First, let us introduce the domain of local uniform convergence:
\begin{defin}[Domain of local uniform convergence]
\label{SECT3DefinDomaiConveLocalUnifo}
let $s$ be a positive integer. We introduce the \emph{domain of local uniform convergence}, denoted as $\Omega_s$, as the subset $\Omega_s$ of the phase space of $s$ particles defined as $\Omega_s = \bigcap_{j=1}^4 \Omega_s^j$ with
$$
\left\{
\begin{array}{rl}
\Omega_s^1 &= \big\{Z_s\ /\ \forall i \neq j, x_i \neq x_j\big\},\\
\Omega_s^2 &= \big\{Z_s\ /\ \forall i, v_i \cdot e_1 \neq 0\big\},\\
\Omega_s^3 &= \big\{Z_s\ /\ \forall i \neq j, v_j \notin v_i + \text{Vect}(x_i-x_j)\big\},\\
\Omega_s^4 &= \big\{Z_s\ /\ \forall i \neq j, v_j \notin \mathcal{S}_0(v_i) + \text{Vect}\big(\mathcal{S}_0(x_i)-x_j\big)\big\},
\end{array}
\right.
$$
\end{defin}
together with the following subsets in the phase space $\mathcal{D}^\varepsilon_s \subset \mathbb{R}^{2ds}$: we set $\Delta_s = \Delta_s(\varepsilon,R,\varepsilon_0,\alpha,\gamma) = \bigcap_{j=1}^6 \Delta_s^j$, where
\begin{align*}
\Delta_s^1& = \Delta_s^1(\varepsilon) = \big\{  \forall\, 1 \leq i \leq s,\ x_i\cdot e_1 > \varepsilon/2 \big\}, \Delta_s^2 = \Delta_s^2(R) = \big\{ \vert V_s \vert \leq R \big\},\\
\Delta_s^3 = \Delta_s^3&(\varepsilon_0) = \big\{ \min_{1 \leq i < j \leq s} \vert x_i - x_j \vert \geq \varepsilon_0 \big\}, \Delta_s^4 = \Delta_s^4(\alpha) = \big\{\ \min_{1 \leq i \leq s} \vert v_i \cdot e_1 \vert \geq \alpha \big\},\\
&\Delta_s^5 = \Delta_s^5(\gamma) = \big\{ \min_{1 \leq i < j \leq s} d\big(v_j,v_i+\text{Vect}(x_i-x_j)\big) \geq \gamma \big\},\\
\Delta_s^6& = \Delta_s^6(\gamma) = \big\{ \min_{1 \leq i < j \leq s} d\big(v_j,\mathcal{S}_0(v_i)+\text{Vect}\big(\mathcal{S}_0(x_i)-x_j\big) \big) \geq \gamma \big\}.
\end{align*}
The first condition avoid the overlapping with the wall. All the others but the fourth enable to use Proposition \ref{SECT3TheorBonneConfiAdjonParti}, so that from an initial configuration taken in $\Delta_s$, all the adjunctions but a small amount of them (with a controlled size) will provide pseudo-trajectories without recollisions, up to perform the adjunctions on particles far enough from the wall. This will be possible for all the times but a small amount of them, thanks to the fourth condition, avoiding grazing collisions, and then particles staying too long too close to the wall.\\
For this subset $\Delta_s$ of the phase space, one can state the following result, characterizing the convergence we obtained in the final theorem.

\begin{propo}[Control of the size of the pathological adjunction parameters for pseudo-trajectories starting from a compact set of $\Omega_s$, and uniform control on the difference between the pseudo-trajectories]
\label{SECT3PropoControParamAdjonPatho}
For any compact set $K$ contained in the domain of local uniform convergence $\Omega_s$, there exist five strictly positive numbers $\overline{\varepsilon}$, $\overline{\varepsilon}_0$, $\overline{\alpha}$, $\overline{\gamma}$ and $\overline{R}$ such that for all $\varepsilon$, $\varepsilon_0$, $\alpha$, $\gamma$ and $R$ such that $\varepsilon \leq \overline{\varepsilon}, \varepsilon_0 \leq \overline{\varepsilon}_0, \alpha \leq \overline{\alpha}, \gamma \leq \overline{\gamma}$ and $R \geq \overline{R}$, we have $K \subset \Delta_s(\varepsilon,R,\varepsilon_0,\alpha,\gamma)$.\\
In addition, for all positive numbers $\eta,a,\delta,\rho$ and $\alpha$ such that $R \geq 1, \eta \leq 1, 2\varepsilon \leq a, 4\sqrt{3} \leq \varepsilon_0, \varepsilon_0 \leq \eta\delta, 3a \leq \rho, \alpha \leq c(d)$ and
$max(\varepsilon_0/\delta,16Ra/\varepsilon_0) \leq \overline{\gamma}$, where $c(d)$ is a constant which depends only on the dimension $d$, for all $Z_s \in \Delta_s$, $T_k = (t_1,\dots,t_k) \in \mathfrak{T}_k$, $J_k = (j_1,\dots,j_k) \in \mathfrak{J}^s_k$ and $A_k = \big((\omega_1,v_{s+1}),\dots,(\omega_k,v_{s+k})\big) \in \mathfrak{A}_k$, there exist two families of respective subsets of $[0,T]$ and $\mathbb{S}^{d-1}\times\mathbb{R}^d$
$$
U_{J_k} = \big(U_{j_1}(Z_s),U_{j_2}(Z^0_{s,1}(t_1)),\dots,U_{j_k}(Z^0_{s,k-1}(t_{k-1})\big)
$$
(for the pathological times of adjunction) and
$$
E_{J_k} = \big(E_{j_1}(Z^0_{s,0}(t_1)),\dots,E_{j_k}(Z^0_{s,k-1}(t_k))\big)
$$
(for the pathological velocity and angular parameter of adjunction) such that for all $1 \leq l \leq k$:
\begin{align}
&\hspace{20mm}\big\vert U_{j_l} \big( Z^0_{s,l-1}(t_{l-1}) \big) \big\vert \leq 2\frac{\rho}{\alpha}, \\
\big\vert E_{j_l} \big(Z^0_{s,l-1}(t_l)\big) \big\vert \leq C(d) &\Big( \eta^d + R^d\Big(\frac{a}{\rho}\Big)^{d-1} \hspace{-3mm} + \hspace{0.5mm} kR^{2d-1}\Big(\frac{a}{\varepsilon_0}\Big)^{d-3/2} \hspace{-3mm} + \hspace{0.5mm} kR^{d+1/2}\Big(\frac{\varepsilon_0}{\delta}\Big)^{d-3/2} + R^{d-1}\alpha + R^d\alpha^{1/8} \Big),
\end{align}
and such that if $t_l \notin U_{j_l}$ and $(\omega_l,v_{s+l}) \notin E_{j_l}$, then:
\begin{itemize}[leftmargin=*]
\item no particle of the pseudo-trajectory $Z^\varepsilon(Z_s,T_k,J_k,A_k)$ associated to the BBGKY hierarchy will undergo a recollision, that is $\big\vert x^{\varepsilon,m}_{s,l}(\tau) - x^{\varepsilon,m'}_{s,l}(\tau) \big\vert > \varepsilon$ for all $0 \leq l \leq k$, $\tau \in [t_{l+1},t_l]$ and $1 \leq m \neq m' \leq s+l$,
\item the distance between the particles of the BBGKY and the Boltzmann pseudo-trajectories is uniformly controlled after the $k$ adjunctions, and more precisely $\big\vert x^{\varepsilon,m}_{s,k}(0) - x^{0,m}_{s,k}(0) \big\vert \leq 2k\varepsilon$.
\end{itemize}
\end{propo}

\begin{remar}
This proposition signifies that the goal of removing the recollisions, except for small subsets of adjunction parameters with a controlled size, has been fulfilled. In addition, the closedness of the two different pseudo-trajectories starting from the same initial configuration and undergoing the same adjunctions has been quantitavily obtained.\\
The definition of the elements $U_{j_l} \big( Z^0_{s,l-1}(t_{l-1}) \big)$ and $E_{j_l} \big(Z^0_{s,l-1}(t_l)\big)$ may look a bit intricated, but they are perfectly well defined by recursion: starting from $Z_s$ and choosing the first particle $j_1$ to undergo an adjunction enables to defines $U_{j_1}$, so we can choose $t_1$ outside, enabling now to define $Z^0_{s,0}(t_1)$, and then $E_{j_1}$. From this point, one can just iterate the process, choosing the second particle $j_2$ for the adjunction, defining then $U_{j_2}$, and so on.
\end{remar}

\begin{proof}[Proof of Proposition \ref{SECT3PropoControParamAdjonPatho}]
The proof is a simple induction using the geometrical lemmas.\\ The first step consists in noticing that once $Z_s$ belongs to $\Delta_s(\varepsilon,R,\varepsilon_0,\alpha,\gamma)$ with $\max(\varepsilon_0/\delta,16Ra/\varepsilon_0) \leq \gamma$, the velocities of this configuration are outside the cylinders described in Lemma \ref{SECT3LemmedeTirPlanAxes_Fixes}, and its results apply: we have $Z_s \in \mathcal{G}^\varepsilon_s(\varepsilon)$, and $T^{s,0}_{-\delta}(Z_s) \in \mathcal{G}^0_s(\varepsilon_0)$. In addition, all the velocities $v_i$ of the configuration $Z_s$ satisfy $\vert v_i \cdot e_1 \vert \geq \alpha$. As a consequence, for $j_1$, the index of the first particle undergoing an adjunction, the set $U_{j_1}(Z_s)$ of times $\tau$ such that $\vert x^0_{j_1}(\tau) \cdot e_1 \vert < \rho$ has a measure smaller than $2\rho/\alpha$. Then, for any time $t_1$ (chosen for the first adjunction) outside $U_{j_1}(Z_s)$, the configuration $Z^0_{s,0}(t_1) = T^{s,0}_{-t_1}(Z_s)$ fulfills the condition of Theorem \ref{SECT3TheorBonneConfiAdjonParti} (in particular, it is easy to show that the positions of the particles of the configuration $Z^\varepsilon_{s,0}(t_1)$ are close to the corresponding particles of $Z^0_{s,0}(t_1)$ since the trajectories are easily explicitable, and the velocities are the same up to apply the symmetry with respect to the wall), so this theorem provides a set $\mathcal{B}_s\big(Z^0_{s,0}(t_1)\big)$ for which, if the adjunction parameters $(\omega_1,v_{s+1})$ are taken outside of, the new configurations of $s+1$ particles $T^{s+1,0}_{-\delta}\big(Z^0_{s,1}(t_1)\big) = Z^0_{s,1}(t_1-\delta)$ and $Z^\varepsilon_{s,1}(t_1)$ belong respectively to $\mathcal{G}^0_{s+1}(\varepsilon_0)$ and $\mathcal{G}^\varepsilon_{s+1}(\varepsilon)$. In order to have new configurations of $s+1$ particles at $t_1$ such that all the velocities are not grazing (such that $\vert v_i \cdot e_1 \vert \geq \alpha$), one has to remove only the set $\{ \vert v_{s+1} \cdot e_1 \vert \leq \alpha \}$, of measure $C(d)R^{d-1}\alpha$, if the adjunction with $(\omega_1,v_{s+1})$ provides a pre-collisional configuration, and one has to remove the set $\mathcal{N}^*(R,\alpha)(v_{j_1})$, described in Lemma \ref{SECT3LemmeScatePavé_VitesRasan} if the adjunction provides a post-collisional configuration (so that in that case the scattering has to be applied immediately, before applying the transport). Those three exclusions define the set $E_{j_1}\big(Z^0_{s,0}(t_1)\big)$, with the size verifying the control stated in the proposition. Then on the new configuration of $s+1$ particles $Z^0_{s,1}(t_1)$ can be applied again the same process, and so on until the last $k$-th adjunction. This concludes the proof of the controls on the size of $U_{j_l}$ and $E_{j_l}$ for all $1 \leq l \leq k$, and the absence of recollision.\\
The control on the divergence between the two pseudo-trajectories $Z^0$ and $Z^\varepsilon$ can be obtained again by recursion, starting from the initial configuration $Z_s$ and the explicit expressions of the trajectories (which are now easy, thanks to the absence of recollision). The difference between the pseudo-trajectories comes from the adjunctions, but unlike the case of the domain without obstacle, it also comes from the bouncings against the obstacles, both of them generating errors of order $\varepsilon$. Proposition \ref{SECT3PropoControParamAdjonPatho} is proven.
\end{proof}

Now that we know which adjunction parameters have to be removed, it is important to see what is the impact on the terms of the Duhamel formula of such removals. Let us first introduce the truncated in adjunction parameters elementary terms of the hierarchies.

\begin{defin}[Truncated in adjunction parameters operators and elementary terms of the hierarchies]
For any integer $1 \leq j \leq s$, any function $f^{(s+1)} \in \mathcal{C}\big([0,T]\times(\overline{\Omega^c}\times\mathbb{R}^d)^{s+1}\big)$ with $f^{(s+1)}(t,\cdot) \in X_{0,s+1,\beta}$, for any $t \in [0,T]$, and for any functions $U : Z_s \mapsto U(Z_s) \in \mathcal{P}([0,t-\delta])$ and $E : Z_s \mapsto E(Z_s) \in \mathcal{P}(\mathbb{S}^{d-1}\times\mathbb{R}^d)$ of measurable subsets $U(Z_s)$ and $E(Z_s)$, we define the \emph{truncated in adjunction parameters, integrated in time collision-transport operator of the Boltzmann hierarchy of type $(\pm,j)$} the function
\begin{align*}
\mathds{1}_{t \geq \delta} \int_0^{t-\delta} \hspace{-4mm} \mathds{1}_{U(Z_s)}(u)&\int_{\substack{\vspace{-2mm} \mathbb{S}^{d-1}_\omega\times\mathbb{R}^d_{v_{s+1}}}} \hspace{-12.5mm} \mathds{1}_{E(Z^0_{s,0}(u))}  \Big[ \omega\cdot\big(v_{s+1}-v^{0,j}_{s,0}(u)\big)\Big]_\pm f^{(s+1)}\big(u,Z^0_{s,1}\big(u,j,(\omega,v_{s+1})\big)(u)\big) \dd \omega \dd v_{s+1} \dd u,
\end{align*}
denoted as $\mathcal{I}^{0,\delta}_{\hspace{-1.5mm}\substack{s \\ \pm,j}}(U,E) f^{(s+1)}(t,Z_s)$.\\
In the case when $E(Z_s) = \mathbb{S}^{d-1}\times\mathbb{R}^d$ for every $Z_s$ (the surgery occurs only on the time variable), we will simply denote the operator as $\mathcal{I}^{0,\delta}_{\hspace{-1.5mm}\substack{s \\ \pm,j}}(U)f^{(s+1)}(t,Z_s)$.\\
Then, for any positive integer $k$, any $J_k = (j_1,J_{k-1}) \in \mathfrak{J}^s_k$, $M_k = (\pm_1,M_{k-1}) \in \mathfrak{M}_k$, for any measurable function $f^{(s+k)}:[0,T]\times\big(\overline{\Omega^c}\times\mathbb{R}^d\big)^{s+k} \rightarrow \mathbb{R}$ with $f^{(s+k)}(t,\cdot) \in X_{0,s+k,\beta}$ for any $t \in [0,T]$, and for any families of functions $U_{J_k} = (U_{j_1},\dots,U_{j_k}) = (U_{j_1},U_{J_{k-1}})$ and $E_{J_k} = (E_{j_1},\dots,E_{j_k}) = (E_{j_1},E_{J_{k-1}})$ (with $U_{j_l} : Z_{s+l} \mapsto U(Z_{s+l}) \in \mathcal{P}([0,t-\delta])$ and $E_{j_l} : Z_{s+l} \mapsto E(Z_{s+l}) \in \mathcal{P}(\mathbb{S}^{d-1}\times\mathbb{R}^d)$ for all $1 \leq l \leq k$), we define the \emph{truncated in adjunction parameters elementary term of the Boltzmann hierarchy of type $(M_k,J_k)$} the function defined by recursion
\begin{align*}
\mathds{1}_{t \geq k\delta} &\int_{(k-1)\delta}^{t-\delta} \hspace{-4mm} \mathds{1}_{U_{j_1}(Z_s)} \int_{\substack{ \vspace{-2mm}\mathbb{S}^{d-1}_\omega \times \mathbb{R}^d_{v_{s+1}}}} \hspace{-12.5mm} \mathds{1}_{E_{j_1}(Z^0_{s,0}(t_1))} \Big[ \omega_1 \cdot \big(v_{s+1} - v^{0,j_1}_{s,0}(t_1)\big) \Big]_\pm \\
& \hspace{35mm} \times \Big( \mathcal{I}^{0,\delta}_{\substack{s+1,s+k-1 \\ M_{k-1},J_{k-1}}}(U_{J_{k-1}},E_{J_{k-1}}) f^{(s+k)} \Big) \big(t_1,Z^0_{s,1}\big(t_1,j_1,(\omega,v_{s+1})\big)(t_1)\big) \dd \omega \dd v_{s+1} \dd t_1,
\end{align*}
denoted as $\mathcal{I}^{0,\delta}_{\substack{s,s+k-1 \\ M_k,J_k}}(U_{J_k},E_{J_k})f^{(s+k)}$.\\
In the same fashion, we introduce the \emph{truncated in adjunction parameters elementary terms of the BBGKY hierarchy of type $(M_k,J_k)$}, defined here for $h^{(s+k)} \in L^\infty\big([0,T]\times\mathcal{D}^\varepsilon_{s+k}\big)$, denoted $\mathcal{I}^{N,\varepsilon,\delta}_{\substack{s,s+k-1 \\ (M_k,J_k)}}(U_{J_k},E_{J_k})h^{(s+k)}$.
\end{defin}

\begin{remar}
Considering the two subsets of pathological adjunction parameters $U_{j_l}$ and $E_{j_l}$ defined and studied in Proposition \ref{SECT3PropoControParamAdjonPatho} the objective is to perform the decomposition $\mathcal{I}^{\cdot,\delta}_{\substack{s,s+k-1 \\ (M_k,J_k)}} = \mathcal{I}^{\cdot,\delta}_{\substack{s,s+k-1 \\ (M_k,J_k)}}(U^c_{J_k},E^c_{J_k}) + \text{remainder terms}$. More explicitely, each operator in the iterated elementary terms will be decomposed as
$$
\mathcal{I}^{0,\delta}_{\substack{s \\ \pm,j}} = \mathcal{I}^{0,\delta}_{\substack{s \\ \pm,j}}(U_j) + \mathcal{I}^{0,\delta}_{\substack{s \\ \pm,j}}(U_j^c,E) + \mathcal{I}^{0,\delta}_{\substack{s \\ \pm,j}}(U_j^c,E_j^c),
$$
where here only the last term will not be a remainder.\\
Concerning the truncated in pathological adjunction parameters of the BBGKY hierarchy, it is important here to notice that the elementary terms \emph{are not defined as usual iterated integrals}. The same work as in Section \ref{SSect2.1__DefinOperaColliBBGKY} has to be done, and the operator obtained in the limit will behave like an integral, in the sense that for two subsets $E_1$ and $E_2$ of adjunction parameters in $\mathbb{S}^{d-1}\times\mathbb{R}^d$ such that $\vert E_1 \cap E_2 \vert = 0$, one has $\mathcal{I}^{N,\varepsilon,\delta}_{\substack{s \\ \pm,j}}(E_1 \cap E_2) = \mathcal{I}^{N,\varepsilon,\delta}_{\substack{s \\ \pm,j}}(E_1) + \mathcal{I}^{N,\varepsilon,\delta}_{\substack{s \\ \pm,j}}(E_2)$, enabling in particular the surgery on pathological adjunction parameters. For more details on the definition of those objects, the reader may refer to \cite{PhDTT}\footnote{In particular, see the Section ``Rigorous definition of the truncated in adjunction parameters, integrated in time transport-collision-transport of the BBGKY hierarchy", starting page 416.}.
\end{remar}

Let us now introduce the relevant notations describing this surgery. We define the \emph{elementary BBGKY term}, denoted $\mathcal{J}^{N,\varepsilon,\delta}_{\substack{s,s+k-1 \\ M_k,J_k}}(U^c,E^c) f^{(s+k)}_{N,0}$, as the function
$$
(t,Z_s) \mapsto \mathcal{T}^{s,\varepsilon}_t \Big( \mathcal{I}^{N,\varepsilon,\delta}_{\substack{s,s+k-1 \\ M_k,J_k}}(U^c_{J_k},E^c_{J_k}) f^{(s+k)}_{N,0}(t,Z_s) \Big)
$$
and the \emph{elementary Boltzmann term}, denoted $\mathcal{J}^{0,\delta}_{\substack{s,s+k-1 \\ M_k,J_k}}(U^c,E^c) f^{(s+k)}_0$, as the function
$$
(t,Z_s) \mapsto \mathcal{I}^{0,\delta}_{\substack{ s,s+k-1 \\ M_k,J_k }}(U^c_{J_k},E^c_{J_k})\Big( u \mapsto \mathcal{T}^{s+k,0}_u f^{(s+k)}_0 \Big)(t,Z_s),
$$
where $U$ and $E$ are the notations for the families of subsets $(U_1,\dots,U_k) \subset [0,T]^k$ and $(E_1,\dots,E_k) \subset \big( \mathbb{S}^{d-1} \times \mathbb{R}^d \big)^k$. With those elementary terms, taking into account the surgeries in adjunction parameters, we can constitute approximations of the solutions:
\begin{align*}
F^{n,R,\delta}_N&(U^c,E^c) = t \mapsto \Big( \mathcal{T}^{s,\varepsilon}_t f^{(s)}_{N,0}(\cdot) \mathds{1}_{\vert V_s \vert \leq R} + \sum_{k=1}^n \mathds{1}_{s \leq N-k} \hspace{-3mm} \sum_{M_k \in \mathfrak{M}_k} \sum_{J_k \in \mathfrak{M}^s_k} \mathcal{J}^{N,\varepsilon,\delta}_{\substack{s,s+k-1 \\ M_k,J_k}}(U^c,E^c) \big( f^{(s+k)}_{N,0} \mathds{1}_{\vert V_{s+k} \vert \leq R} \big)(t,\cdot) \Big)_{1 \leq s \leq N}
\end{align*}
for the BBGKY hierarchy, and
\begin{align*}
F^{n,R,\delta}&(U^c,E^c) = t \mapsto \Big( \mathcal{T}^{s,0}_t f^{(s)}_0(\cdot) \mathds{1}_{\vert V_s \vert \leq R} + \sum_{k=1}^n \sum_{M_k \in \mathfrak{M}_k} \sum_{J_k \in \mathfrak{J}^s_k} \mathcal{J}^{0,\delta}_{\substack{s,s+k-1 \\ M_k,J_k}}(U^c,E^c) \big( f^{(s+k)}_0 \mathds{1}_{\vert V_{s+k} \vert \leq R} \big)(t,\cdot) \Big)_{s \geq 1}
\end{align*}
for the Boltzmann hierarchy.\\
The purpose of the following lemma is then to measure the error produced by the surgery in the pathological adjunction parameters.

\begin{lemma}[Surgery with the adjunction parameters for the elementary terms of the hierarchies]
\label{SECT3LemmeChiruParamAdjonPatho}
Let $\beta_0>0$ and $\mu_0$ be two real numbers. Then, there exist two strictly positive constants $c(d)$ and $C_4(d,\beta_0,\mu_0)$ such that the following holds.\\
Let $N$ and $n$ be two positive integers, $\varepsilon$, $R$ and $\delta$ be three strictly positive numbers such that the Boltzmann-Grad limit $N\varepsilon^{d-1} = 1$ holds. For any pair $F_{N,0} = (f^{(s)}_{N,0})_{1 \leq s \leq N}$ and $F_0 = (f^{(s)}_0)_{s \geq 1}$ of sequences of initial data belonging respectively to $\textbf{X}_{N,\varepsilon,\beta_0,\mu_0}$ and $\textbf{X}_{0,\beta_0,\mu_0}$, and any strictly positive numbers $a$, $\varepsilon_0$, $\rho$, $\eta$ and $\alpha$ such that $R \geq 1, \eta \leq 1, 2\varepsilon \leq a, 4\sqrt{3}a \leq \varepsilon_0, \varepsilon_0 \leq \eta\delta, 3a \leq \rho, \alpha \leq c(d)$, one has for the BBGKY hierarchy:
\begin{align}
&\mathds{1}_{\Delta_s}(Z_s) \Big\vert \mathcal{T}^{s,\varepsilon}_t\big( H^{n,R,\delta}_N \big)^{(s)} - \big(F^{n,R,\delta}_N(U^c,E^c)\big)^{(s)} \Big\vert (t,Z_s) \exp\Big(\frac{\widetilde{\beta}_\lambda(t)}{2} \vert V_s \vert^2 \Big) \nonumber\\
&\hspace{10mm} \leq C_4 n(s+n)R \Big( \frac{\rho}{\alpha} + \eta^d + R^d\Big(\frac{a}{\rho}\Big)^{d-1} \hspace{-3mm} + nR^{2d-1}\Big(\frac{a}{\varepsilon_0}\Big)^{d-3/2} \hspace{-3mm} + nR^{d+1/2}\Big(\frac{\varepsilon_0}{\delta}\Big)^{d-3/2}\nonumber\\
&\hspace{89mm}  + R^{d-1}\alpha + R^d\alpha^{1/8}\Big) \vertii{(f^{(s)}_{N,0})_{1 \leq s \leq N}}_{N,\varepsilon,\beta_0,\mu_0}
\end{align}
and for the Boltzmann hierarchy:
\begin{align}
&\mathds{1}_{\Delta_s}(Z_s) \Big\vert \big( F^{n,R,\delta}\big)^{(s)} - \big(F^{n,R,\delta}(U^c,E^c)\big)^{(s)} \Big\vert (t,Z_s) \exp\Big(\frac{\widetilde{\beta}_\lambda(t)}{2} \vert V_s \vert^2 \Big) \nonumber\\
&\hspace{10mm} \leq C_4 n(s+n)R \Big( \frac{\rho}{\alpha} + \eta^d + R^d\Big(\frac{a}{\rho}\Big)^{d-1} \hspace{-3mm} + nR^{2d-1}\Big(\frac{a}{\varepsilon_0}\Big)^{d-3/2} \hspace{-3mm} + nR^{d+1/2}\Big(\frac{\varepsilon_0}{\delta}\Big)^{d-3/2}\nonumber\\
&\hspace{95mm}  + R^{d-1}\alpha + R^d\alpha^{1/8}\Big) \vertii{(f^{(s)}_0)_{s \geq 1}}_{0,\beta_0,\mu_0}
\end{align}
with $\Delta_s = \Delta_s\big(\varepsilon,R,\varepsilon_0,\max(16Ra/\varepsilon_0,\varepsilon_0/\delta)\big)$, and $U$ and $E$ being the two families of pathological adjunction parameters described in Proposition \ref{SECT3PropoControParamAdjonPatho}.
\end{lemma}

The proof of Lemma \ref{SECT3LemmeChiruParamAdjonPatho} is presented in \cite{PhDTT}.\\
\newline
Now it is time to obtain the crucial result of this work. Can we show that the convergence of the pseudo-trajectories implies the convergence of the solutions of the hierarchies? Relying on the regularity of the initial data, the following result answers affirmatively the question, casting a decisive bridge between the two hierarchies.\\
For this result, let us investigate the explicit expressions of the elementary terms. For the Boltzmann hierarchy, the final term $\mathcal{J}^{0,\delta}_{\substack{s,s+k-1 \\ M_k,J_k}}(U^c,E^c) f^{(s+k)}_0$ writes:
\begin{align*}
&\mathds{1}_{t \geq k\delta} \int_{\hspace{-0.75mm}\substack{\vspace{-2mm}(k-1)\delta}}^{\hspace{0.25mm}\substack{\vspace{1mm}t-\delta}} \hspace{-6mm} \mathds{1}_{U^c_{j_1}(Z_s)} \hspace{-1mm} \int_{\hspace{-0.5mm}\substack{\vspace{-4mm}\mathbb{S}^{d-1}_{\omega_1}\times\mathbb{R}^d_{v_{s+1}}}} \hspace{-13mm} (\pm_1) \mathds{1}_{E^c_{j_1}(Z^0_{s,0}(t_1))} \Big[ \omega_1 \cdot\big(v_{s+1}-v^{0,j_1}_{s,0}(t_1)\big)\Big]_{\pm_1}   \\[8pt]
&\times\mathds{1}_{t_1 \geq (k-1)\delta} \int_{\hspace{-0.75mm}\substack{\vspace{-2mm}(k-2)\delta}}^{\hspace{0.25mm}\substack{\vspace{1mm}t_1-\delta}} \hspace{-6mm} \mathds{1}_{U^c_{j_2}(Z^0_{s,1}(t_1))} \hspace{-1mm} \int_{\hspace{-0.5mm}\substack{\vspace{-4mm}\mathbb{S}^{d-1}_{\omega_2} \times \mathbb{R}^d_{v_{s+2}}}} \hspace{-13mm} (\pm_2) \mathds{1}_{E^c_{j_2}(Z^0_{s,1}(t_2))} \Big[ \omega_2 \cdot \big( v_{s+2} - v^{0,j_2}_{s,1}(t_2) \big) \Big]_{\pm_2} \\
&\dots \\
&\times \mathds{1}_{\hspace{-0.75mm}\substack{\vspace{-2mm}t_{k-1}\geq\delta}} \int_0^{\hspace{0.25mm}\substack{\vspace{1mm}t_{k-1}-\delta}} \hspace{-9mm} \mathds{1}_{U^c_{j_k}(Z^0_{s,k-1}(t_{k-1}))} \int_{\hspace{-0.5mm}\substack{\vspace{-4mm}\mathbb{S}^{d-1}_{\omega_k} \times \mathbb{R}^d_{v_{s+k}}}} \hspace{-13mm} (\pm_k) \mathds{1}_{E^c_{j_k}(Z^0_{s,k-1}(t_k))} \Big[ \omega_k \cdot \big( v_{s+k} - v^{0,j_k}_{s,k-1}(t_k) \big) \Big]_{\pm_k} \\[5pt]
&\hspace{2mm} \times \mathds{1}_{\vert V_{s+k} \vert \leq R} f^{(s+k)}_0 \big(Z^0_{s,k}(0)\big) \dd \omega_k \dd v_{s+k} \dd t_k \dots \dd \omega_2 \dd v_{s+2} \dd t_2 \dd \omega_1 \dd v_{s+1} \dd t_1.
\end{align*}
Concerning the BBGKY hierarchy, one recalls that the terms are \emph{not} defined as usual integrals. However, thanks to the surgery in adjunction parameters, we will see that a new writing can be provided for those terms. But there are two remaining differences between the elementary terms of the two hierarchies: a different sequence of initial data, and the presence of \emph{prefactors} (the product of the $(N-s)\varepsilon^{d-1}$ terms). Both points will be addressed in the next section, so here let us introduce finally an hybrid version of the elementary term of the BBGKY hierarchy, getting rid of those differences, denoted $\overline{\mathcal{J}}^{\varepsilon,\delta}_{\substack{s,s+k-1 \\ M_k,J_k}}(U^c,E^c)$, and defined as:
\begin{align*}
&\mathds{1}_{t \geq k\delta} \int_{\hspace{-0.75mm}\substack{\vspace{-2mm}(k-1)\delta}}^{\hspace{0.25mm}\substack{\vspace{1mm}t-\delta}} \hspace{-6mm} \mathds{1}_{U^c_{j_1}(Z_s)} \hspace{-1mm} \int_{\hspace{-0.5mm}\substack{\vspace{-4mm}\mathbb{S}^{d-1}_{\omega_1}\times\mathbb{R}^d_{v_{s+1}}}} \hspace{-13mm} (\pm_1) \mathds{1}_{E^c_{j_1}(Z^0_{s,0}(t_1))} \Big[ \omega_1 \cdot\big(v_{s+1}-v^{\varepsilon,j_1}_{s,0}(t_1)\big)\Big]_{\pm_1}   \\[8pt]
&\times\mathds{1}_{t_1 \geq (k-1)\delta} \int_{\hspace{-0.75mm}\substack{\vspace{-2mm}(k-2)\delta}}^{\hspace{0.25mm}\substack{\vspace{1mm}t_1-\delta}} \hspace{-6mm} \mathds{1}_{U^c_{j_2}(Z^0_{s,1}(t_1))} \hspace{-1mm} \int_{\hspace{-0.5mm}\substack{\vspace{-4mm}\mathbb{S}^{d-1}_{\omega_2} \times \mathbb{R}^d_{v_{s+2}}}} \hspace{-13mm} (\pm_2) \mathds{1}_{E^c_{j_2}(Z^0_{s,1}(t_2))} \Big[ \omega_2 \cdot \big( v_{s+2} - v^{\varepsilon,j_2}_{s,1}(t_2) \big) \Big]_{\pm_2} \\
&\dots \\
&\times \mathds{1}_{\hspace{-0.75mm}\substack{\vspace{-2mm}t_{k-1}\geq\delta}} \int_0^{\hspace{0.25mm}\substack{\vspace{1mm}t_{k-1}-\delta}} \hspace{-9mm} \mathds{1}_{U^c_{j_k}(Z^0_{s,k-1}(t_{k-1}))} \int_{\hspace{-0.5mm}\substack{\vspace{-4mm}\mathbb{S}^{d-1}_{\omega_k} \times \mathbb{R}^d_{v_{s+k}}}} \hspace{-13mm} (\pm_k) \mathds{1}_{E^c_{j_k}(Z^0_{s,k-1}(t_k))} \Big[ \omega_k \cdot \big( v_{s+k} - v^{\varepsilon,j_k}_{s,k-1}(t_k) \big) \Big]_{\pm_k} \\[5pt]
&\hspace{2mm} \times \mathds{1}_{\vert V_{s+k} \vert \leq R} f^{(s+k)}_0 \big(Z^\varepsilon_{s,k}(0)\big) \dd \omega_k \dd v_{s+k} \dd t_k \dots \dd \omega_2 \dd v_{s+2} \dd t_2 \dd \omega_1 \dd v_{s+1} \dd t_1
\end{align*}
(the only difference between $\mathcal{J}^{0,\delta}_{\substack{s,s+k-1 \\ M_k,J_k}}(U^c,E^c)$ and $\overline{\mathcal{J}}^{\varepsilon,\delta}_{\substack{s,s+k-1 \\ M_k,J_k}}(U^c,E^c)$ lies in the choice of the pseudo-trajectories).

\begin{lemma}[Error coming from the divergence of the trajectories]
\label{SECT3LemmeErreuVenanDiverTraje}
Let $s$ and $n$ be two positive integers, $\beta_0>0$ and $\mu_0$ be two real numbers. Then, there exists a time $T'$ satisfying $0 \leq T' < T$ (where $T$ is given by Theorem \ref{SECT2TheorExistUniciSolutHiera}) such that, for any nonnegative, normalized (for the $L^1$ norm) function $f_0 \in X_{0,1,\beta_0}$ with $\sqrt{f_0}$ being Lipschitz-continuous with respect to the position variable uniformly in the velocity variable, and any compact set $K$ of the domain of local uniform convergence $\Omega_s$, there exist five strictly positive numbers $\overline{\varepsilon}$, $\overline{\varepsilon}_0$, $\overline{\alpha}$, $\overline{\gamma}$ and $\overline{R}$ (depending only on $K$) such that for every strictly positive numbers $R$, $\delta$, $\varepsilon$, $a$, $\varepsilon_0$, $\rho$, $\eta$ and $\alpha$ which satisfy $\varepsilon \leq \overline{\varepsilon}, \varepsilon_0 \leq \overline{\varepsilon}_0, \alpha \leq \overline{\alpha}, \max\big(16R\varepsilon/\varepsilon_0,\varepsilon_0/\delta\big) \leq \overline{\gamma}, R \geq \overline{R}$ and $4\sqrt{3} \leq \varepsilon_0, 3a \leq \rho, \varepsilon_0 \leq \eta\delta, R \geq 1, \eta \leq 1, \alpha \leq c(d)$ and finally $2n\varepsilon \leq a$, then in the Boltzmann-Grad limit $N\varepsilon^{d-1} = 1$, for all the parameters fixed but $\varepsilon$ and $N$, one has the following uniform convergence on $K$ and on $[0,T']$, with the explicit control on the rate of convergence:
\begin{align*}
\Big\vert \kern-0.3ex \Big\vert \mathds{1}_K \sum_{k=1}^n \sum_{M_k \in \mathfrak{M}_k} \sum_{J_k \in \mathfrak{J}^s_k} \Big(& \mathcal{J}^{0,\delta}_{\substack{s,s+k-1 \\ M_k,J_k }}(U^c,E^c)\big( f^{\otimes (s+k)}_0 \mathds{1}_{\vert V_{s+k} \vert \leq R} \big) - \overline{\mathcal{J}}^{\varepsilon,\delta}_{\substack{s,s+k-1 \\ M_k,J_k}}(U^c,E^c) \big( f^{\otimes (s+k)}_0 \mathds{1}_{\vert V_{s+k} \vert \leq R} \big) \Big) (Z_s) \Big\vert \kern-0.3ex \Big\vert_{L^\infty} \\
&\hspace{60mm} \leq C(d) (s+n)^2nR\frac{\varepsilon}{\alpha} \big\vert \nabla_x \sqrt{f_0} \big\vert_\infty \vertii{ F_0 }_{0,\beta_0/2,\mu_0}.
\end{align*}
\end{lemma}

\begin{proof}[Proof of Lemma \ref{SECT3LemmeErreuVenanDiverTraje}]
The first step consists in noticing that the elementary term $\overline{\mathcal{J}}^{\varepsilon,\delta}_{\substack{s,s+k-1 \\ M_k,J_k}}$ is defined as a usual integral, thanks to the first point of Proposition \ref{SECT3PropoControParamAdjonPatho}: after surgery, the pseudo-trajectories of the BBGKY hierarchy do not produce any recollision, and therefore between two adjunctions they are described only using the free transport and symmetries. It is now possible to apply the plan used in Section \ref{SSect2.1__DefinOperaColliBBGKY} (relying on the fact that the hard sphere transport coincides locally with the free transport) to give a sense to the collision operator, but this time without any time restriction (since there is no recollision to be avoided, so that the hard sphere transport coincides with the free transport, globally in time).\\
\newline
The rest of the proof relies on the explicit expression of $\mathcal{J}^{0,\delta}_{\substack{s,s+k-1 \\ M_k,J_k }} - \overline{\mathcal{J}}^{\varepsilon,\delta}_{\substack{s,s+k-1 \\ M_k,J_k}}$.\\
In particular, the velocities of the particles $j_l$ chosen for the adjunctions are the same at the times of adjunction $t_j$ for the two pseudo-trajectories, that is $v^{\varepsilon,j_l}_{s,l-1}(t_l) = v^{0,j_l}_{s,l-1}(t_l)$ for all $1 \leq l \leq k$, as a consequence of the second point of Proposition \ref{SECT3PropoControParamAdjonPatho}, together with the fact that the adjunctions happen only when the particle chosen is far from the wall, and the relation between $\varepsilon$ and $\rho$: either the particles $j_l$ of both hierarchies have bounced against the wall at $t_l$, or none of them has.\\
Then, the only difference between the expressions of $\mathcal{J}^{0,\delta}_{\substack{s,s+k-1 \\ M_k,J_k }}$ and $\overline{\mathcal{J}}^{\varepsilon,\delta}_{\substack{s,s+k-1 \\ M_k,J_k}}$ lies in $Z^0_{s,k}(0)$ and $Z^\varepsilon_{s,k}(0)$, the respective arguments of the function $f^{(s+k)}_0$, corresponding to the final configurations of the two pseudo-trajectories. Taking now into account that the velocities for the two pseudo-trajectories at time $0$ may differ by the symmetry $\mathcal{S}_0$, we denote $j_1,\dots,j_p$ the labels of the particles having different velocities at time $0$, and $t^{j_l}_0$ (for $1 \leq l \leq p$) the times such that $x^{0,j_l}_{s,k}(t^{j_l}_0)\cdot e_1 = 0$. Let us describe the process for the easy case $s+k=2$ to fix the ideas. We decompose:
\begin{align}
f^{\otimes 2}_0\big(X^0_{s,k}(0),V^0_{s,k}(0)\big) - f^{\otimes 2}_0\big(&X^\varepsilon_{s,k}(0),V^\varepsilon_{s,k}(0)\big) \nonumber\\
=&\, \Big( f_0\big(x^{0,1}_{s,k}(0),v^{0,1}_{s,k}(0)\big) - f_0\big(x^{0,1}_{s,k}(t^{j_1}_0),v^{0,1}_{s,k}(0)\big) \Big) f_0\big(x^{0,2}_{s,k}(0),v^{0,2}_{s,k}(0)\big) \label{SECT3PREUVErreuDiverTrajeLign1}\\
&+ f_0\big(x^{0,1}_{s,k}(t^{j_1}_0),v^{0,1}_{s,k}(0)\big) \Big( f_0\big(x^{0,2}_{s,k}(0),v^{0,2}_{s,k}(0)\big) - f_0\big(x^{0,2}_{s,k}(t^{j_2}_0),v^{0,2}_{s,k}(0)\big)\Big) \label{SECT3PREUVErreuDiverTrajeLign2}\\
&+ f_0\big(x^{0,1}_{s,k}(t^{j_1}_0),v^{0,1}_{s,k}(0)\big) f_0\big(x^{0,2}_{s,k}(t^{j_2}_0),v^{0,2}_{s,k}(0)\big) - f_0^{\otimes 2}\big(X^\varepsilon_{s,k}(0),V^\varepsilon_{s,k}(0)\big), \label{SECT3PREUVErreuDiverTrajeLign3}
\end{align}
the terms \eqref{SECT3PREUVErreuDiverTrajeLign1} and \eqref{SECT3PREUVErreuDiverTrajeLign2} being controlled respectively by $\vert x^{0,1}_{s,k}(0) - x^{0,1}_{s,k}(t^{j_1}_0) \vert$ and $\vert x^{0,2}_{s,k}(0) - x^{0,2}_{s,k}(t^{j_2}_0) \vert$, while we can use the boundary conditions for the term \eqref{SECT3PREUVErreuDiverTrajeLign3}, replacing $v^{0,l}_{s,k}(0)$ by $v^{\varepsilon,l}_{s,k}(0)$, and then using the divergence between the pseudo-trajectories (second point of Proposition \ref{SECT3PropoControParamAdjonPatho}).\\
In order to make those arguments quantitative, considering that $t^{j_l}_0 \leq (2k\varepsilon + \varepsilon/2)/\alpha$ (because $t=0$ lies in the interval delimited by $t^{j_l}_0$ and $t^{j_l}_\varepsilon$, and $\vert t^{j_1}_\varepsilon - t^{j_l}_0 \vert \leq (2k+1/2)\varepsilon/\alpha$), we have $\vert x^{0,j_l}_{s,k}(0) - x^{0,j_l}_{s,k}(t^{j_l}_0) \vert \leq (2k+1/2)\varepsilon R/\alpha$, and writing:
\begin{align*}
&\Big\vert \Big(f_0\big(x^{0,1}_{s,k}(0),v^{0,1}_{s,k}(0)\big) - f_0\big(x^{0,1}_{s,k}(t^{j_1}_0),v^{0,1}_{s,k}(0)\big)\Big) \prod_{l=2}^{s+k} f_0\big(x^{0,l}_{s,k}(0),v^{0,l}_{s,k}(0)\big) \Big\vert \\
&\leq \Big( \sqrt{f_0\big(x^{0,1}_{s,k}(0),v^{0,1}_{s,k}(0)\big)} + \sqrt{f_0\big(x^{0,1}_{s,k}(t^{j_1}_0),v^{0,1}_{s,k}(0)\big)} \Big) \Big\vert \big\vert x^{0,1}_{s,k}(0) - x^{0,1}_{s,k}(t^{j_1}_0) \big\vert \big\vert \nabla_x \sqrt{f_0} \big\vert_\infty \Big\vert \prod_{l=2}^{s+k} f_0\big(x^{0,l}_{s,k}(0),v^{0,l}_{s,k}(0)\big) \Big\vert,
\end{align*}
implies (because $\sqrt{f_0} \in X_{0,1,\beta_0/2}$)
\begin{align*}
\big\vert f^{(s+k)}&\big(Z^0_{s,k}(0)\big) - f^{(s+k)}_0\big(Z^\varepsilon_{s,k}(0)\big) \big\vert \\
&\leq 2 \big\vert \nabla_x \sqrt{f_0} \big\vert_\infty (s+k)\Big(2(s+k) \frac{2k+1/2}{\alpha}R\varepsilon + k\varepsilon\Big) \exp \Big(-\frac{\beta_0}{4}\vert v^{0,1}_{s,k}(0) \vert^2\Big)\exp\Big(-\frac{\beta_0}{4} \sum_{l=2}^{s+k} \vert v^{0,l}_{s,k}(0) \vert^2 \Big).
\end{align*}
The conclusion is now obtained thanks to the conservation of the kinetic energy along the pseudo-trajectories $\exp \Big(-\frac{\beta_0}{4}\vert v^{0,1}_{s,k}(0) \vert^2\Big)\exp\Big(-\frac{\beta_0}{4} \sum_{l=2}^{s+k} \vert v^{0,l}_{s,k}(0) \vert^2 \Big) = \exp\Big( -\frac{\beta_0}{4} \vert V_{s+k} \vert^2 \Big)$, and the contracting property of the integrated in time transport-collision operators. Lemma \ref{SECT3LemmeErreuVenanDiverTraje} is proven.
\end{proof}

\section{The quantified main result: the Lanford's theorem in the half-space}

To complete the proof of Lanford's result and provide a quantitative version of Theorem \ref{SECT3TheorLanford___Vers_Quali}, it remains to evaluate the error produced with the replacement of the elementary terms $\mathcal{J}^{N,\varepsilon,\delta}_{\substack{s,s+k-1 \\ M_k,J_k}}$ by the hybrid ones $\overline{\mathcal{J}}^{\varepsilon,\delta}_{\substack{s,s+k-1 \\ M_k,J_k}}$. We recall that the two differences between those terms lie in the sequence of initial data and on the absence of the prefactor.\\
\newline
Let us start this last section with a result concerning the \emph{admissible Boltzmann initial data}, proving that the tensorized initial data $\big((f_0^{\otimes s})_{s \geq 1}\big)_N$ of the Boltzmann hierarchy enable to define an associated sequence of initial data $\big(f_{N,0}^{(s)}\big)_{1 \leq s \leq N}$ for the BBGKY hierarchy (for any numbers of particles $N$) which converges, for all $s$, locally uniformly on $\Omega_s$. For more details, the reader may refer to \cite{PhDTT}\footnote{See Section 16.1 starting page 503.}, or especially to \cite{GSRT}\footnote{See in particular Section 6.1 "Quasi-independence", starting page 43.} for a detailed discussion and the proof of the following result.

\begin{propo}[The tensorized initial data are Boltzmann admissible]
\label{SECT3PropoDonnéTensoBoltzAdmis}
Let $\beta_0 > 0$ and $\mu_0$ be two real numbers. For any nonnegative, normalized function $f_0$ belonging to $X_{0,1,\beta_0}$ such that $e^{\mu_0} \vert f_0 \vert_{0,1,\beta_0} \leq 1$, the sequence of chaotic configurations $\big(f^{\otimes s}_0\big)_{s \geq 1}$ (defined for all $s \geq 1$ as $f^{\otimes s}_0(Z_s) = \prod_{i=1}^s f_0(z_i)$) is a sequence of admissible Boltzmann data, that is, if for any $N \in \mathbb{N}^*$ we define:
$$
f^{(N)}_{N,0}: Z_N \mapsto \mathcal{Z}^{-1}_N \mathds{1}_{\mathcal{D}^\varepsilon_N} f^{\otimes N}_0(Z_N)
$$
with $\displaystyle{\mathcal{Z}_N = \int_{\hspace{-0.25mm}\substack{\vspace{-0.5mm}\mathbb{R}^{2dN}}} \hspace{-5mm} \mathds{1}_{\mathcal{D}^\varepsilon_N} f^{\otimes N}_0 \dd Z_N}$, and $f^{(s)}_{N,0}(Z_s) = \displaystyle{\int}_{\hspace{-2mm}\substack{\vspace{-0.75mm}\mathbb{R}^{2d(N-s)}}} \hspace{-9mm} \mathds{1}_{\mathcal{D}^\varepsilon_N} f^{(N)}_{N,0} \dd z_{s+1} \dots \dd z_N$, then we have:
\begin{itemize}[leftmargin=*]
\item for any $N \in \mathbb{N}^*$ with $N\varepsilon^{d-1} = 1$, the sequence of initial data $F_{N,0} = \big(f^{(s)}_{N,0}\big)_{1 \leq s \leq N}$ belongs to $\textbf{X}_{N,\varepsilon,\beta_0,\mu_0}$,
\item $\sup_{N \geq 1} \vertii{ \big(f^{(s)}_{N,0}\big)_{1 \leq s \leq N} }_{N,\varepsilon,\beta_0,\mu_0} < +\infty$,
\item for any positive integer $s$, we have $f^{(s)}_{N,0} \underset{N \rightarrow + \infty}{\longrightarrow} f^{(s)}_0$ locally uniformly on $\Omega_s$, with the explicit rate of convergence:
\begin{equation}
\label{SECT3PROPOTensoBoltzAdmisQuant}
\big\vert \mathds{1}_{\mathcal{D}^\varepsilon_s}f^{\otimes s}_0 - f^{(s)}_{N,0} \big\vert_\infty \leq C(d)s\varepsilon \vert f_0 \vert_{L^\infty(\overline{\Omega^c},L^1(\mathbb{R}^d))}.
\end{equation}
\end{itemize}
\end{propo}

It is now possible to draw the final link between the solutions of the two hierarchies: the next lemma quantifies the difference $\mathcal{J}^{N,\varepsilon,\delta}_{\substack{s,s+k-1 \\ M_k,J_k}} - \overline{\mathcal{J}}^{\varepsilon,\delta}_{\substack{s,s+k-1 \\ M_k,J_k}}$.

\begin{lemma}[Error coming from the substitution of the initial data and the prefactors]
\label{SECT4LemmeErreuVenanDInitPrefa}
Let $s$ and $n$ be two positive integers, $\beta_0>0$ and $\mu_0$ be two reals numbers. Let $f_0$ be a nonnegative, normalized function belonging to $X_{0,1,\beta_0}$ and let $K$ be a compact set of the domain of local uniform convergence $\Omega_s$.\\
Then, there exists five strictly positive numbers $\overline{\varepsilon}$, $\overline{\varepsilon}_0$, $\overline{\alpha}$, $\overline{\gamma}$ and $\overline{R}$, depending only on $K$, such that for every strictly positive numbers $R$, $\delta$, $\varepsilon$, $a$, $\varepsilon_0$, $\rho$, $\eta$ and $\alpha$ which satisfy $\varepsilon \leq \overline{\varepsilon}, \varepsilon_0 \leq \overline{\varepsilon}_0, \alpha \leq \overline{\alpha}, \max\big(16R\varepsilon/\varepsilon_0,\varepsilon_0/\delta\big) \leq \overline{\gamma}$ and $R \geq \overline{R}$, and $4\sqrt{3} \leq \varepsilon_0, 3a \leq \rho, \varepsilon_0 \leq \eta\delta, R \geq 1, \eta \leq 1, \alpha \leq c(d)$ and finally $2n\varepsilon \leq a$, we have in the Boltzmann-Grad limit $N\varepsilon^{d-1} = 1$ the uniform convergence on the compact set $K$, uniform on the time interval $[0,T]$, of the sequence of the sum of the elementary BBGKY terms with the sequence of initial data $\big(F_{N,0}\big)_{N \geq 0}$, associated to the sequence of tensorized initial data $\big(f^{\otimes s}_0\big)_{s \geq 1}$, towards the sum of the hybrid terms with the sequence of tensorized initial data $\big(f^{\otimes s}_0\big)_{s \geq 1}$, that is:
\begin{align*}
\Big\vert\kern-0.3ex\Big\vert \mathds{1}_K \sum_{k=1}^n \sum_{M_k \in \mathfrak{M}_k} \sum_{J_k \in \mathfrak{J}^s_k} \Big(  &\overline{\mathcal{J}}^{\varepsilon,\delta}_{\substack{s,s+k-1 \\ M_k,J_k}}(U^c,E^c) \big( f^{\otimes(s+k)}_0 \mathds{1}_{\vert V_{s+k} \vert \leq R} \big) - \mathcal{J}^{N,\varepsilon,\delta}_{\substack{s,s+k-1 \\ M_k,J_k}}(U^c,E^c) \big( f^{(s+k)}_{N,0} \mathds{1}_{\vert V_{s+k} \vert \leq R} \big)  \Big) \Big\vert\kern-0.3ex\Big\vert_{L^\infty} \\
&\hspace{60mm} \leq C(d) (s+n)^2 \varepsilon \vert f_0 \vert_{L^\infty(\overline{\Omega^c},L^1(\mathbb{R}^d))} \vertii{ F_0 }_{0,\beta_0,\mu_0}.
\end{align*}
\end{lemma}

\begin{proof}[Proof of Lemma \ref{SECT4LemmeErreuVenanDInitPrefa}]
This result is a direct consequence of the control \eqref{SECT3PROPOTensoBoltzAdmisQuant} of Proposition \ref{SECT3PropoDonnéTensoBoltzAdmis}, together with the following bounds on the prefactors:
\begin{align*}
1 \geq \frac{(N-s)!}{(N-s-k)!}\varepsilon^{k(d-1)} \geq \big( (N-s-k)\varepsilon^{d-1} \big)^k &\geq 1 + k \ln(N-s-k/N) \geq 1 - 2k\frac{(s+k)}{N} \geq 1 - 2\frac{(s+n)^2}{N}\cdotp
\end{align*}
for $N$ large enough (with respect to $s$ and $n$, that are fixed here).
\end{proof}

\begin{theor}[Lanford's theorem in the half-space]
\label{SECT4TheorPrincipal_deLanford_}
Let $\beta_0>0$ and $\mu_0$ be two real numbers. There exist two times $0 \leq T' < T$ such that the following holds:\\
let $f_0$ be a nonnegative normalized function belonging to $X_{0,1,\beta_0}$ which satisfies $\vert f_0 \vert_{0,1,\beta_0} \leq \exp(-\mu_0)$ and such that $\sqrt{f_0}$ is Lipschitz with respect to the position variable uniformly in the velocity variable.\\
Then if one considers the solution $F = \big( f^{(s)} \big)_{s \geq 1}$ on $[0,T]$ of the Boltzmann hierarchy with the tensorized initial datum $F_0 = \big( f^{\otimes s}_0 \big)_{s \geq 1}$, and if one considers for every positive integer $N$ the solution $F_N = \big(f^{(s)}_N\big)_{1 \leq s \leq N}$ on $[0,T]$ of the BBGKY hierarchy with the initial datum $F_{N,0} = \big( f^{(s)}_{N,0} \big)_{1 \leq s \leq N}$, where $\big(F_{N,0}\big)_{N \geq 0}$ is the sequence of initial data associated to the tensorized initial datum $\big( f^{\otimes s}_0 \big)_{s \geq 1}$ (in the sense of Proposition \ref{SECT3PropoDonnéTensoBoltzAdmis}), one has that, in the Boltzmann-Grad limit $N\varepsilon^{d-1} = 1$, for every positive integer $s$, the following locally uniform convergence on the domain of local uniform convergence $\Omega_s$, uniform on the time interval $[0,T']$, holds:
$$
f^{(s)}_N \underset{N \rightarrow +\infty}{\longrightarrow} f^{(s)},
$$
moreover, whatever the dimension $d$ is, the rate of convergence is of order $O(\varepsilon^\gamma)$, for all $\gamma \in\ ]0,13/128[$.
\end{theor}

\begin{proof}[Proof of Theorem \ref{SECT4TheorPrincipal_deLanford_}]
Collecting all the errors obtained along the different cut-offs, if one can express the parameters $a,\varepsilon_0,\eta,\rho,\alpha,\delta,R$ and $n$ as explicit functions of $\varepsilon$, fulfilling the conditions:
\begin{align*}
n\varepsilon/a \underset{\varepsilon \rightarrow 0}{\longrightarrow} 0,\ \ a/\varepsilon_0 \underset{\varepsilon \rightarrow 0}{\longrightarrow} 0,\ \ \varepsilon_0/&(\eta \delta) \underset{\varepsilon \rightarrow 0}{\longrightarrow} 0,\ \ \varepsilon_0/\delta \underset{\varepsilon \rightarrow 0}{\longrightarrow} 0,\ \ a/\rho \underset{\varepsilon \rightarrow 0}{\longrightarrow} 0,\ \ \alpha \underset{\varepsilon \rightarrow 0}{\longrightarrow} 0 \ \ \text{and}\ \ R\underset{\varepsilon \rightarrow 0}{\longrightarrow}+\infty,
\end{align*}
we would recover an explicit rate of convergence of the solutions of the BBGKY hierarchy towards the solution of the Boltzmann hierarchy.\\
To do so, we may choose
\begin{align*}
\delta = \varepsilon^{1/2}& \vert \ln(\varepsilon) \vert^3,\ \ a= \varepsilon \vert \ln(\varepsilon) \vert^2,\ \ \varepsilon_0 = \varepsilon^{3/4} \vert \ln(\varepsilon) \vert^2,\ \ \eta = \varepsilon^{1/4},\\
&\rho = \varepsilon^{15/16} \vert \ln(\varepsilon) \vert^2 \ \ \text{and}\ \ \alpha = \varepsilon^{13/16} \vert \ln(\varepsilon) \vert^2,
\end{align*}
providing the claimed rate of convergence, and concluding the proof of Theorem~\ref{SECT4TheorPrincipal_deLanford_}.
\end{proof}

\begin{remar}
The bound $T'$ of the time interval on which the convergence holds is smaller than the time $T$ of co-existence of the solutions of the two hierarchies (provided by Theorem \ref{SECT2TheorExistUniciSolutHiera}). This loss comes from the quantitative control of the difference between the terms with different pseudo-trajectories, especially from the Lipschitz control on $\sqrt{f_0}$, implying to consider only powers of the initial data (with a power $a$ smaller than $1$) on which the collision operator will act. But this power forces to relax the weight $\beta$ of the space $X_{0,1,\beta}$ in which $f^a_0$ lies, making the size of the time interval in which the collision operator is contracting smaller.\\
The limitation in the rate of convergence comes essentially from the geometrical estimates, that are clearly not optimal. This limitation is then of technical order, and specific to the proof presented in this work. Nevertheless, it would be surprising to recover the rate of convergence $O(\varepsilon)$ by optimal means. Indeed, concerning the derivation of the Boltzmann equation without obstacle, one can find in \cite{BGSR} an improvement of the explicit rate of convergence previously obtained in \cite{GSRT}, namely $\varepsilon\ln(\varepsilon)$, obtained by the integration of the singularity in time. Therefore, in the case of particles evolving outside of an obstacle, except if this obstacle has a regularizing effect (which is probably not the case), one shouldn't expect a better convergence than $\varepsilon \ln(\varepsilon)$.
\end{remar}

\appendix
\appendixpage

\section{The proof of the (almost everywhere) well-posedness of the hard sphere transport}
\label{AppenSectiBonneDefinSpherDures}

\begin{proof}[Proof of Proposition \ref{SECT1PropoBonneDefinSpherDures}]
Let $\delta \in\ ]0,1]$, $\rho \geq 1$ and $R \geq 1$ be three strictly positive real numbers. We start by removing configurations featuring large distances between the particles. We remove also configurations with positions far from $x=0$, or with high energies (that is, norms of the velocities): we set
$$
\mathcal{D}^\varepsilon_{N,\rho,R}=\mathcal{D}^\varepsilon_N\cap\left(\left(B_{\mathbb{R}^d}(0,\rho)\right)^N\times B_{\mathbb{R}^{dN}}(0,R)\right).
$$
For every integer $1\leq i\leq N$, and any pair $1\leq j\neq k\leq N$ with $j\neq i$ and $k\neq i$, we also consider
$$
\mathcal{A}^i_{j,k}(\delta)=\left\{Z_N\in\mathcal{D}^\varepsilon_{N,\rho,R}\ /\ |x_i-x_j|\leq 2\delta R+\varepsilon,\ |x_i-x_k|\leq 2\delta R+\varepsilon\right\}.
$$
If $Z_N$ does not belong to this subset of initial configurations, the particle $i$ will not be able to collide with the particles $j$ and $k$ on the time interval $[0,\delta]$. We have also:
\begin{align}
\label{BonneDefinMajorColliColli}
\big\vert \mathcal{A}^i_{j,k} \big\vert \leq C_1(d,N)\rho^{d(N-2)}R^{dN}\Big(\sum_{k=1}^d{d\choose k}\big( 2\delta R\big)^k\varepsilon^{d-k} \Big)^2.
\end{align}
Similarly, for all integers $1\leq i\neq j\leq N$, we consider
$$
\mathcal{B}^i_j(\delta)=\left\{Z_N\in\mathcal{D}^\varepsilon_{N,\rho,R}\ /\ |x_i-x_j|\leq 2\delta R+\varepsilon, d(\Omega,x_i)\leq\delta R+\varepsilon/2\right\}.
$$
If $Z_N$ does not belong to $\mathcal{B}^i_j(\delta)$, the particle $i$ will not collide with the particle $j$ nor bounce against the obstacle during the time interval $[0,\delta]$. We have in addition:
\begin{align}
\label{BonneDefinMajorColliRebon}
\big\vert \mathcal{B}^i_j \big\vert \leq C_2(d,N)\rho^{d(N-1)-1}R^{dN+1}\delta\sum_{k=1}^d{d\choose k}\left(2\delta R\right)^k\varepsilon^{d-k}.
\end{align}
Considering $\mathcal{D}^\varepsilon_{N,\rho,R}\ \backslash\ \mathcal{C}^\varepsilon_{N,\rho,R}(\delta,1)$ with
$$
\mathcal{C}^\varepsilon_{N,\rho,R}(\delta,1) = \bigg( \Big( \hspace{-8mm} \bigcup_{\substack{1\leq i\leq N\\ 1\leq j\neq k\leq N, \ j\neq i,\ k\neq i}} \hspace{-10mm} \mathcal{A}^i_{j,k}(\delta)\Big) \cup \Big( \hspace{-3mm} \bigcup_{\substack{1\leq i\leq N\\ 1\leq j\leq N,\ j\neq i}} \hspace{-4mm} \mathcal{B}^i_j(\delta)\Big) \bigg),
$$
we obtained a subset of $\mathcal{D}^\varepsilon_{N,\rho,R}$ composed of initial configurations which are all leading to a dynamics well defined on the whole time interval $[0,\delta]$, with:
\begin{align}
\label{TaillCondiMauvaDefinDelta}
\big\vert \mathcal{C}^\varepsilon_{N,\rho,R}(\delta,1) \big\vert \leq\ C_3(d,N)R^{dN+2}\left(\rho^{d(N-2)}\varepsilon^{2(d-1)}+\rho^{d(N-1)-1}\varepsilon^{d-1}\right)\delta^2
\end{align}
for $\delta$ small enough with $\varepsilon$, $R$ and $\rho$ fixed.
\newline
Now let $t$ be a given strictly positive constant, we choose $\delta > 0$ such that $t/\delta$ is a positive integer $m$. By the definition of the dynamics, and since $Z_N \in B_{\mathbb{R}^d}(0,\rho)$, at time $\delta$, the transport sends $Z_N$ in $\mathcal{D}^\varepsilon_{N,\rho+\delta R,R}$. Remembering that $\varepsilon\leq1$, the bound \eqref{TaillCondiMauvaDefinDelta} can be rewritten using:
\begin{equation}
\label{MajorSommePuissRho__Epsil}
\left(\rho^{d(N-2)}\varepsilon^{2(d-1)}+\rho^{d(N-1)-1}\varepsilon^{d-1}\right)\leq 2\rho^{d(N-1)-1}\varepsilon^{d-1}.
\end{equation}
Then, following the same steps as above and thanks to the bound \eqref{TaillCondiMauvaDefinDelta} rewritten with \eqref{MajorSommePuissRho__Epsil}, and up to excluding a small subset $\mathcal{C}^\varepsilon_{N,\rho+\delta R,R}(\delta,1)$ of $\mathcal{D}^\varepsilon_{N,\rho+\delta R,R}$ of size bounded by $C_4(d,N)R^{dN+2}\left(\rho+\delta R\right)^{d(N-1)-1}\varepsilon^{d-1}\delta^2$, the dynamics is well-defined until $2\delta$. Since the hard sphere flow preserves the measure (see for example \cite{CeIP}\footnote{In particular, see Appendix 4.A "More About Hard-Sphere Dynamics".}), it is possible to exclude from the set of the initial configuration a subset denoted $\mathcal{C}^\varepsilon_{N,\rho,R}(\delta,2)$, the union of $\mathcal{C}^\varepsilon_{N,\rho,R}(\delta,1)$ and $T^{N,\varepsilon}_{-\delta}\big(\mathcal{C}^\varepsilon_{N,\rho+\delta R,R}(\delta,1)\big) \cap \big(B_{\mathbb{R}^d}(0,R)^N\times B_{\mathbb{R}^{dN}}(0,R)\big)$, such that outside, the dynamics is well-defined until $2\delta$.\\
By induction, up to excluding a subset $\mathcal{C}^\varepsilon_{N,\rho,R}(\delta,m)$ which has a size smaller than
$$
C_4(d,N)R^{dN+2}\varepsilon^{d-1}\delta^2\sum_{k=0}^{m-1}\left((\rho+k\delta R)^{d(N-1)-1}\right),
$$
the dynamics is well defined on the whole time interval $[0,t]$. Taking $\rho=R$, the size of the excluded set is then bounded by (remembering that $t = m\delta$):
\begin{equation}
\label{TaillEnsemFinalBonneConfi}
C_4(d,N)R^{d(2N-1)+1}\varepsilon^{d-1}t(1+(m-1)\delta)^{d(N-1)-1}\delta\leq C(d,N,t,\varepsilon)R^{d(2N-1)+1}\delta.
\end{equation}
Considering the subset $\mathcal{F}^\varepsilon_{N,R}(t) = \bigcap_{m\ \in\ \mathbb{N}^*}\mathcal{C}^\varepsilon_{N,R,R}(t/m,m)$, of measure zero, any initial configuration taken in
$\mathcal{D}^\varepsilon_{N,R,R}=\mathcal{D}^\varepsilon_N\cap\left(\left(B_{\mathbb{R}^d}(0,R)\right)^N\times B_{\mathbb{R}^{dN}}(0,R)\right)$ and outside $\mathcal{F}^\varepsilon_{N,R}(t)$ produces a well-defined dynamics on $[0,t]$.\\
Considering now a sequence $\left(R_n\right)_{n\in\mathbb{N}}$ going to infinity with $R_0\geq1$, and $\mathcal{F}^\varepsilon_N(t)=\bigcup_{n\in\mathbb{N}}\mathcal{F}^\varepsilon_{N,R_n}(t)$, of measure zero, we see that it contains all the initial configurations leading to an ill-defined dynamics before the time $t$.\\
Finally, considering $\mathcal{F}^\varepsilon_N=\bigcup_{p\in\mathbb{N}^*}\mathcal{F}^\varepsilon_N(pt)$, any initial configuration $Z_N\in\mathcal{D}^\varepsilon_N$ taken outside this subset $\mathcal{F}^\varepsilon_N$,  leads to a well-defined dynamics, and the first point of Proposition \ref{SECT1PropoBonneDefinSpherDures} is proven.\\
For the second point, we will consider an initial configuration of $\mathcal{D}^\varepsilon_N$, which is not an element of $\mathcal{F}^\varepsilon_N$. We saw that this subset is of measure zero, and the dynamics from this initial configuration is well-defined on the whole time interval $\mathbb{R}_+$.\\
On the one hand, if one assumes that $\mathcal{E}(Z_N)$ admits an accumulation point, say $t_0>0$ (that is there exists a sequence of events $(t_k)_{k\in\mathbb{N}}$ such that $t_k \underset{k\rightarrow+\infty}{\longrightarrow}t_0$), since there is only a finite number of particles (here : $N$), there exists $1\leq N_0\leq N$ and a subsequence $(t_{\varphi(k)})_{k\in\mathbb{N}}$ of $(t_k)_{k\in\mathbb{N}}$, such that all the events of this subsequence correspond to the particle $N_0$: each event $t_{\varphi(k)}$ is either a collision or a bouncing involving $N_0$.\\
On the other hand, going back to the definition of the subset $\mathcal{F}^\varepsilon_N$, the fact that $Z_N\notin\mathcal{F}^\varepsilon_N$ exactly means that:
$$
\forall t>0,\ \forall\ p\in\mathbb{N}^*,\ \forall\ n\in\mathbb{N},\ Z_N\notin \mathcal{F}^\varepsilon_{N,R_n}(pt),
$$
where $\mathcal{F}^\varepsilon_{N,\rho_n,R_n}(pt)$ is defined above, that is:
$$
\forall t>0,\ \forall\ p\in\mathbb{N}^*,\ \forall\ n\in\mathbb{N},\ \exists\ m\in\mathbb{N}^*\ /\ Z_N\in\mathcal{C}^\varepsilon_{N,R_n}(pt/m,m).
$$
In particular, one chooses $t=2t_0$, where $t_0$ is the accumulation point of the set of events $\mathcal{E}(Z_N)$ mentionned above, so that $t_0\in\ ]0,t[$, and then there exists $k_0\in\mathbb{N}$ so that for all $k\geq k_0$, we have $t_{\varphi(k)}\in\ ]0,t[$. Since the sequence $\big(R_n\big)_{n\in\mathbb{N}}$ is increasing and tends to infinity, there exists $n_0\in\mathbb{N}$ such that $Z_N\in\big(B_{\mathbb{R}^d}(0,R_{n_0})\big)^N\times B_{\mathbb{R}^{dN}}(0,R_{n_0})$. Finally, for $t=2t_0$, $n=n_0$ and $p=1$, there exists $m_0\in\mathbb{N}^*$ such that $Z_N\in\mathcal{C}^\varepsilon_{N,R_{n_0}}(t/m_0,m_0)$. So all the events concerning the particle $N_0$ are separated by a time interval of length larger than $t/m_0$, which is obviously a contradiction. Proposition \ref{SECT1PropoBonneDefinSpherDures} is entirely proven.
\end{proof}

\section{The proof of the contracting property of the BBGKY and Boltzmann operators}
\label{AppenSectiOperaColliContractio}

\subsection{Stability of the spaces $\widetilde{\textbf{X}}_{N,\varepsilon,\widetilde{\beta},\widetilde{\mu}^\alpha}$ and $\widetilde{\textbf{X}}_{0,\widetilde{\beta},\widetilde{\mu}^\alpha}$ under the action of the operators $\mathfrak{E}_{N,\varepsilon}$ and $\mathfrak{E}_0$}

The first step to obtain Theorem \ref{SECT2TheorExistUniciSolutHiera} page \pageref{SECT2TheorExistUniciSolutHiera} is to show that the BBGKY and the Boltzmann operators $\mathfrak{E}_{N,\varepsilon}\big((f^{(s)}_{N,0})_{1 \leq s \leq N},\cdot\,\big)$ and $\mathfrak{E}_0\big((f^{(s)}_0)_{s \geq 1},\cdot \,\big)$ (for appropriate sequences of initial data $(f^{(s)}_{N,0})_{1 \leq s \leq N}$ and $(f^{(s)}_0)_{s \geq 1}$), send respectively the spaces $\widetilde{\textbf{X}}_{N,\varepsilon,\widetilde{\beta},\widetilde{\mu}^\alpha}$ and $\widetilde{\textbf{X}}_{0,\widetilde{\beta},\widetilde{\mu}^\alpha}$ into themselves.\\
Concerning the first terms $f^{(s)}_{N,0}$ and $\mathcal{T}^{s,0}_tf^{(s)}_0$ of the two operators, on the one hand it is immediate that if $(f^{(s)}_{N,0})_{1\leq s \leq N} \in \textbf{X}_{N,\varepsilon,\widetilde{\beta}(0),\widetilde{\mu}(0)^\alpha}$, then the constant function with respect to time $t \mapsto (f^{(s)}_{N,0})_{1\leq s \leq N}$ belongs to the space $\widetilde{\textbf{X}}_{N,\varepsilon,\widetilde{\beta},\widetilde{\mu}^\alpha}$, and this for any time interval $[0,T]$, as soon as the two weight functions $\widetilde{\beta}$ and $\widetilde{\mu}$ are non increasing (as an obvious consequence of the embedding result Proposition \ref{SECT2PropoIncluEspaXHierarchie}). On the other hand, concerning the Boltzmann hierarchy, the same result can be proven (see \cite{GSRT}\footnote{See Lemma 20 page 239}), that is if $(f^{(s)}_0)_{s\geq1}$ belongs to $\textbf{X}_{0,\widetilde{\beta}(0),\widetilde{\mu}(0)^\alpha}$, then $t \mapsto (\mathcal{T}^{s,0}_t f^{(s)}_0)_{s \geq 1}$ belongs to $\widetilde{\textbf{X}}_{0,\widetilde{\beta},\widetilde{\mu}^\alpha}$, up to assume in addition that the weight functions $\widetilde{\beta}$ and $\widetilde{\mu}$ are \emph{decreasing} (and not only non increasing).\\
This additional hypothesis is another element showing that the conjugated hierarchies (as \eqref{SECT2HieraBBGKYInteg/TmpsConju}) are more regular than the "usual" hierarchies \eqref{SECT1HieraBBGKYVersiInteg/Tmps} and \eqref{SECT1HieraBoltzVersiInteg/Tmps}.\\
\newline
Concerning the second terms defining the BBGKY and the Boltzmann operators, that is, the integrated in time collision operators, the functional spaces are designed exactly so that the elements of the spaces $\widetilde{\textbf{X}}_{N,\varepsilon,\widetilde{\beta},\widetilde{\mu}^\alpha}$ and $\widetilde{\textbf{X}}_{0,\widetilde{\beta},\widetilde{\mu}^\alpha}$, under the action of those integrated in time collision operators, provide functions with finite $\vertiii{\cdot}_{N,\varepsilon,\tilde{\beta},\tilde{\mu}^\alpha}$ and $\vertiii{\cdot}_{0,\tilde{\beta},\tilde{\mu}^\alpha}$ norms.

\subsection{The contracting property of the operators $\mathfrak{E}_{N,\varepsilon}$ and $\mathfrak{E}_0$}
\label{AppenSSectContractioOperaHiera}

Before Theorem \ref{SECT2TheorExistUniciSolutHiera}, the choice of the time interval on which we defined the spaces $\widetilde{\textbf{X}}_{N,\varepsilon,\widetilde{\beta},\widetilde{\mu}^\alpha}$ and $\widetilde{\textbf{X}}_{0,\widetilde{\beta},\widetilde{\mu}^\alpha}$ was arbitrary (and then could be chosen as long as one wants). Here, in order to use the fixed point theorem in the Banach spaces, we need to show that the operators $\mathfrak{E}_{N,\varepsilon}$ and $\mathfrak{E}_0$ are contracting mappings in the $\vertiii{\cdot}_{N,\varepsilon,\widetilde{\beta},\widetilde{\mu}^\alpha}$ and $\vertiii{\cdot}_{0,\widetilde{\beta},\widetilde{\mu}^\alpha}$ respectively. A time restriction will be necessary to deduce this contracting property.\\
\newline
The original idea and computation are due to Ukai (\cite{Ukai}) and Uchiyama (\cite{Uchi}), but a recent version is of course presented in \cite{GSRT}, and also in \cite{PhDTT}. We first point out the fact that the computation of the norm of those collision operators is the same for the BBGKY hierarchy and for the Boltzmann hierarchy. We present only the case $\alpha=1$, for the integrated in time transport-collision-transport operator of the BBGKY hierarchy.\\
For any $t \in [0,T]$ and for any $1\leq s\leq N-1$, and with $Q$ denoting
$$
Q = \Big\vert \int_0^t \mathcal{T}^{s,\varepsilon}_{-u} \mathcal{C}^{N,\varepsilon}_{s,s+1}\mathcal{T}^{s+1,\varepsilon}_u h^{(s+1)}_N(u,Z_s) \dd u \Big\vert,
$$
we have (thanks to Theorem \ref{SECT2TheorDefinOperaColliBBGKY}):
\begin{align*}
&Q\leq 2(N-s)\varepsilon^{d-1} \frac{\vert \mathbb{S}^{d-1} \vert}{2} \sum_{i=1}^s \int_0^t \Big\vert h^{(s+1)}_N(u,Z_{s+1})\exp\Big(\frac{\widetilde{\beta}(u)}{2}\sum_{j=1}^{s+1}\vert v_j \vert^2\Big) \Big\vert_{L^\infty(\mathcal{D}^\varepsilon_{s+1})}\\
&\ \ \ \ \ \ \ \ \ \ \ \ \ \ \ \times \int_{\mathbb{R}^d} \big( \vert v_i \vert + \vert v_{s+1} \vert \big) \exp\Big( -\frac{\widetilde{\beta}(u)}{2}\sum_{j=1}^{s+1} \vert v_j \vert^2 \Big) \dd v_{s+1} \dd u.
\end{align*}
Remembering that we want to obtain a bound on the $\vertiii{\cdot}_{N,\varepsilon,\widetilde{\beta},\widetilde{\mu}^1}$ norm of the function of sequences:
$$
t\mapsto \Big( \mathds{1}_{s\leq N-1} \int_0^t \mathcal{T}^{s,\varepsilon}_{-u}\mathcal{C}^{N,\varepsilon}_{s,s+1}\mathcal{T}^{s+1,\varepsilon}_u h^{(s+1)}_N(u,\cdot) \dd u \Big)_{1\leq s\leq N},
$$
we consider the product of $Q$ with $\exp\Big( \frac{\widetilde{\beta}(t)}{2}\sum_{i=1}^s\vert v_i \vert^2 \Big) \times \exp\big( s\widetilde{\mu}(t) \big)$ and bound this product uniformly in $t\in[0,T]$ and $1\leq s\leq N-1$, which gives in the end (see \cite{PhDTT} or \cite{GSRT} for more details):
\begin{align}
\label{SECT2InegaContractioOperaColli}
Q(t,s,Z_s) \leq C_3(d,N,\varepsilon) \widetilde{\beta}_\lambda(T)^{-d/2}&\exp\big(-\widetilde{\mu}_\lambda(T)\big)\frac{\big( 1 + \widetilde{\beta}_\lambda(T)^{-1/2} \big)}{\lambda}\nonumber\\
&\ \ \ \ \ \ \ \ \times\vertiii{ \big(h^{(s+1)}_N\big)_{1\leq s\leq N} }_{N,\varepsilon,\widetilde{\beta}_\lambda,\widetilde{\mu}_\lambda^1},
\end{align}
We obtained an explicit control of the $\vertiii{\cdot}_{N,\varepsilon,\widetilde{\beta}_\lambda,\widetilde{\mu}_\lambda^1}$ norm of the integral term of the BBGKY operator. It remains to choose wisely $\lambda$ and $T$ to obtain a contracting mapping.\\
\newline
To finish showing that the integrated in time transport-collision-transport operator sends the space $\widetilde{\textbf{X}}_{N,\varepsilon,\widetilde{\beta}_\lambda,\widetilde{\mu}_\lambda^1}$ into itself, it is also important to check the left continuity for all $t\in\ ]0,T]$ with respect to the $\vert\cdot\vert_{\varepsilon,s,\widetilde{\beta}_\lambda(t)}$ norm, for any $s$. This verification is presented in \cite{PhDTT} in details, and leads to the discussion page \pageref{SSSecDefinFonctSuiteUnifB} about balancing the strength of the continuity in time and the value of the parameter $\alpha$.\\
To perform this verification we consider, for each $1\leq s\leq N$ \emph{fixed}:\\
$Q \exp\Big(\frac{\widetilde{\beta}_\lambda(t)}{2}\vert V_s \vert^2\Big) \exp(s\widetilde{\mu}_\lambda(t))$ with $Q = \Big\vert \displaystyle{\int}_u^t \mathcal{T}^{s,\varepsilon}_{-\tau}\mathcal{C}^{N,\varepsilon}_{s,s+1}\mathcal{T}^{s+1,\varepsilon}_\tau h^{(s+1)}_N(\tau,Z_s)\dd \tau \Big\vert$.
Using the definition of the norms $\vert \cdot \vert_{\varepsilon,s,\widetilde{\beta}_\lambda(t)}$, $\vertii{\cdot}_{N,\varepsilon,\widetilde{\beta}_\lambda,\widetilde{\mu}_\lambda(t)^1}$ and $\vertiii{\cdot}_{N,\varepsilon,\widetilde{\beta}_\lambda,\widetilde{\mu}_\lambda^1}$, and the fact that the functions $\widetilde{\beta}_\lambda$ and $\widetilde{\mu}_\lambda$ are decreasing, we find, as above:
\begin{align*}
Q \exp\Big(\frac{\widetilde{\beta}_\lambda(t)}{2}\vert V_s \vert^2\Big) \exp(s\widetilde{\mu}_\lambda(t)) \leq&\ C(d,N,\varepsilon) \widetilde{\beta}_\lambda(T)^{-d/2} \vertiii{\big(h^{(s)}_N\big)_{1\leq s\leq N}}_{N,\varepsilon,\widetilde{\beta}_\lambda,\widetilde{\mu}_\lambda^1} \exp(-\widetilde{\mu}_\lambda(T)) \\
&\hspace{7mm} \times \Big(\sum_{i=1}^s\vert v_i \vert + s\widetilde{\beta}_\lambda(T)^{-1/2}\Big) \int_u^t \exp\Big(\frac{\lambda}{2}(\tau-t)\sum_{i=1}^s\vert v_i \vert^2 + \lambda s(\tau - t)\Big) \dd \tau.
\end{align*}
Using now the Cauchy-Schwarz inequality:
\begin{align*}
\int_u^t \exp\Big(\frac{\lambda}{2}(\tau-t)\sum_{i=1}^s\vert v_i \vert^2 + \lambda s(\tau - t)\Big) \dd \tau &\leq \sqrt{\int_u^t 1^2\dd \tau} \sqrt{\int_u^t \exp\Big(\frac{\lambda}{2}(\tau-t)\sum_{i=1}^s\vert v_i \vert^2 + \lambda s(\tau - t)\Big)^2 \dd \tau}\\
&\leq \sqrt{t-u} \frac{1}{\sqrt{\lambda\big(\sum_{i=1}^s\vert v_i \vert^2 + 2s\big)}},
\end{align*}
we obtain
\begin{align*}
\Bigg\vert \int_u^t \mathcal{T}^{s,\varepsilon}_{-\tau}\mathcal{C}^{N,\varepsilon}_{s,s+1}&\mathcal{T}^{s+1,\varepsilon}_\tau h^{(s+1)}_N(\tau,Z_s)\dd \tau \Bigg\vert \exp\Big(\frac{\widetilde{\beta}_\lambda(t)}{2}\vert V_s \vert^2\Big)\\
\leq& C(d,N,\varepsilon) \widetilde{\beta}_\lambda(T)^{-1/2} \exp(-(s+1)\widetilde{\mu}_\lambda(T))\frac{\big(1 + \widetilde{\beta}_\lambda(T)^{-1/2}\big)}{\sqrt{\lambda}}\sqrt{s}\sqrt{t-u} \vertiii{\big( h^{(s)}_N\big)_{1\leq s\leq N}}_{N,\varepsilon,\widetilde{\beta}_\lambda,\widetilde{\mu}_\lambda^1}
\end{align*}
that is, for every $1\leq s\leq N$:
$$
\big\vert h^{(s)}_N(t) - h^{(s)}_N(u) \big\vert_{\varepsilon,s,\widetilde{\beta}_\lambda(t)} \underset{u\rightarrow t^-}{\longrightarrow} 0,
$$
hence the continuity in time (in the sense \eqref{SECT2DEFINEspacSuiteTempsConte} of Definition \ref{SECT2DefinEspacSuiteMargiTemps} page \pageref{SECT2DefinEspacSuiteMargiTemps}) for the second term of the BBGKY operator.\\
\newline
Thanks to those controls, and following for example \cite{GSRT}\footnote{See Section 5.4, and in particular the very end of the proof of Lemma 5.4.3.}, we can find, for any $\beta_0>0$ and $\mu_0\in\mathbb{R}$, first a constant $\lambda$ (which is the decrease rate of the weights $\widetilde{\beta}$ and $\widetilde{\mu}$, that is, the speed of the loss of regularity), and then a constant $T$ (which is the length of the time interval on which $\widetilde{\beta}$ remains positive, so that the limitation on the size of the time interval on which we have existence of the solutions comes from this criterion) such that the $\vertiii{\cdot}_{\cdot,\widetilde{\beta}_\lambda,\widetilde{\mu}^1_\lambda}$ norm of the integrated in time collision operator is smaller than $1$, hence the result of Theorem \ref{SECT2TheorExistUniciSolutHiera} page \pageref{SECT2TheorExistUniciSolutHiera} thanks to the fixed point theorem in the Banach spaces.\\

\textbf{Acknowledgements.} The author expresses his warm gratitude to Isabelle Gallagher and Laurent Desvillettes, for their constant support and invaluable advice, dispensed throughout the years spent preparing his thesis, from which this article is taken. The author is also deeply grateful to Chiara Saffirio and to the University of Basel, for the ideal conditions and the serene atmosphere which made this work possible. Finally, the author acknowledges the support of the Swiss National Science Foundation through the Eccellenza project PCEFP2{\_}181153, and of the NCCR SwissMAP.

\bibliographystyle{plain}
\bibliography{bibli}

\end{document}